\documentclass[a4paper,10pt
]{amsart}
\usepackage{amsmath,amssymb,amsthm}

\usepackage[alphabetic, abbrev]{amsrefs}
\usepackage{enumerate
}
\usepackage{hyperref}
\usepackage{tikz-cd}
\usepackage{tikz}
\usepackage[T1]{fontenc}
\usepackage{enumitem}

\numberwithin{equation}{section}
\newtheorem{theorem}[subsection]{Theorem}
\newtheorem{corollary}[subsection]{Corollary}
\newtheorem{lemma}[subsection]{Lemma}
\newtheorem{proposition}[subsection]{Proposition}
\theoremstyle{definition}
\newtheorem{definition}[subsection]{Definition}
\newtheorem{remark}[subsection]{Remark}
\newtheorem{example}[subsection]{Example}

\newtheorem{construction}[subsection]{Construction}

\newcommand{\cal}{\mathcal}

\newcommand{\arxivlink}[1]{\href{http://arxiv.org/abs/#1}{\texttt{arXiv:#1}}}

\title[]{On Deformation Theory in \\ Higher Logarithmic Geometry}
\author{Tommy Lundemo}
\address{Department of Mathematics and Informatics, University of Wuppertal, Germany}
\email{lundemo@uni-wuppertal.de}
\date{\today}

\begin{document}
\maketitle

\begin{abstract} We initiate the study of deformation theory in the context of derived and higher log geometry. After reconceptualizing the ``exactification''-procedures in ordinary log geometry in terms of Quillen's approach to the cotangent complex, we construct an ``exactified tangent bundle'' over the category of log ring spectra. The fibers recover the categories of modules over the underlying ring spectra, and the resulting cotangent complex functor specializes to log topological Andr\'e--Quillen homology on each fiber. As applications, we characterize log square-zero extensions and derive a log variant of \'etale rigidity, applicable to some tamely ramified extensions of ring spectra. 
\end{abstract}

\tableofcontents

\section{Introduction}

In classical deformation theory, the theory of Kodaira--Spencer \cite{KS58} is an invaluable tool that reduces the problem of extending a first-order deformation to a linear problem governed by the tangent sheaf.  This theory has lent itself to very meaningful generalizations in derived algebraic geometry and homotopy theory with the works of Basterra \cite{Bas99}, To\"en--Vezzosi \cite{TV08}, and Lurie \cite[Chapter 7]{Lur17}, in increasing order of generality. 

Recall that a spectrum is \emph{connective} if its homotopy groups are concentrated in non-negative degrees. If $R$ is a connective ${\Bbb E}_{\infty}$-ring spectrum, the \emph{Postnikov tower} \[R \to \cdots \to \tau_{\le 2}(R) \to \tau_{\le 1}(R) \to \pi_0(R)\] is one of square-zero extensions, and exhibits the truncation map $R \to \pi_0(R)$ as an infinitesimal thickening.  Intuition from classical algebraic geometry \cite[Th\'eor\`eme 18.1.2]{egaiv} suggests that \'etale objects over $R$ should identify with those over $\pi_0(R)$. This is realized as the following special case of \cite[Theorem 7.5.0.6]{Lur17}:

\begin{theorem}[Lurie's \'etale rigidity]\label{thm:lurieetale} Let $R$ be a connective ${\Bbb E}_{\infty}$-ring. Base-change along the truncation map $R \to \pi_0(R)$ induces an equivalence of categories \[{\rm CAlg}_{R/}^{\rm \acute{e}t} \xrightarrow{\simeq} {\rm CAlg}_{\pi_0(R)/}^{\rm \acute{e}t}\] relating the categories of \'etale $R$-algebras and \'etale $\pi_0(R)$-algebras. 
\end{theorem}

In this paper, we shall be concerned with a generalization of Theorem \ref{thm:lurieetale}. The definition of \'etaleness pursued in \cite{Lur17} is borrowed from derived algebraic geometry: A map $R \to A$ is \emph{\'etale} if $\pi_0(R) \to \pi_0(A)$ is \'etale and the canonical map $\pi_*(R) \otimes_{\pi_0(R)} \pi_0(A) \xrightarrow{} \pi_*(A)$ is an isomorphism. This notion is necessarily correct in the context of connective ${\Bbb E}_{\infty}$-rings, as it is equivalent to the vanishing of the cotangent complex ${\Bbb L}_{A / R}$ as soon as $\pi_0(R) \to \pi_0(A)$ is of finite presentation. 

Statements like Theorem \ref{thm:lurieetale} are extremely appealing in that they reduce the delicate problem of constructing extensions of ${\Bbb E}_{\infty}$-ring spectra to ordinary algebra. For example, there is no extension ${\Bbb S} \to {\Bbb S}[i]$ realizing the ``Gaussian sphere spectrum'' \cite{SVW99}, but Theorem \ref{thm:lurieetale} provides an essentially unique extension of ${\Bbb E}_{\infty}$-rings ${\Bbb S}[1/2] \to {\Bbb S}[1/2, i]$ once the ramification at the prime $2$ is killed. 

On the other hand, the nature of \'etale morphisms imposes strict conditions on which extensions of ring spectra one can construct using Theorem \ref{thm:lurieetale}. While it provides an essentially unique extension $R \to R[\sqrt[n]{x}]$ as soon as $n$ and $x$ are units in the ordinary ring $\pi_0(R)$, examples that involve a homotopy class $x$ in strictly positive degree are not covered by Theorem \ref{thm:lurieetale}. 

\subsection{Logarithmic geometry} In an attempt to circumvent this problem, we shall make use of ideas from a variant of algebraic geometry - \emph{logarithmic geometry} - in which the notions of \'etaleness and smoothness are less rigid than they are in ordinary algebraic geometry \cite{Kat89}. For example, the ring  map ${\Bbb Z}_{(3)} \to {\Bbb Z}_{(3)}[\sqrt{3}]$ participates in a log \'etale morphism, despite not being \'etale. 

By definition\footnote{For the purposes of this introduction, we ignore the distinction between pre-log and log rings.}, a \emph{log ring} is a triple $(A, M, \alpha)$ with $A$ a commutative ring, $M$ a commutative monoid, and $\alpha \colon M \to (A, \cdot)$ a morphism of commutative monoids. The log \'etale extension above is realized as the map $({\Bbb Z}_{(3)}, \langle 3 \rangle) \to ({\Bbb Z}_{(3)}[\sqrt{3}], \langle \sqrt{3} \rangle)$, where the underlying monoids are freely generated by the given element. 

Rognes \cite{Rog09} has extended the notion of a log ring to the context of higher algebra. As we shall explain below, the resulting \emph{log ring spectra} give rise to more flexible notions of \'etaleness, which for instance includes examples like the inclusion of the connective Adams summand $\ell_p \to {\rm ku}_p$. On coefficient rings, this realizes the root adjunction ${\Bbb Z}_p[v_1] \to {\Bbb Z}_p[\sqrt[p - 1]{v_1}]$ of the class $v_1 \in \pi_{2p - 2}(\ell_p)$. 

Given a log ring spectrum $(R, P)$, we shall denote by ${\rm Log}_{(R, P)/}$ the category of $(R, P)$-algebras. These are log ring spectra $(A, M)$ together with a morphism $(R, P) \to (A, M)$ of such. We shall argue that the following is the correct generalization of Theorem \ref{thm:lurieetale} to the context of log geometry:

\begin{theorem}[Precise statement in Theorem \ref{thm:precisechartedlogetale}]\label{thm:chartedlogetale} Let $(R, P)$ be a connective log ring spectrum. The \emph{log Postnikov} truncation $(R, P) \to (\pi_0(R), P)$ induces a base-change functor \[(-, -) \otimes_{(R, P)} (\pi_0(R), P) \colon {\rm Log}^{}_{(R, P)/} \xrightarrow{} {\rm Log}^{}_{(\pi_0(R), P)/}\] which is an equivalence once restricted to \emph{charted log \'etale} objects. 
\end{theorem}

The statement of Theorem \ref{thm:chartedlogetale} may seem surprising at first sight, as the ``monoid variable'' $P$ has remained constant. We explain in Section \ref{subsec:intrologpostnikov} that this in fact what one would expect from the nature of square-zero extensions in classical log geometry. One of the main contributions of the present paper, then, is that the log-geometric notions indeed play well with the higher-algebraic approach to deformation theory developed by Lurie \cite[Chapter 7]{Lur17}. One immediate drawback is that the subcategory of ${\rm Log}_{(\pi_0(R), P)/}$ consisting of ``charted log \'etale'' objects is no longer purely algebraic in nature, as opposed to the situation of Theorem \ref{thm:lurieetale}. Nonetheless, Theorem \ref{thm:chartedlogetale} is applicable in concrete examples, such as the inclusion of the Adams summand $\ell_p \to {\rm ku}_p$. We explain this in Section \ref{subsec:adams}. The verbatim analog of Theorem \ref{thm:chartedlogetale} applies in the context of animated log rings, cf.\ Section \ref{subsec:animated}. 

The proof of Theorem \ref{thm:chartedlogetale} involves the development of the deformation theory of log ring spectra. The results that go into this are also applicable for animated log rings, as we point out in Section \ref{subsec:animated}. After discussing some examples, we shall spend the remainder of the introduction motivating and explaining the terminology used in Theorem \ref{thm:chartedlogetale}. This also includes a discussion of our main constructions in the context of ordinary log geometry. 

\subsection{Log ring spectra}\label{subsec:intrologring} As we review in Sections \ref{sec:commjspace}, \ref{sec:gradedeinfinity}, and \ref{sec:logringspectra}, the notion of a log ring has been generalized to the context of ${\Bbb E}_{\infty}$-ring spectra by Rognes \cite{Rog09}, with further developments by Sagave--Schlichtkrull \cites{SS12, SS19}, Sagave \cite{Sag14}, and Rognes--Sagave--Schlichtkrull \cites{RSS15, RSS18}. While there are variations, we shall use the definition of \emph{log ring spectra} pursued in the three latter papers. These are pairs $(A, M)$, where $A$ is an ${\Bbb E}_{\infty}$-ring and $M$ is a ``graded'' ${\Bbb E}_{\infty}$-space; see Sections \ref{sec:commjspace} and \ref{sec:gradedeinfinity} for a review. The category of graded ${\Bbb E}_{\infty}$-spaces participate in an adjunction $({\Bbb S}^{\cal J}[-], \Omega^{\cal J}(-))$ with the category of ${\Bbb E}_{\infty}$-rings, which we think of as a graded analog of the adjunction $({\Bbb S}[-], \Omega^{\infty}(-))$. The structure map of $(A, M)$ is a morphism $\alpha \colon M \to \Omega^{\cal J}(A)$ of graded ${\Bbb E}_{\infty}$-spaces.

\subsection{The example of the Adams summand}\label{subsec:adams} For odd primes $p$, the $p$-complete connective complex $K$-theory spectrum ${\rm ku}_p$ splits as $p - 1$ shifted copies of the \emph{Adams summand} $\ell_p$. Sagave \cite[Theorem 1.6]{Sag14} proved that the resulting inclusion participates in a map $(\ell_p, \langle v_1 \rangle_*) \to ({\rm ku}_p, \langle u \rangle_*)$ of log ring spectra which is formally log \'etale. Inspecting the proof, we readily find that this map is charted log \'etale in the sense of Definition \ref{def:chartedlogetale}.   

As we discuss and motivate in Section \ref{subsec:intrologpostnikov}, truncations of log ring spectra take a somewhat surprising form: The ``monoid'' variable is not allowed to change, and the bottom Postnikov stage of $(\ell_p, \langle v_1 \rangle_*)$ is $({\Bbb Z}_p, \langle v_1 \rangle_*)$. The structure map is given by the ``inverse image'' log structure $\langle v_1 \rangle_* \to \Omega^{\cal J}(\ell_p) \to \Omega^{\cal J}({\Bbb Z}_p)$ induced by ${\ell}_p \to {\Bbb Z}_p$. 

In this case, Theorem \ref{thm:chartedlogetale} states that the inclusion of the Adams summand, as a map of log ring spectra, is uniquely determined by the base-changed map $({\Bbb Z}_p, \langle v_1 \rangle_*) \to({\Bbb Z}_p \otimes_{\ell_p} {\rm ku}_p, \langle u \rangle_*)$. We once again stress that neither the source nor target of this map are discrete: The source involves the non-discrete graded ${\Bbb E}_{\infty}$-space $\langle v_1 \rangle_*$, while the target is non-discrete in both the ring and monoid variable. However, as explained in \cite[Proof of Proposition 4.15]{Sag14}, the target of this map can be constructed without reference to ${\rm ku}_p$. From the graded ${\Bbb E}_{\infty}$-space $\langle v_1 \rangle_*$ one can construct a graded ${\Bbb E}_{\infty}$-space $\langle \sqrt[p - 1]{v_1} \rangle_*$, and the base-changed map is equivalent to $({\Bbb Z}_p, \langle v_1 \rangle_*) \to ({\Bbb Z}_p \otimes_{{\Bbb S}^{\cal J}[\langle v_1 \rangle_*]} {\Bbb S}^{\cal J}[\langle \sqrt[p - 1]{v_1} \rangle_*], \langle \sqrt[p - 1]{v_1} \rangle)$. Once again, the underlying ${\Bbb E}_{\infty}$-ring of the target is not discrete, but is algebraic in nature in that its underlying spectrum is a sum $\oplus_{i = 0}^{p - 2} {\Bbb Z}_p[2i]$ of shifted copies of ${\Bbb Z}_p$. 

\subsection{Adjoining roots}\label{subsec:adjroots} As the above discussion suggests, the data required to obtain a root adjunction $R \to R[\sqrt[n]{x}]$ from Theorem \ref{thm:chartedlogetale} is a sufficiently well-behaved log structure $(R, \langle x \rangle_*)$ on $R$. This should allow for the construction of a graded ${\Bbb E}_{\infty}$-space $\langle \sqrt[n]{x} \rangle_*$ for which the base-changed morphism $(\pi_0(R), \langle x \rangle_*) \to (\pi_0(R) \otimes_{{\Bbb S}^{\cal J}[\langle x \rangle_*]} {\Bbb S}^{\cal J}[\langle \sqrt[n]{x} \rangle_*], \langle \sqrt[n]{x} \rangle)$ is charted log \'etale. 

These are rather restrictive assumptions, and implicit in our requirement that $(R, \langle x \rangle_*)$ is ``sufficiently well-behaved'' is a condition that roughly amounts to saying that the homotopy class $x \in \pi_d(R)$ is \emph{strict}. For $x \in \pi_0(R)$, this condition ensures that there is a factorization \[\begin{tikzcd}[row sep = tiny]{\Bbb S}[\sqcup_{k \ge 0} B\Sigma_k] \ar{dr} \ar{rr}{x} & & R \\ \vspace{10 mm} & {\Bbb S}[{\Bbb N}] \ar[dashed]{ur} \end{tikzcd}\] of the map from the free ${\Bbb E}_{\infty}$-ring ${\Bbb S}[\sqcup_{k \ge 0} B \Sigma_k]$ determined by $x$, where the map to the flat affine line ${\Bbb S}[{\Bbb N}] =: {\Bbb S}[t]$ is obtained by collapsing path components. For strict $x \in \pi_0(R)$, adjoining an $n$th root to $x$ is as simple as it is in ordinary algebra: Base-change the map ${\Bbb S}[t] \to R$ along the map ${\Bbb S}[t] \xrightarrow{t \mapsto t^n} {\Bbb S}[t]$.  

For homotopy classes $x \in \pi_d(R)$ with $d > 0$, we refer to Remark \ref{rem:strictness} for the notion of strictness that we have in mind. Under these assumptions, we can realize the relevant root adjunction as the base-change $R \otimes_{{\Bbb S}^{\cal J}[\langle x \rangle_*]} {\Bbb S}^{\cal J}[\langle \sqrt[n]{x} \rangle_*]$. It is thus not the case that Theorem \ref{thm:chartedlogetale} is a tool to \emph{construct} these root adjunctions; rather, it shows that they are uniquely determined by the map $\langle x \rangle_* \to \langle \sqrt[n]{x} \rangle_*$ of graded ${\Bbb E}_{\infty}$-spaces (which \emph{can} be used to construct the root adjunctions). For this uniqueness statement to be applicable, we must further require that $n$ be invertible in $\pi_0(R)$ (while $x \in \pi_d(R)$ need not be invertible), which corresponds to the log \'etaleness of \emph{tamely ramified} extensions. Another example for which Theorem \ref{thm:chartedlogetale} is applicable, then, is the root adjunction ${\rm ko}_{p} \to {\rm ko}_{p}[\sqrt{\alpha}]$ for an odd prime $p$ and the class $\alpha \in \pi_4({\rm ko}_p)$. We refer to Lawson \cite{Law20} and Ausoni--Bayındır--Moulinos \cite{ABM23} for further discussion of root adjunctions in different contexts, and to Hesselholt--Pstr\k{a}gowski \cite[Theorem 1.10]{HP23} for a variant of Theorem \ref{thm:lurieetale} that involves \'etale algebras over the coefficient ring $\pi_*(R)$.  

We end this discussion with a reality check. Devalapurkar \cite{Dev20} proved that there are no $K(n)$-local ${\Bbb E}_{\infty}$-rings $R$ for which $\pi_0(R)$ contains a primitive $p$th root of unity, unless $n = 0$ (see also \cite{Ang08}). As the extension ${\Bbb Z}_p \to {\Bbb Z}_p[\zeta_p]$ participates in a log \'etale map, one might fear that Theorem \ref{thm:chartedlogetale} provides e.g.\ an extension ${\rm ku}_p \to {\rm ku}_p[\zeta_p]$, which would contradict Devalapurkar's result upon inverting the Bott class $u$. Thankfully, the restrictive conditions under which we can construct well-behaved log structures on ${\Bbb E}_{\infty}$-rings do not allow for this: As explained by Rognes \cite{RogLoen}, the element $p$ is not strict in the ${\Bbb E}_{\infty}$-ring ${\rm ku}_p$,  and Theorem \ref{thm:chartedlogetale} does not apply to construct an extension contradicting Devalapurkar's result. 

\subsection{Exactification and repletion} We now proceed to discuss the material and terminology that go into Theorem \ref{thm:chartedlogetale}. We begin the discussion in the context of classical log geometry. 

The log \'etaleness of a morphism of log rings $(R, P) \to (A, M)$ is witnessed by the vanishing of a certain $A$-module $\Omega^1_{(A, M) / (R, P)}$ of \emph{log differentials}. This is a module over the underlying commutative ring $A$, as opposed to an object of a category dependent on the log structures. The construction of the log differentials passes through a procedure called \emph{exactification}; as we explain in Remark \ref{rem:replete}, we shall refer to this as \emph{repletion}.

The process of repletion is a very common theme in log geometry, and most invariants of log rings and log schemes, such as log de Rham cohomology, log crystalline cohomology, log (topological) Hochschild homology, and log prismatic cohomology, all pass through some variant of the repletion construction. 

\subsection{The replete abelianization functor} We consider the repletion procedure as a systematic way of ``cashing out'' the additional data provided by the log structure, a philosophy that we now aim to make precise. A very rough two-step recipe to construct an invariant of log rings is to 

\begin{enumerate}
\item mimic a construction in ordinary algebraic geometry to obtain a morphism $(B, N) \to (A, M)$ of log rings; and 
\item apply the repletion procedure to this morphism to obtain a new morphism $(B^{\rm rep}, N^{\rm rep}) \to (A, M)$ of log rings, with $B^{\rm rep} := B \otimes_{{\Bbb Z}[N]} {\Bbb Z}[N^{\rm rep}]$. 
\end{enumerate}

\noindent In cases of interest, the latter morphism will induce an isomorphism at the level of commutative monoids, so that the only relevant data is the underlying morphism of commutative rings. Concrete examples of this process include

\begin{enumerate}[label = (\alph*)]
\item the repletion $((A \otimes A)^{\rm rep}, (M \oplus M)^{\rm rep}) \to (A, M)$ of the diagonal map $(A \otimes A, M \oplus M) \to (A, M)$. The underlying map $(A \otimes A)^{\rm rep} \to A$ is called the \emph{log diagonal} by Kato--Saito \cite[Section 4]{KS04}. At the level of underlying monoids, the map $((A \otimes A)^{\rm rep}, (M \oplus M)^{\rm rep}) \to (A, M)$ is an isomorphism on associated log structures. At the level of underlying rings, its module of indecomposables (conormal) is the module of log differentials $\Omega^1_{(A, M)}$.
\item Following \cite{Rog09, SSV16}, we can also make sense of the repletion procedure in the context of simplicial commutative monoids. Applying this to the collapse map $S^1 \otimes (A, M) \to (A, M)$, we obtain the definition of log Hochschild homology ${\rm HH}(A, M)$ pursued by Rognes \cite{Rog09}.
\end{enumerate}

As the above examples suggest, the repletion procedure is an operation applied to log rings $(B, N) \to (A, M)$ over a fixed log ring $(A, M)$. We shall most often apply repletion in \emph{pointed} contexts; that is, when the map $(B, N) \to (A, M)$ admits a section. We thus consider repletion as a functor \[(-, -)^{\rm rep} \colon {\rm Log}_{(A, M)//(A, M)} \to {\rm Log}^{\rm rep}_{(A, M)//(A, M)}\] from the category of augmented $(A, M)$-algebras to augmented $(A, M)$-algebras that are \emph{replete}.  

The moral starting point to modern approaches to the cotangent complex is an observation of Quillen (that we recall as Theorem \ref{thm:quillen}), which identifies abelian group objects in augmented commutative rings ${\rm CRing}_{/A}$ with the category of modules over $A$, and the abelianization ${\rm Ab}_+(B)$ of an augmented commutative ring $B \to A$ with the $A$-module $\Omega^1_B \otimes_B A$. 

Using the repletion functor $(-, -)^{\rm rep}$, we explain in Section \ref{sec:prelude} how to construct a \emph{replete abelianization functor} \[{\rm Ab}_+^{\rm rep}(-, -)  \colon {\rm Log}_{/(A, M)} \to {\rm Ab}({\rm Log}^{\rm rep}_{(A, M)//(A, M)})\] from log rings augmented over $(A, M)$ to abelian group objects in augmented replete $(A, M)$-algebras. The following log variant of Quillen's result explains how the repletion procedure extracts data from the log structure to produce meaningful invariants in categories that no longer depend on the log structure.

\begin{proposition}\label{prop:repabel} Let $(A, M)$ be a (discrete) log ring. 
\begin{enumerate}
\item The forgetful functor ${\rm Log}^{\rm rep}_{(A, M)//(A, M)} \to {\rm CAlg}_{A//A}$ is an equivalence of categories. 
\item The assignment ${\rm Mod}_A \xrightarrow{} {\rm Ab}({\rm Log}^{\rm rep}_{(A, M)//(A, M)}), J \mapsto (A \oplus J, M \oplus J)$ defines an equivalence of categories. 
\item There is a natural isomorphism ${\rm Ab}_+^{\rm rep}(B, N) \cong \Omega^1_{(B, N)} \otimes_B A$ of $A$-modules, where $\Omega^1_{(B, N)}$ is the module of log differentials. 
\end{enumerate}
\end{proposition}

We observe that (1) realizes the goal of ``cashing out'' the log structure by means of repletion. The second part of Proposition \ref{prop:repabel} is also a consequence of \cite[Lemma 4.13]{Rog09}; see Remark \ref{rem:rognescomp} for discussion on this point. 

\subsection{A proof outline of Lurie's \'etale rigidity}\label{subsec:proofoutline} In order to motivate the results that we generalize to the context of log ring spectra, let us give a rough outline of the proof of Theorem \ref{thm:lurieetale}.  Given the deformation-theoretic nature of our approach, this is closer in spirit to the exposition of \cite{dagiv} than it is to \cite{Lur17}.

\begin{enumerate}
\item For any \emph{square-zero extension} $\widetilde{R} \to R$ of connective ${\Bbb E}_{\infty}$-rings, the base-change functor $- \otimes_{\widetilde{R}} R \colon {\rm CAlg}_{\widetilde{R}/}^{\rm \acute{e}t} \xrightarrow{} {\rm CAlg}^{\rm \acute{e}t}_{R/}$ is an equivalence.
\item Any connective ${\Bbb E}_{\infty}$-ring is the limit of its \emph{Postnikov tower} \[\cdots \to \tau_{\le 2}(R) \to \tau_{\le 1}(R) \to \tau_{\le 0}(R) \simeq \pi_0(R),\] and each map in the tower is a square-zero extension.
\item We may reduce to the truncated case: ${\rm CAlg}_{R/}^{\rm \acute{e}t} \simeq {\rm lim}({\rm CAlg}_{\tau_{\le n}(R)/}^{\rm \acute{e}t})$. 
\end{enumerate}

We stated Theorem \ref{thm:lurieetale} and the first step of the proof outline above in terms of the base-change functor. Once restricted to \'etale objects, the base-change functor along the Postnikov tower coincides with the relevant truncation functor, so that Theorem \ref{thm:lurieetale} could equally well have been stated for the functor $A \mapsto \pi_0(A)$. This observation is used in a crucial way in the third step of the above proof outline. 

\subsection{A log cotangent formalism} The constructions on ordinary log rings that we have discussed so far, such as repletion, have natural analogs in the context of log ring spectra. Moreover, there is an analog of the log differentials, called \emph{log topological Andr\'e--Quillen homology} ${\rm TAQ}(A, M)$, developed by Rognes \cite[Sections 11 and 13]{Rog09} and Sagave \cite{Sag14}. 

The analog of abelianization in the context of higher algebra is \emph{stabilization}. For a presentably symmetric monoidal, stable ${\infty}$-category ${\cal C}$ with a commutative algebra object $A$ and $A$-module $J$, the analog of Quillen's observation is the equivalence in the composite \begin{equation}\label{introcomposite}{\rm Mod}_A({\cal C}) \simeq {\rm Sp}({\rm CAlg}({\cal C})_{/A}) \xrightarrow{\Omega^{\infty}} {\rm CAlg}({\cal C})_{/A},\end{equation} and we may \emph{define} the square-zero extension $A \oplus J$ to be the image of $J$ under \eqref{introcomposite}. See \cite[Remark 7.3.4.15]{Lur17} for details. 

Lurie's cotangent complex formalism \cite[Section 7.3]{Lur17} globalizes this perspective:  Specializing to the category ${\cal C} = {\rm CAlg}$ of ${\Bbb E}_{\infty}$-ring spectra, there is a \emph{presentable fibration} $T_{\rm CAlg} \to {\rm CAlg}$ which, on each fiber, recovers the stabilization ${\rm Sp}({\rm CAlg}_{/A}) \simeq {\rm Mod}_A$.  We can thus informally think of objects of the category $T_{\rm CAlg}$ as pairs $(A, J)$, where $A$ is an ${\Bbb E}_{\infty}$-ring and $J$ is an $A$-module. There is a natural functor $T_{\rm CAlg} \to {\rm Fun}(\Delta^1, {\rm CAlg})$ which admits the informal description $(A, J) \mapsto (A \oplus J \to A)$. This admits a left adjoint \emph{relative} to ${\rm CAlg}$ in the sense of \cite[Section 7.3.2]{Lur17}, which is the essential ingredient in the construction of an \emph{absolute cotangent complex} functor ${\Bbb L} \colon {\rm CAlg} \to T_{\rm CAlg}$. 

\begin{theorem}\label{thm:repletetangent} There is a presentable fibration $T^{\rm rep}_{\rm Log} \to {\rm Log}$ such that:
\begin{enumerate}
\item The fiber $(T^{\rm rep}_{\rm Log})_{(A, M)}$ over any log ring spectrum $(A, M)$ is canonically equivalent to the category ${\rm Mod}_A$ of $A$-modules, and 
\item there is a map $T^{\rm rep}_{\rm Log} \to {\rm Fun}(\Delta^1, {\rm Log})$ of presentable fibrations over ${\rm Log}$ which admits a left adjoint relative to ${\rm Log}$.
\end{enumerate}
\end{theorem}

From Theorem \ref{thm:repletetangent}, we find that we can think of an object of the \emph{replete tangent bundle} $T^{\rm rep}_{\rm Log}$ as a pair $((A, M), J)$, where $(A, M)$ is a log ring spectrum and $J$ is an $A$-module. This gives a canonical candidate for split square-zero extensions $(A, M) \oplus J$, informally defined as the domain of the image of $((A, M), J)$ under the functor $T^{\rm rep}_{\rm Log} \to {\rm Fun}(\Delta^1, {\rm Log})$. As we explain in Remark \ref{rem:splitsquarezero}, this recovers the notion of split log square-zero extensions considered previously in the literature.

\subsection{The log cotangent complex} Analogously to the situation for ${\Bbb E}_{\infty}$-rings, the left adjoint provided by Theorem \ref{thm:repletetangent} is a key step in the construction of a \emph{log cotangent complex} ${\Bbb L}^{\rm rep} \colon {\rm Log} \to T^{\rm rep}_{\rm Log}$. In the same way that the $A$-modules ${\Bbb L}_A$ are modelled by classical topological Andr\'e--Quillen homology, we have:

\begin{theorem}\label{thm:logcotangent} The absolute log cotangent complex ${\Bbb L}^{\rm rep}$ recovers log topological Andr\'e--Quillen homology as the $A$-modules ${\Bbb L}^{\rm rep}_{(A, M)}$.
\end{theorem}

\subsection{The log Postnikov tower}\label{subsec:intrologpostnikov} We consider Theorem \ref{thm:logcotangent} and the results leading up to it as evidence that we have captured the correct variant of Lurie's cotangent complex formalism in the context of log ring spectra. We shall now argue that this also harmonizes with intuition from classical log geometry. 

The formalism leading to Theorem \ref{thm:logcotangent} effectively forces a definition of square-zero extensions of log ring spectra upon us: $(\widetilde{R}, \widetilde{P}) \to (R, P)$ is \emph{log square-zero} if there is a map $(d, d^\flat) \colon (R, P) \to (R, P) \oplus J[1]$ over $(R, P)$ which fits in a cartesian square \[\begin{tikzcd}[row sep = small](\widetilde{R}, \widetilde{P}) \ar{r} \ar{d} & (R, P) \ar{d}{(d, d^\flat)} \\ (R, P) \ar{r}{(d_0, d_0^\flat)} &  (R, P) \oplus J[1],\end{tikzcd}\] with $(d_0, d_0^\flat)$ a canonical inclusion (cf.\ \cite[Definition 7.4.1.6, Remark 7.4.1.7]{Lur17}).  

On the other hand, in classical log geometry \cite[(3.1), (3.2)]{Kat89}, one says that a map is square-zero if the underlying map of rings is square-zero and it is \emph{strict}; this effectively means that it induces an isomorphism on the level of underlying monoids.  The following reconciles these two notions:

\begin{theorem}\label{thm:squarezeroiff} Let $(p, p^\flat) \colon (\widetilde{R}, \widetilde{P}) \to (R, P)$ be a map of log rings where $\widetilde{R} \to R$ is a square-zero extension by a connective $R$-module $J$. Then $(p, p^\flat)$ is a log square-zero extension if and only if it is strict. 
\end{theorem}

As a consequence, for a log ring spectrum $(R, P)$ with $R$ connective, we obtain a tower \begin{equation}\label{intropostnikov}\cdots \to (\tau_{\le 2}(R), P) \to (\tau_{\le 1}(R), P) \to (\pi_0(R), P)\end{equation} of log square-zero extensions of $R$, that we shall refer to as the \emph{log Postnikov tower}. While the presence of the graded ${\Bbb E}_{\infty}$-space $P$ as a log structure on the discrete ring $\pi_0(R)$ might seem strange at first sight, this is perfectly healthy from a log geometric perspective. For example, if $A$ is a discrete valuation ring with residue field $k$, then the log ring $(k, \langle \pi \rangle)$ sending all positive powers of $\pi$ to $0$ is the \emph{standard log point}. We consider the log ring spectrum $(\pi_0(R), P)$ to be an analog of this.

\subsection{Log \'etale rigidity} We now have the tools to actively pursue the log variant of Theorem \ref{thm:lurieetale}. A map of log ring spectra is \emph{formally log \'etale} if the $A$-module ${\Bbb L}_{(A, M) / (R, P)}^{\rm rep}$ vanishes. See \cite[Proposition 3.12]{Kat89} for the log variant of the relationship between \'etaleness and the vanishing of differentials in the context of classical log geometry, and e.g.\ \cite[Example 4.9]{BLPO23} for a discussion of the fact that formal log \'etaleness is strictly stronger than log \'etaleness, even under finiteness hypotheses. 

With reference to the proof outline of Section \ref{subsec:proofoutline}, the first step goes through:

\begin{theorem}\label{thm:logetalebase} Let $(\widetilde{R}, \widetilde{P}) \to (R, P)$ be a log square-zero extension by a $0$-connected $R$-module. The base-change functor restricted to formally log \'etale objects \[(-, -) \otimes_{(\widetilde{R}, \widetilde{P})} (R, P) \colon {\rm Log}_{(\widetilde{R}, \widetilde{P})/}^{\rm fl\acute{e}t} \xrightarrow{} {\rm Log}_{(R, P)/}^{\rm fl\acute{e}t}\] is an equivalence. 
\end{theorem}

Moreover, the second step of the proof outline of Section \ref{subsec:proofoutline} is provided by the convergent tower \eqref{intropostnikov}. However, we see no reason for the third step to be true; the equivalences provided by base-change along the tower \eqref{intropostnikov} do not necessarily glue to an equivalence involving ${\rm Log}^{\rm fl\acute{e}t}_{(R, P)/}$, even under finiteness hypotheses. Chief among the numerous reasons that this could break down is the observation that the category of log ring spectra does not seem to arise as the category of algebras over some operad, and it does not seem feasible to argue as in \cite[Proof of Theorem 7.5.1.11]{Lur17} to reduce to the truncated case.

To solve this problem, we take one further cue from classical log geometry. The log ring spectra that we have used as the basis of the theory are \emph{not} an analog of classical log structures, but rather of \emph{charts} of log structures. Charts of log structures come with their own notion of \'etaleness, which \emph{a priori} is strictly stronger than being formally log \'etale, even when finiteness hypotheses are imposed. In Definition \ref{def:chartedlogetale}, we mimic the notion of \'etale charts in log geometry to obtain a category ${\rm Log}_{(R, P)/}^{\rm chl\acute{e}t}$ of \emph{charted log \'etale} $(R, P)$-algebras. This is the notion that plays the role of ``\'etale'' in the statement of Theorem \ref{thm:chartedlogetale}. 

\subsection{The deformation theory of animated log rings}\label{subsec:animated} While we work with one specific notion of log ring spectra throughout this paper, our arguments are fairly general and adapt easily to the context of animated log rings. In particular, Theorems \ref{thm:squarezeroiff}, \ref{thm:logetalebase}, and \ref{thm:chartedlogetale} all have natural analogs in this setting. 

\subsection{Future perspectives} In forthcoming work, Rognes--Sagave--Schlichtkrull use an $\infty$-categorical definition of log ring spectra that generalizes the notion used in \cites{RSS15, RSS18}. We strongly expect that the present results are valid in this model-independent setup, and that the results herein to provide a solid blueprint of a log variant of Lurie's more general \'etale rigidity statement \cite[Theorem 7.5.0.6]{Lur17}. In these contexts, there are occasionally other interesting towers that arise from square-zero extensions to consider; for example, there is a tower of ${\rm MU}$-modules \[BP\langle n \rangle \to \cdots \to B_3 \to B_2 \to B_1 \simeq BP \langle n - 1 \rangle\] relating the truncated Brown--Peterson spectra $BP \langle n \rangle$ and $BP \langle n - 1 \rangle$ \cite[Proposition 2.6.2, Proof of Theorem 2.0.6]{HW22}. The heuristics of \cite[Remark 9.8]{BLPO23} (based on \cite{SS19} and \cite{HW22}) suggest that $BP\langle n \rangle$ should admit a well-behaved log structure generated by $v_n$. The analogs of our results in that context, then, would uniquely determine log \'etale extensions of $BP \langle n \rangle$ from an extension of $BP \langle n - 1 \rangle$. This would provide additional structure to the root adjunctions $BP\langle n \rangle[\sqrt[p - 1]{v_n}]$ considered by Ausoni--Bayındır--Moulinos \cite{ABM23}. 

In \cite[Remark 9.8]{BLPO23} we motivate the necessity of a theory of \emph{spectral log geometry}. In analogy with the theory developed in \cite{SAG}, we expect that the results of this paper will serve as an important technical tool for the development of such a theory.  

\subsection{Outline} In the largely expository Section \ref{sec:prelude} we prove Proposition \ref{prop:repabel}. Sections \ref{sec:commjspace}, \ref{sec:gradedeinfinity}, and \ref{sec:logringspectra} provides the necessary background on graded ${\Bbb E}_{\infty}$-spaces and log ring spectra, and also contains some preliminary results. In Section \ref{sec:repletetangent} we prove Theorems \ref{thm:repletetangent} and \ref{thm:logcotangent}, while in Section \ref{sec:logpostnikov} we prove Theorem \ref{thm:squarezeroiff}. In the final Section \ref{sec:logetalerigidity}, we give proofs of Theorems \ref{thm:logetalebase} and \ref{thm:chartedlogetale}. 

\subsection{Acknowledgments} The author would like to thank Federico Binda, Jack Davies, Jens Hornbostel, Doosung Park, Piotr Pstr\k{a}gowski, Maxime Ramzi, Birgit Richter, Steffen Sagave, and Paul Arne {\O}stv{\ae}r for helpful discussions and comments related to this material. This research was conducted in the framework of the DFG-funded research training group GRK 2240: \emph{Algebro-Geometric Methods in Algebra, Arithmetic and Topology}. Some material already appears in the author's PhD thesis \cite{Lun22}, which was partially supported by the NWO-grant 613.009.121. Finally, the author would like to thank an anonymous referee for a thorough report that improved many aspects of the exposition and simplified several arguments.

\section{Prelude: The replete abelianization functor}\label{sec:prelude} In this section, we explain how our approach to the cotangent complex of log ring spectra manifests itself in the linear context of log rings. After recalling Quillen's approach to the cotangent complex and the basics on log rings, we describe how the repletion construction naturally gives rise to a functor \[{\rm Ab}_+^{\rm rep}(-, -) \colon {\rm Log}_{/(A, M)} \to {\rm Ab}({\rm Log}_{(A, M)//(A, M)}^{\rm rep})\] for a fixed log ring $(A, M)$. We proceed to identify the category ${\rm Log}_{(A, M)//(A, M)}^{\rm rep}$ with that of ordinary augmented $A$-algebras ${\rm CAlg}_{A//A}$, so that Quillen's classical result (recalled as Theorem \ref{thm:quillen}) identifies its category of abelian group objects with that of modules over the ring $A$. Hence the replete abelianization functor ${\rm Ab}_+^{\rm rep}(-, -)$ determines an $A$-module for each log ring over $(A, M)$, and we explain in Proposition \ref{prop:diffagree} that ${\rm Ab}_+^{\rm rep}(A, M) \cong \Omega^1_{(A, M)}$ - the classical module of log differentials.  The replete abelianization functor admits a non-abelian left derived functor, and the resulting animated module ${\Bbb L}{\rm Ab}_+^{\rm rep}(A, M)$ recovers Gabber's log cotangent complex \cite[Section 8]{Ols05}.

The inclusion of this preliminary section is mostly to motivate and contextualize our forthcoming constructions, and it is not strictly necessary for the later sections. This material is inspired by and partially overlaps with \cite[Sections 3 and 4]{Rog09}, and we refer to Remark \ref{rem:rognescomp} for further details on this point.

\subsection{Abelian group objects} Let ${\cal C}$ be a category with finite products and terminal object $*$. Recall that an \emph{abelian group object} $G$ of ${\cal C}$ comes with maps \[e \colon * \to G, \quad m \colon G \times G \to G, \quad {\rm inv} \colon G \to G,\] subject to the expected axioms. We write ${\rm Ab}({\cal C})$ for the category of abelian group objects in ${\cal C}$.

Let ${\rm CRing}$ denote the category of commutative rings. For a fixed commutative ring $A$, we denote by ${\rm CRing}_{/A}$ the category of commutative rings over $A$, that is, the category whose objects are ring maps $B \to A$. For an $A$-module $J$, we denote by $A \oplus J$ the split square-zero extension of $A$ by $J$. The following serves as a moral starting point for the approach to the cotangent complex typically taken in the context of derived and higher algebra (see e.g.\ \cite[Remark 7.3.2.17]{Lur17}):

\begin{theorem}[\cite{Qui}]\label{thm:quillen} Let $A$ be a commutative ring.

\begin{enumerate}
\item There is a well-defined functor \[{\rm Mod}_A \to {\rm Ab}({\rm CRing}_{/A}), \quad J \mapsto A \oplus J,\] which is an equivalence of categories.  
\item The forgetful functor ${\rm Ab}({\rm CRing}_{/A}) \to {\rm CRing}_{/A}$ admits a left adjoint \[{\rm Ab}_+ \colon {\rm CRing}_{/A} \to {\rm Ab}({\rm CRing}_{/A}).\] For $B \in {\rm CRing}_{/A}$, the $A$-module determined by ${\rm Ab}_+(B)$ under the above equivalence is isomorphic to $A \otimes_B \Omega^1_{B}$, with $\Omega^1_B$ the module of differentials. 
\end{enumerate}  
\end{theorem}

\begin{remark}[Abelianization as indecomposables]\label{rem:indec} Extension of scalars determines a functor $(-)_+ \colon {\rm CRing}_{/A} \to {\rm CRing}_{A//A}$ from the category of commutative rings augmented over $A$ to that of commutative rings \emph{pointed} at $A$: that is, augmented commutative $A$-algebras. As abelian group objects are already pointed, the category ${\rm Ab}({\rm CRing}_{A // A})$ is still equivalent to that of $A$-modules, and the resulting functor \[{\rm Ab}(-) \colon {\rm CRing}_{A//A} \to {\rm Ab}({\rm CRing}_{A//A}) \simeq {\rm Mod}_A\] can be identified with the indecomposables functor $C \mapsto I/I^2$, where $I$ is defined to be the kernel of the augmentation map $C \to A$. We observe that there is a canonical isomorphism ${\rm Ab}_+(B) \cong {\rm Ab}((B)_+) = {\rm Ab}(A \otimes_{\Bbb Z} B)$ of $A$-modules. 
\end{remark}

The resulting functor ${\rm Ab}_+ \colon {\rm CRing}_{/A} \to {\rm Mod}_A$ admits a non-abelian left derived functor, whose value at $A$ is the (absolute) cotangent complex ${\Bbb L}_A$ of $A$. 

\subsection{Log rings} We now aim to discuss Theorem \ref{thm:quillen} in the context of log geometry. Let us recall the basic definitions:

\begin{definition} A \emph{pre-log ring} $(A, M, \alpha)$ consists of a commutative ring $A$, a commutative monoid $M$, and a morphism of commutative monoids $\alpha \colon M \to (A, \cdot)$.
\end{definition}

By adjunction, the structure map $\alpha$ determines a unique map $\overline{\alpha} \colon {\Bbb Z}[M] \to A$ of commutative rings. We shall write ${\rm PreLog}$ for the resulting category of pre-log rings. Its coproduct is that of the underlying rings and monoids, that is, \[(A \otimes_{\Bbb Z} B, M \oplus N, \alpha \oplus \beta)\] is the coproduct of $(A, M, \alpha)$ and $(B, N, \beta)$, where $\alpha \oplus \beta$ is adjoint to the map ${\Bbb Z}[M \oplus N] \cong {\Bbb Z}[M] \otimes_{\Bbb Z} {\Bbb Z}[N] \xrightarrow{\overline{\alpha} \otimes \overline{\beta}} A \otimes_{\Bbb Z} B$ of commutative rings. 

We shall denote by ${\rm GL}_1(A)$ the group of multiplicative units in a commutative ring $A$, perhaps more commonly denoted $A^\times$. 

\begin{definition} A pre-log ring $(A, M, \alpha)$ is \emph{log} if the map $\alpha^{-1}{\rm GL}_1(A) \to {\rm GL}_1(A)$ is an isomorphism.
\end{definition}

We shall write ${\rm Log}$ for the resulting category of log rings. The forgetful functor ${\rm Log} \to {\rm PreLog}$ admits a left adjoint \[(-, -, -)^a \colon {\rm PreLog} \to {\rm Log}, \quad (A, M, \alpha) \mapsto (A, M^a, \alpha^a),\] where $M^a$ is defined as the pushout of the diagram $M \xleftarrow{} \alpha^{-1}{\rm GL}_1(A) \to {\rm GL}_1(A)$ and $\alpha^a$ is determined by its universal property along $\alpha$ and the inclusion of the units ${\rm GL}_1(A) \to (A, \cdot)$. 

\subsection{Exactification and repletion} The following notions are, either implicitly or explicitly, a key ingredient in the construction of most invariants of log rings and log schemes:

\begin{definition} Let $f^\flat \colon N \to M$ be a map of commutative monoids. We define
\begin{enumerate}
\item $f^\flat$ to be \emph{exact} if the diagram \[\begin{tikzcd}[row sep = small]N \ar{r} \ar{d}{f^{\flat}} & N^{\rm gp} \ar{d}{f^{\flat, {\rm gp}}} \\ M \ar{r} & M^{\rm gp} \end{tikzcd}\] is cartesian;
\item the \emph{exactification} $f^{\flat, {\rm ex}} \colon N^{\rm ex} \to M$ to be the base-change of $f^{\flat, {\rm gp}}$ along $M \to M^{\rm gp}$; and 
\item if $f^\flat$ participates in a map $(f, f^\flat) \colon (B, N) \to (A, M)$ of pre-log rings, its \emph{exactification} $(f^{\rm ex}, f^{\flat, {\rm ex}}) \colon (B^{\rm ex}, N^{\rm ex}) \to (A, M)$ is defined by setting $f^{\rm ex}$ to be the canonical map \[B^{\rm ex} := B \otimes_{{\Bbb Z}[N]} {\Bbb Z}[N^{\rm ex}] \to A \otimes_{{\Bbb Z}[M]} {\Bbb Z}[M] \cong A\] of commutative rings. 
\end{enumerate}
\end{definition}

For fixed $(A, M)$, we would like this construction to determine a functor \[(-, -)^{\rm ex} \colon {\rm PreLog}_{/(A, M)} \to {\rm PreLog}_{/(A, M)}^{\rm ex}, \quad (B, N) \mapsto (B^{\rm ex}, N^{\rm ex})\] from pre-log rings over $(A, M)$ to pre-log rings over $(A, M)$ with exact structure map to $M$.  For this, one should restrict attention to \emph{integral} monoids, that is, those monoids that inject into their group completion. See \cite[Proposition I.4.2.17]{Ogu18}.

In log geometry, exactness is thus typically discussed in the context of integral monoids. To circumvent the lack of a notion of integrality in derived contexts, Rognes identified conditions under which the exactification procedure lends itself to a homotopically meaningful generalization:

\begin{definition}\label{def:replete}(\cite[Definition 3.6]{Rog09}) Let $f^\flat \colon N \to M$ be a map of commutative monoids with the property that $f^{\flat, {\rm gp}} \colon N^{\rm gp} \to M^{\rm gp}$ is surjective. We define 
\begin{enumerate}
\item $f^\flat$ to be \emph{replete} if it is exact;
\item the \emph{repletion} $f^{\flat, {\rm rep}}$ to be its exactification; and 
\item if $f^\flat$ participates in a map $(f, f^\flat) \colon (B, N) \to (A, M)$ of pre-log rings, its \emph{repletion} $(f^{\rm rep}, f^{\flat, {\rm rep}}) \colon (B^{\rm rep}, N^{\rm rep}) \to (A, M)$ to be its exactification. 
\end{enumerate}
\end{definition}

The condition that $f^{\flat, {\rm gp}}$ be surjective is also used in Kato--Saito's \cite[Section 4]{KS04}. Writing ${\rm PreLog}^{\rm vsur}_{/(A, M)}$ for the full subcategory of ${\rm PreLog}_{/(A, M)}$ consisting of those $(f, f^\flat) \colon (B, N) \to (A, M)$ with $f^{\flat, {\rm gp}}$ surjective, \cite[Lemma 3.8]{Rog09} implies that the exactification construction gives a well-defined functor \[(-, -)^{\rm rep} \colon {\rm PreLog}_{/(A, M)}^{\rm vsur} \to {\rm PreLog}^{\rm rep}_{/(A, M)}\] to the category of pre-log rings over $(A, M)$ with replete structure map to $M$. We remark that, if $(B, N) \in {\rm PreLog}_{(A, M) // (A, M)}$ is an augmented $(A, M)$-algebra, then $N \to M$ is automatically virtually surjective. Thus we obtain a functor \[{\rm PreLog}_{/(A, M)} \xrightarrow{(-, -)_+} {\rm PreLog}_{(A, M) // (A, M)} \xrightarrow{(-, -)^{\rm rep}} {\rm PreLog}_{(A, M) // (A, M)}^{\rm rep},\] where $(B, N)_+ := (A \otimes_{{\Bbb Z}} B, M \oplus N)$. This will be used in Construction \ref{constr:replabelianization2}. 

\begin{example}\label{ex:splitsquarezero1} Let $(A, M, \alpha)$ be a pre-log ring and let $J$ be an $A$-module. The \emph{split square-zero extension} $(A, M) \oplus J := (A \oplus J, M \oplus J)$ has structure map \[M \oplus J \to (A \oplus J, \cdot), \quad (m, j) \mapsto (\alpha(m), \alpha(m)j)\] and is replete over $(A, M)$ via the projection $(A \oplus J, M \oplus J) \to (A, M)$. The choice of structure map will be explained in Example \ref{ex:splitsquarezero2}.
\end{example}

\begin{remark}\label{rem:replete} We shall only apply the exactification construction in pointed contexts, that is, when the map $f^\flat \colon N \to M$ admits a section. The condition that $f^{\flat, {\rm gp}}$ be surjective always holds in this situation, and the notions of exactification and repletion coincide. Nonetheless, we choose to use Rognes' terminology throughout. This is due to the fact that we will crucially use the notion of replete morphisms (even in non-pointed contexts) in the derived setting, and also an acknowledgment that we are largely inspired by the approach of \cite{Rog09}. 
\end{remark}

Let $(B, N)$ be an augmented $(A, M)$-algebra. This determines a fixed splitting \begin{equation}\label{unitssplitdiscrete}{\rm GL}_1(A) \oplus ({\rm GL}_1(B) / {\rm GL}_1(A)) \xrightarrow{\cong} {\rm GL}_1(B).\end{equation} We shall consider the pre-log structure \begin{equation}\label{prelogb}M \oplus ({\rm GL}_1(B) / {\rm GL}_1(A)) \xrightarrow{} (B, \cdot)\end{equation} induced by $M \to (A, \cdot) \to (B, \cdot)$ and ${\rm GL}_1(B) / {\rm GL}_1(A) \to {\rm GL}_1(B) \to (B, \cdot)$.  

Let us say that a morphism $(B, N) \to (A, M)$ is \emph{split replete} if it admits a section and if $N \to M$ is exact (and hence replete).

\begin{lemma}\label{lem:replogiremains} Let $(B, N) \in {\rm PreLog}^{\rm rep}_{(A, M)//(A, M)}$ be an augmented replete pre-log ring over a pre-log ring $(A, M, \alpha)$. Then $(B, N^a)$ is naturally isomorphic to $(B, M^a \oplus ({\rm GL}_1(B)/{\rm GL}_1(A)))$. In particular, $(B, N^a)$ is split replete over $(A, M^a)$. 
\end{lemma}

\begin{proof} By \cite[Lemma 3.11]{Rog09}, the fixed splitting of $N \to M$ induces a splitting $M \oplus (N^{\rm gp}/M^{\rm gp}) \xrightarrow{\cong} N$ over and under $M$, so that the structure map of $(B, N)$ factors as the lower horizontal composite in the commutative diagram \[\begin{tikzcd}[row sep = small]\alpha^{-1}{\rm GL}_1(A) \oplus (N^{\rm gp}/M^{{\rm gp}}) \ar{r} \ar{d} & {\rm GL}_1(A) \oplus {\rm GL}_1(B) \ar{r} \ar{d} & {\rm GL}_1(B) \ar{d} \\ M \oplus (N^{\rm gp}/M^{{\rm gp}}) \ar{r} & (A, \cdot) \oplus {\rm GL}_1(B) \ar{r} & (B, \cdot).\end{tikzcd}\] The left-hand square is cartesian by definition, while the right-hand square is cartesian since the structure map $A \to B$ admits a retraction so that an element of $A$ is a unit precisely when it maps to one in $B$. The defining pushout square for the logification $N^a$ is thus isomorphic to the outer rectangle  \[\begin{tikzcd}[row sep = small]\alpha^{-1}{\rm GL}_1(A) \oplus (N^{\rm gp}/M^{{\rm gp}}) \ar{r} \ar{d} & {\rm GL}_1(A) \oplus {\rm GL}_1(B) \ar{r} \ar{d} & {\rm GL}_1(B) \ar{d} \\ M \oplus (N^{\rm gp}/M^{{\rm gp}}) \ar{r} & M^a \oplus {\rm GL}_1(B) \ar{r} & N^a.\end{tikzcd}\] Since the left-hand  square is cocartesian, so is the right-hand square.  Hence both squares in the diagram \[\begin{tikzcd}[row sep = small]{\rm GL}_1(A)  \ar{r} \ar{d} & {\rm GL}_1(A) \oplus {\rm GL}_1(B) \ar{r} \ar{d} & {\rm GL}_1(B) \ar{d} \\ M^a \ar{r} & M^a \oplus {\rm GL}_1(B) \ar{r} & N^a\end{tikzcd}\] are cocartesian. The result follows from the splitting \eqref{unitssplitdiscrete}.
\end{proof}

\subsection{The replete abelianization functor} We now work towards the construction of the replete abelianization functor. Let us first record the following consequence of Lemma \ref{lem:replogiremains}:

\begin{corollary}\label{cor:replaugmented} Let $(A, M)$ be a log ring. Then the natural forgetful functor ${\rm Log}^{\rm rep}_{(A, M)//(A, M)} \to {\rm CRing}_{A//A}$ is an equivalence of categories, with quasi-inverse \[{\rm CRing}_{A//A} \to {\rm Log}^{\rm rep}_{(A, M) // (A, M)}, \quad B \mapsto (B, M \oplus ({\rm GL}_1(B) / {\rm GL}_1(A))),\] where the structure map on the target is described in \eqref{prelogb}.  
\end{corollary}

\begin{proof} As all objects in the statement are log rings (as opposed to pre-log rings), they are naturally isomorphic to their logifications. The result thus follows from Lemma \ref{lem:replogiremains}, which provides a natural isomorphism \[(B, M \oplus ({\rm GL}_1(B)/{\rm GL}_1(A))) \xrightarrow{\cong} (B, N)\] for any augmented replete $(A, M)$-algebra $(B, N) \in {\rm Log}^{\rm rep}_{(A, M)//(A, M)}$. 
 \end{proof}

We find it conceptually appealing that the replete category is independent of the log structure already before passing to abelian group objects. See Example \ref{ex:loghh} for the example of log (topological) Hochschild homology.  

\begin{example}\label{ex:splitsquarezero2} Let us identify the image of the split square-zero extension $A \oplus J$ under the equivalence of Corollary \ref{cor:replaugmented}. The units ${\rm GL}_1(A \oplus J)$ split as ${\rm GL}_1(A)$ and a copy $1 + J$ of the underlying abelian group $J$. This means that there is an isomorphism $M \oplus ({\rm GL}_1(A \oplus J)/{\rm GL}_1(A)) \cong M \oplus (1 + J)$, and the resulting structure map \eqref{prelogb} is given by \[M \oplus (1 + J) \to (A \oplus J, \cdot), \quad (m, (1, j)) \mapsto (\alpha(m), 0)(1, j) = (\alpha(m), \alpha(m)j).\] This recovers Example \ref{ex:splitsquarezero1}.  
\end{example}

As a consequence of Theorem \ref{thm:quillen}, Corollary \ref{cor:replaugmented}, and Example \ref{ex:splitsquarezero2}, we obtain the following:

\begin{corollary}\label{cor:ablogrep} Let $(A, M)$ be a log ring. There assignment \[{\rm Mod}_A \to {\rm Ab}({\rm Log}^{\rm rep}_{(A, M)//(A, M)}), \qquad J \mapsto (A \oplus J, M \oplus J) \] defines an equivalence of categories.  \qed
\end{corollary}

We now have all necessary ingredients for the following construction, which is the linear variant of one of the main constructions of this paper:

\begin{construction}\label{constr:replabelianization2} Let $(A, M)$ be a pre-log ring. Consider the following diagram \[\begin{tikzcd}[row sep = small]{\rm PreLog}_{/(A, M)} \ar[bend right = 5 mm, dashed, swap]{rdd}{{\rm Ab}_+^{\rm rep}(-, -)} \ar{r}{(-, -)_+} & {\rm PreLog}_{(A, M) // (A, M)} \ar{r}{(-, -)^{\rm rep}} & {\rm PreLog}_{(A, M) // (A, M)}^{\rm rep} \ar{d}{(-, -)^a} \\ \vspace{10 mm} & {\rm CRing}_{A//A} \ar{d}{{\rm Ab}(-)}&  {\rm Log}^{\rm rep}_{(A, M^a) // (A, M^a)} \ar[swap]{l}{\simeq} \ar{l}{\text{Cor } \ref{cor:replaugmented}}\\ \vspace{10 mm} & {\rm Mod}_A & {\rm Ab}({\rm Log}^{\rm rep}_{(A, M^a) // (A, M^a)}) \ar{u} \ar[swap]{l}{\simeq} \ar{l}{\text{Cor } \ref{cor:ablogrep}}\end{tikzcd}\] of categories and functors. Here $(-, -)_+$ is extension of scalars (along the map $({\Bbb Z}, \{1\}) \to (A, M)$ of pre-log rings), $(-, -)^{\rm rep}$ is the repletion functor, and $(-, -)^a$ is the logification functor. The functor ${\rm Log}^{\rm rep}_{(A, M^a)//(A, M^a)} \to {\rm CAlg}_{A//A}$ is the forgetful functor, which is an equivalence by Corollary \ref{cor:replaugmented}. The functor ${\rm Ab}(-)$ is the indecomposables functor of Remark \ref{rem:indec}.
\end{construction} 

\begin{definition} We define the \emph{replete abelianization} functor \[{\rm Ab}_+^{\rm rep} \colon {\rm PreLog}_{/(A, M)} \to {\rm Mod}_A\] to be the functor resulting from Construction \ref{constr:replabelianization2}. 
\end{definition}

By construction, ${\rm Ab}_+^{\rm rep}$ admits the explicit description \begin{equation}\label{repabexplicit}{\rm Ab}_+^{\rm rep}(B, N) \cong {\rm Ab}((A \otimes_{\Bbb Z} B)^{\rm rep}) = {\rm Ab}((A \otimes_{\Bbb Z} B) \otimes_{{\Bbb Z}[M \oplus N]} {\Bbb Z}[(M \oplus N)^{\rm rep}]),\end{equation} where the repletion is taken with respect to the map $M \oplus N \to M$ induced by the identity on $M$ and the structure map $N \to M$.

\begin{remark}\label{rem:rognescomp}As the objects $(B, N)$ of ${\rm Log}^{\rm rep}_{(A, M)//(A, M)}$ have \emph{strict} structure maps $(B, N) \to (A, M)$, the content of Corollary \ref{cor:ablogrep}, taken in isolation, is equivalent to that of \cite[Lemma 4.13]{Rog09}. Nonetheless, there is some expositional difference in the ways that we arrive at Corollary \ref{cor:ablogrep}. In \emph{loc.\ cit.}\ one starts with the category ${\rm Log}^{\rm str}_{/(A, M)}$ of strict augmented log rings, and, after deducing Corollary \ref{cor:ablogrep}, considers the possibility of a larger category whose abelian group objects would be a sensible candidate of ``log modules'' \cite[Remark 4.14]{Rog09}. Our approach is in some sense completely orthogonal to this: We \emph{begin} with the larger category of all augmented $(A, M)$-algebras, and actively use repletion as a means of ``cashing out'' the data provided by the log structure, aiming to land in a category independent of the log structure. Corollary \ref{cor:replaugmented} shows that we have succeeded in this, already before passing to abelian group objects. 
\end{remark}

\subsection{The log differentials as replete abelianization} In classical log geometry, one defines a \emph{log derivation} $(d, d^\flat) \colon (A, M) \to J$ to consist of a derivation $d \colon A \to J$ and a monoid map $d^\flat \colon M \to (J, +)$ satisfying $d(\alpha(m)) = \alpha(m)d^\flat(m)$. Inspired by Theorem \ref{thm:quillen} and the modern approach to the cotangent complex (of e.g\ \cite[Chapter 7]{Lur17}), we aim to make this haromize with the following definition:

\begin{definition} The module of \emph{log differentials} is the replete abelianization ${\rm Ab}_+^{\rm rep}(A, M)$ of $(A, M) \in {\rm PreLog}_{/(A, M)}$.  
\end{definition}

There is a well-established notion of differentials associated to a pre-log ring, explicitly defined by \[\Omega^1_{(A, M)} := \frac{\Omega^1_A \oplus (A \otimes_{\Bbb Z} M^{\rm gp})}{d\alpha(m) \sim \alpha(m) \otimes [m]},\] where $[m]$ denotes the image of $m$ under the canonical map $M \to M^{\rm gp}$. The module $\Omega^1_{(A, M)}$ corepresents log derivations.

\begin{proposition}\label{prop:diffagree} There is a canonical isomorphism $\Omega^1_{(A, M)} \cong {\rm Ab}_+^{\rm rep}(A, M)$.
\end{proposition}

\begin{proof} By \eqref{repabexplicit} we have that \[{\rm Ab}_+(A, M) \cong {\rm Ab}((A \otimes_{\Bbb Z} A) \otimes_{{\Bbb Z}[M \oplus M]} {\Bbb Z}[(M \oplus M)^{\rm rep}]).\] This coincides with a description of the log differentials $\Omega^1_{(A, M)}$ of Kato--Saito \cite[Section 4]{KS04}: In this linear setting, this is checked in \cite[Proposition 2.2.1.1]{Lun22}. \end{proof}

\begin{remark} The map $\Omega^1_{(A, M)} \to \Omega^1_{(A, M^a)}$ is an isomorphism by \cite[1.7]{Kat89}. 
\end{remark}

We now aim to prove Proposition \ref{prop:repabel}. Recall the split square-zero extension $(A \oplus J, M \oplus J)$ of Example \ref{ex:splitsquarezero2}. We first record the following consequence of Corollary \ref{cor:ablogrep}:

\begin{corollary}\label{cor:logder} Let $(A, M)$ be a pre-log ring. There is a natural isomorphism \[{\rm Hom}_{{\rm Mod}_A}({\rm Ab}_+^{\rm rep}(A, M), J) \cong {\rm Hom}_{{\rm Log}_{/(A, M^a)}}((A, M^a), (A \oplus J, M^a \oplus J)).\]
\end{corollary}

\begin{proof} By Corollary \ref{cor:ablogrep} Construction \ref{constr:replabelianization2}, the first-mentioned ${\rm Hom}$-set is naturally isomorphic to \[{\rm Hom}_{{\rm Log}_{(A, M^a)//(A, M^a)}}(((A \otimes_{\Bbb Z} A)^{\rm rep}), (M \oplus M)^{{\rm rep}, a}), (A \oplus J, M^a \oplus J)).\] Since logification and repletion are left adjoints and the map $(A \oplus J, M^a \oplus J) \to (A, M^a)$ is a replete map of log rings, this is naturally isomorphic to \[{\rm Hom}_{{\rm PreLog}_{(A, M)//(A, M^a)}}((A \otimes_{\Bbb Z} A, M \oplus M), (A \oplus J, M^a \oplus J)).\] By restriction of scalars along $({\Bbb Z}, \{1\}) \to (A, M)$ and once again exploiting that logification is a left adjoint, we obtain the description predicted by the corollary. 
\end{proof}

Corollary \ref{cor:logder} this implies that the set of log derivations ${\rm Der}((A, M), J)$ is naturally isomorphic to  \[{\rm Hom}_{{\rm Log}_{/(A, M^a)}}((A, M^a), (A \oplus J, M^a \oplus J))\] of augmented maps $(A, M^a) \to (A \oplus J, M^a \oplus J)$.  This perspective on log derivations is taken as the definition in e.g.\ \cite[Definition 4.15]{Rog09} and in the context of topological log structures (cf.\ \cite[Section 11]{Rog09} or \cite{Sag14}). 

Combined with Corollary \ref{cor:ablogrep}, the following gives a strong analog of Theorem \ref{thm:quillen} in the context of log geometry:

\begin{lemma}\label{lem:replabdiff} Under the equivalence of Corollary \ref{cor:ablogrep}, the replete abelianization functor admits the explicit description ${\rm Ab}_+^{\rm rep}(B, N) \cong A \otimes_B \Omega^1_{(B, N)}$. 
\end{lemma}

\begin{proof} By the Yoneda lemma, it suffices to prove that there is a natural isomorphism relating the two ${\rm Hom}$-sets on top in the diagram \[\begin{tikzpicture}[baseline= (a).base]
\node[scale=.93] (a) at (0,0){\begin{tikzcd}[column sep = tiny, row sep = small]{\rm Hom}_{{\rm Mod}_A}(A \otimes_B \Omega^1_{(B, N)}, J) \ar{r} \ar{d}{\cong} \ar[swap]{d}{\text{Restriction along } B \to A} & {\rm Hom}_{{\rm Mod}_A}({\rm Ab}_+^{\rm rep}(B, N), J) \\ {\rm Hom}_{{\rm Mod}_B}(\Omega^1_{(B, N)}, J) \ar{d}{\cong} \ar[swap]{d}{\text{Prop \ref{prop:diffagree} + Cor \ref{cor:logder}}} & {\rm Hom}_{{\rm CRing}_{A//A}}((A \otimes_{\Bbb Z} B)^{\rm rep}, A \oplus J) \ar{u}{\cong} \ar[swap]{u}{\text{Eq.} \eqref{repabexplicit}} \\ {\rm Hom}_{{\rm Log}_{/(B, N^a)}}((B, N^a), (B \oplus J, N^a \oplus J)) \ar{r}{\cong} & {\rm Hom}_{{\rm Log}_{/(A, M^a)}}((B, N^a), (A \oplus J, M^a \oplus J)),\ar{u}{\cong} \end{tikzcd}};\end{tikzpicture}\] and so it only remains to provide explanations for the two natural isomorphisms above that have yet to receive one. The bottom horizontal map, induced by the augmentation $(B, N^a) \to (A, M^a)$, is an isomorphism since the square \[\begin{tikzcd}[row sep = small](B \oplus J, N^a \oplus J) \ar{r} \ar{d} & (A \oplus J, M^a \oplus J) \ar{d} \\ (B, N^a) \ar{r} & (A, M^a)\end{tikzcd}\] is a pullback. Finally, we observe that all maps in the composite \[\begin{tikzcd}[row sep = small]{\rm Hom}_{{\rm Log}_{/(A, M^a)}}((B, N^a), (A \oplus J, M^a \oplus J)) \ar{d}{\cong} \\ {\rm Hom}_{{\rm PreLog}_{(A, M)//(A, M^a)}}((A \otimes_{\Bbb Z} B, M \oplus N), (A \oplus J, M^a \oplus J)) \ar{d}{\cong} \\ {\rm Hom}_{{\rm Log}_{(A, M^a)//(A, M^a)}^{\rm rep}}(((A \otimes_{\Bbb Z} B)^{\rm rep}, (M \oplus N)^{{\rm rep}, a}), (A \oplus J, M^a \oplus J)) \ar{d}{\cong} \\  {\rm Hom}_{{\rm CRing}_{A//A}}((A \otimes_{\Bbb Z} B)^{\rm rep}, A \oplus J) \end{tikzcd}\] are isomorphisms by logification being left adjoint to the inclusion ${\rm Log} \to {\rm PreLog}$ and cobase-change along the unit map $({\Bbb Z}, \{1\}) \to (A, M)$, repletion and logification being left adjoints, and the equivalence of Corollary \ref{cor:replaugmented}, respectively. This concludes the proof. 
\end{proof}

\begin{proof}[Proof of Proposition \ref{prop:repabel}] Part (1) is Corollary \ref{cor:replaugmented}. Part (2) follows from Theorem \ref{thm:quillen}, Corollary \ref{cor:replaugmented}, and Example \ref{ex:splitsquarezero2}. Finally, part (3) is Lemma \ref{lem:replabdiff}. 
\end{proof}

\section{Commutative ${\cal J}$-space monoids}\label{sec:commjspace} We now review the $QS^0$-graded ${\Bbb E}_{\infty}$-spaces of \cite{SS12}, modeled by \emph{commutative ${\cal J}$-space monoids}. These will play the role of commutative monoids in the definition of log ring spectra we pursue in Section \ref{sec:logringspectra}. While we will pass to underlying $\infty$-categories starting from Section \ref{sec:gradedeinfinity}, we work in the model that we review below throughout this section. This is due to our applications of the Bousfield--Friedlander theorem \cite[Theorem B.4]{BF78} (see e.g.\ the proof of Proposition \ref{prop:matherscube}), that we have not been able to phrase in a model-independent manner (but see Remark \ref{rem:bfindependent}).  For this reason, we have gathered all results dependent on the Bousfield--Friedlander theorem in this section, to be proved in the model described below. 

\subsection{Commutative ${\cal J}$-space monoids} We give a very brief recollection of the material on commutative ${\cal J}$-space monoids we shall use throughout. We refer to \cites{RSS18, RSS15, Sag14, SS12} for increasingly detailed expositions. 

Following \cite[Section 4]{SS12}, let ${\cal J}$ denote Quillen's localization construction $\Sigma^{-1}\Sigma$ on the category $\Sigma$ of finite sets and bijections. Objects of ${\cal J}$ are pairs $({\bf n}, {\bf m})$ where ${\bf n}$ denotes the finite set $\{1, \dots, n\}$. A ${\cal J}$-space is a functor from ${\cal J}$ to the category ${\cal S}$ of simplicial sets. This is a symmetric monoidal category $(S^{\cal J}, \boxtimes, U^{\cal J})$ and commutative monoids therein are \emph{commutative ${\cal J}$-space monoids}. The resulting category ${\cal C}{\cal S}^{\cal J}$ admits a \emph{positive ${\cal J}$-model structure} by \cite[Proposition 4.10]{SS12}. Weak equivalences $M \to N$ in this model structure are those maps that induce a weak equivalence $M_{h{\cal J}} \to N_{h{\cal J}}$ on (Bousfield--Kan) homotopy colimits over ${\cal J}$. We refer to its fibrations as \emph{positive fibrations}. 

By \cite[Theorem 1.7]{SS12}, the category of commutative ${\cal J}$-space monoids model ${\Bbb E}_{\infty}$-spaces over $QS^0 = \Omega^{\infty}({\Bbb S})$. We therefore think of commutative ${\cal J}$-space monoids as ($QS^0$-)graded ${\Bbb E}_{\infty}$-spaces. 

\subsection{Group completion} We say that a commutative ${\cal J}$-space monoid $M$ is \emph{grouplike} if the commutative monoid $\pi_0(M_{h{\cal J}})$ is a group. There is a \emph{group completion model structure} ${\cal C}{\cal S}^{\cal J}_{\rm gp}$ on the category of commutative ${\cal J}$-space monoids \cite[Theorem 5.5]{Sag16}. It arises as a left Bousfield localization of the positive ${\cal J}$-model structure and fibrant objects therein are precisely the (positive fibrant) grouplike commutative ${\cal J}$-space monoids. The \emph{group completion} $M \to M^{\rm gp}$ of a given commutative ${\cal J}$-space monoid $M$ is a fibrant replacement in this model structure. In the same way that grouplike ${\Bbb E}_{\infty}$-spaces model connective spectra, grouplike commutative ${\cal J}$-space monoids model connective spectra over the sphere \cite[Theorem 1.6]{Sag16}.

\subsection{Replete morphisms} The analog of Definition \ref{def:replete} in this context reads:

\begin{definition}[\cite{RSS15}] Let $N \to M$ be a map of commutative ${\cal J}$-space monoids. We say that it is
\begin{enumerate}
\item \emph{virtually surjective} if $\pi_0(N^{\rm gp}_{h{\cal J}}) \to \pi_0(M^{\rm gp}_{h{\cal J}})$ is a surjection of abelian groups;
\item \emph{exact} if the square \[\begin{tikzcd}[row sep = small]N \ar{r} \ar{d} & N^{\rm gp} \ar{d} \\ M \ar{r} & M^{\rm gp} \end{tikzcd}\] is homotopy cartesian in the positive ${\cal J}$-model structure; and
\item \emph{replete} if it is virtually surjective and exact. 
\end{enumerate}
\end{definition}

As explained in \cite[Lemma 3.17]{RSS15}, one convenient way to model the repletion $N^{\rm rep} \to M$ of a virtually surjective $N \to M$ is as a fibrant replacement of $N$ relative to $M$ in the group completion model structure. 

\subsection{Mather's cube lemma} Mather's second cube lemma \cite[Theorem 25]{Mat} states that, for a commutative cube of spaces in which the vertical faces are homotopy cartesian and the bottom face is homotopy cocartesian, the top face is homotopy cocartesian as well. For more general homotopy theories one has to assume that pushouts are \emph{universal}, i.e., that they commute with forming the homotopy pullback along any map, for this to hold.

We will need the following weakened form of Mather's second cube lemma in the category of commutative ${\cal J}$-space monoids:

\begin{proposition}\label{prop:matherscube} Let \[\begin{tikzcd}[row sep = tiny]
&
M_1
\ar{rr}{}
\ar[]{dd}[near end]{}
& & M_{12}
\ar{dd}{}
\\
M_{\emptyset}
\ar[crossing over]{rr}[near start]{}
\ar{dd}[swap]{}
\ar{ur}
& & M_2
\ar{ur}
\\
&
N_1
\ar[near start]{rr}{}
& & N_{12}
\\
N_\emptyset
\ar{ur}
\ar{rr}
& & N_2
\ar[crossing over, leftarrow, near start]{uu}{}
\ar{ur}
\end{tikzcd}\] be a commutative diagram of cofibrant commutative ${\cal J}$-space monoids in which the vertical faces are homotopy cartesian and the bottom face is homotopy cocartesian. If $M_{12}$ is replete over $N_{12}$, then the top face is homotopy cocartesian.
\end{proposition}

\begin{proof} There is no loss of generality in assuming that the vertical arrows are fibrations and that the vertical faces are pullbacks: Indeed, one can fix the bottom square and functorially replace $M_{12} \to N_{12}$ by a fibration, and define the vertical faces by pullback. Since the positive ${\cal J}$-model structure is right proper, the resulting cube is weakly equivalent to that under consideration.  

Let $B^\boxtimes(-, -, -)$ denote the two-sided bar construction in ${\cal J}$-spaces. We wish to prove that the canonical map $B^\boxtimes(M_1, M_\emptyset, M_2) \to M_{12}$ is a weak equivalence. Since $B^\boxtimes(N_1, N_\emptyset, N_2) \to N_{12}$ is a weak equivalence by assumption, we find that it suffices to prove that the left-hand square in the commutative diagram \[\begin{tikzcd}[row sep = small]B^{\boxtimes}(M_1, M_{\emptyset}, M_2) \ar{r} \ar{d} & B^\boxtimes(M_{12}, M_{12}, M_{12}) \ar{d} \ar{r} & B^{\boxtimes}(M_{12}^{\rm gp}, M_{12}^{\rm gp}, M_{12}^{\rm gp}) \ar{d} \\ B^{\boxtimes}(N_1, N_\emptyset, N_2) \ar{r}{\simeq} & B^\boxtimes(N_{12}, N_{12}, N_{12}) \ar{r} & B^{\boxtimes}(N_{12}^{\rm gp}, N_{12}^{\rm gp}, N_{12}^{\rm gp})\end{tikzcd}\] of commutative ${\cal J}$-space monoids is homotopy cartesian. Since $M_{12} \to N_{12}$ is replete and in particular exact, it suffices to prove that the outer rectangle is homotopy cartesian.

By \cite[Corollary 11.4]{SS12}, this occurs precisely when the associated square of Bousfield--Kan homotopy colimits is homotopy cartesian as a square of simplicial sets. Moreover, \cite[Lemma 2.11]{Sag14} allows us to rewrite the resulting square as \[\begin{tikzcd}[row sep = small]B^{\times}((M_1)_{h{\cal J}}, (M_\emptyset)_{h{\cal J}}, (M_2)_{h{\cal J}}) \ar{r} \ar{d} & B^{\times}((M^{\rm gp}_{12})_{h{\cal J}},(M^{\rm gp}_{12})_{h{\cal J}},(M^{\rm gp}_{12})_{h{\cal J}}) \ar{d} \\ B^{\times}((N_1)_{h{\cal J}},(N_\emptyset)_{h{\cal J}},(N_2)_{h{\cal J}}) \ar{r} & B^\times((N^{\rm gp}_{12})_{h{\cal J}}, (N^{\rm gp}_{12})_{h{\cal J}}, (N^{\rm gp}_{12})_{h{\cal J}}).\end{tikzcd}\] This square arises as the realization of a square of bisimplicial sets which is pointwise homotopy cartesian. As the ${\Bbb E}_\infty$-spaces involved on the right-hand side are grouplike, the resulting bisimplicial sets satisfy the $\pi_*$-Kan condition, and virtual surjectivity of the map $M_{12} \to N_{12}$ implies that that the right-hand vertical map is a Kan fibration on vertical path components. Hence the Bousfield--Friedlander theorem \cite[Theorem B.4]{BF78} applies to conclude the proof.
\end{proof}

\noindent The following was pointed out to us by Maxime Ramzi:

\begin{remark}\label{rem:bfindependent} All of our applications of the two-sided bar construction compute the derived tensor product of algebras in some symmetric monoidal category, as opposed to one of mere modules. This means that we are simply computing a derived pushout in the category of commutative algebra objects. The author has not been able to reconcile the construction of \cite[Construction 4.4.2.7]{Lur17} with the usual explicit formula for the two-sided bar construction (upon which the above argument crucially relies). For the pushouts that we are interested in, however, there are Bousfield--Kan formulas available (see e.g.\ \cite[Corollary 2.18]{Hau22}). We expect this to allow us to phrase the above argument in a model-independent manner, but we have not pursued the details of this. 
\end{remark}

\subsection{Commutative ${\cal J}$-space monoids and units} To a positive fibrant commutative ${\cal J}$-space monoid $M$, we may consider the subobject $M^\times$, where $M^\times({\bf n}, {\bf m})$ is defined to consist of those components of $M({\bf n}, {\bf m})$ that represent units in $\pi_0(M_{h{\cal J}})$. In \cite[Theorem 5.12]{Sag16}, a  \emph{units model structure} is constructed on the category of commutative ${\cal J}$-space monoids, which arises as a right Bousfield localization of the positive ${\cal J}$-model structure. It is Quillen equivalent to the group completion model structure (via the identity functor), and exhibits $M \mapsto M^\times$ as a colocalization.

 There is a Quillen adjunction \begin{equation}\label{quillenadjunction}{\Bbb S}^{\cal J}[-] \colon {\cal C}{\cal S}^{\cal J} \rightleftarrows {\cal C}{\rm Sp}^{\Sigma} \colon \Omega^{\cal J}(-)\end{equation} relating the positive ${\cal J}$-model structure and the positive stable model structure on commutative symmetric ring spectra of \cite{MMSS01}. For positive fibrant symmetric ring spectrum $R$, we may consider the \emph{graded units} ${\rm GL}_1^{\cal J}(R) := \Omega^{\cal J}(R)^\times$.  To any positive fibrant commutative ${\cal J}$-space monoid $M$ one can associate a \emph{graded signed monoid} $\pi_{0, *}(M)$. When applied to the inclusion ${\rm GL}_1^{\cal J}(R) \to \Omega^{\cal J}(R)$, this realizes the inclusion of the units ${\rm GL}_1(\pi_*(R))$ of the graded commutative ring $\pi_*(R)$ (cf.\ \cite[Proposition 4.26]{SS12}). 

\begin{remark} We will make use of the graded signed monoid $\pi_{0, *}(M)$ when $M = {\rm GL}_1^{\cal J}(R)$ or $\Omega^{\cal J}(R)$. For e.g.\ $M = \langle x \rangle_*$ assigned to a well-behaved homotopy class $x \in \pi_d(R)$ as in the introduction, $\pi_{0, *}(\langle x \rangle_*)$ contains information about all stable homotopy groups of the sphere (see Remark \ref{rem:strictness}). This suggests that $\pi_{0, *}(M)$ is too complicated to be a tractable algebraic invariant even in examples we would like to regard as simple, although its qualitative properties are of theoretical use. 
\end{remark}

We now describe a technical result that will be used in some of our arguments. Consider a homotopy cartesian square \begin{equation}\label{veritcalsurjunits}\begin{tikzcd}[row sep = small]\widetilde{A} \ar{r} \ar{d} & \widetilde{B} \ar{d} \\ A \ar{r} & B\end{tikzcd}\end{equation} of positive fibrant commutative symmetric ring spectra, and let $\widetilde{R} \to \widetilde{A}$ be a map with $\widetilde{R}$ positive fibrant. Let $P$ be a commutative ${\cal J}$-space monoid, and suppose that we are given a map ${\rm GL}_1^{\cal J}(\widetilde{R}) \to P$. Use the factorization properties of the positive ${\cal J}$-model structure to build a commutative diagram\[\begin{tikzcd}[row sep = small]U^{\cal J} \ar[tail]{r} & G_{\widetilde{R}} \ar[tail]{d} \ar[two heads]{r}{\simeq} & {\rm GL}_1^{\cal J}(\widetilde{R}) \ar{d} \\ \vspace{10 mm} & P^{\rm c} \ar[two heads]{r}{\simeq} & P\end{tikzcd}\] with tailed arrows cofibrations and two-headed arrows fibrations, and $U^{\cal J}$ the initial commutative ${\cal J}$-space monoid. 

\begin{lemma}\label{lem:bousfieldfriedlander} In the situation described above, assume further that the right-hand vertical morphism in \eqref{veritcalsurjunits} induces a surjection ${\rm GL}_1(\pi_*(\widetilde{B})) \to {\rm GL}_1(\pi_*(B)).$ Then the square \[\begin{tikzcd}[row sep = small]P^{\rm c} \boxtimes_{G_{\widetilde{R}}} {\rm GL}_1^{\cal J}(\widetilde{A}) \ar{r} \ar{d} & P^{\rm c, gp} \boxtimes_{G_{\widetilde{R}}} {\rm GL}_1^{\cal J}(\widetilde{B})  \ar{d} \ar{d} \\ P^{\rm c} \boxtimes_{G_{\widetilde{R}}} {\rm GL}_1^{\cal J}(A) \ar{r} & P^{\rm c, gp} \boxtimes_{G_{\widetilde{R}}} {\rm GL}_1^{\cal J}(B)\end{tikzcd} \] of commutative ${\cal J}$-space monoids is homotopy cartesian.
\end{lemma}

\begin{proof} The square in question is modelled by the square \[\begin{tikzcd}[row sep = small]B^\boxtimes(P^{\rm c}, G_{\widetilde{R}}, {\rm GL}_1^{\cal J}(\widetilde{A})) \ar{r} \ar{d} & B^\boxtimes(P^{\rm c, gp}, G_{\widetilde{R}}, {\rm GL}_1^{\cal J}(\widetilde{B})) \ar{d} \\ B^\boxtimes(P^{\rm c}, G_{{\widetilde{R}}}, {\rm GL}_1^{\cal J}(A)) \ar{r} & B^\boxtimes(P^{\rm c, gp}, G_{\widetilde{R}}, {\rm GL}_1^{\cal J}(B)) \end{tikzcd}\] of two-sided bar constructions in ${\cal J}$-spaces. We wish to argue with the Bousfield--Friedlander theorem as in the proof of Proposition \ref{prop:matherscube}. The cofibrancy hypotheses ensure that we can indeed reduce to checking whether the square \[\begin{tikzcd}[row sep = small]B^\times_{\bullet}(P_{h{\cal J}}^{\rm c}, (G_{\widetilde{R}})_{h{\cal J}}, ({\rm GL}_1^{\cal J}(\widetilde{A}))_{h{\cal J}}) \ar{r} \ar{d} & B^\times_{\bullet}(P^{\rm c, gp}_{h{\cal J}}, (G_{\widetilde{R}})_{h{\cal J}}, ({\rm GL}_1^{\cal J}(\widetilde{B}))_{h{\cal J}}) \ar{d} \\ B^\times_{\bullet}(P_{h{\cal J}}^{\rm c}, (G_{\widetilde{R}})_{h{\cal J}}, ({\rm GL}_1^{\cal J}(A))_{h{\cal J}}) \ar{r} & B^\times_{\bullet}(P^{\rm c, gp}_{h{\cal J}}, (G_{\widetilde{R}})_{h{\cal J}}, ({\rm GL}_1^{\cal J}(B))_{h{\cal J}}) \end{tikzcd}\]  of bisimplicial sets is cartesian after realization. The square of bisimplicial sets is pointwise homotopy cartesian and the objects on the right-hand side satisfy the $\pi_*$-Kan condition as all commutative ${\cal J}$-space monoids involved are grouplike.  Combining \cite[Corollary 4.17, Proposition 4.26]{SS12}, we find that the condition on the units ensures that the morphism ${\rm GL}_1^{\cal J}(\widetilde{B}) \to {\rm GL}_1^{\cal J}(B)$ is virtually surjective so that the right-hand vertical morphism in the square of bisimplicial sets induces a Kan fibration on vertical path components. Hence the Bousfield--Friedlander theorem \cite[Theorem B.4]{BF78} applies to conclude the proof. 
\end{proof}

\subsection{Replete juggling} Let us also record the following consequence of the proof of \cite[Lemma 3.12]{Lun21}, which relies on the Bousfield--Friedlander theorem:

\begin{lemma}\label{lem:repbasechange} Let $P \to M$ be a cofibration of cofibrant commutative ${\cal J}$-space monoids. The canonical map \[P \boxtimes_{(P \boxtimes P)^{\rm rep}} (M \boxtimes M)^{\rm rep} \xrightarrow{} (M \boxtimes_P M)^{\rm rep}\] is a ${\cal J}$-equivalence, where all repletions are formed with respect to the natural multiplication maps. 
\end{lemma}

\begin{proof} We can argue exactly as in \cite[Lemma 3.12]{Lun21}. Since there is a natural isomorphism $M \boxtimes_P M \cong P \boxtimes_{(P \boxtimes P)} (M \boxtimes M)$ and group completions commute with homotopy pushouts \cite[Lemma 2.9]{Lun21}, it suffices to prove that the square of two-sided bar constructions \[\begin{tikzcd}[row sep = small]B^\boxtimes(P, (P \boxtimes P)^{\rm rep}, (M \boxtimes M)^{\rm rep}) \ar{r} \ar{d} & B^\boxtimes(P^{\rm gp}, (P \boxtimes P)^{\rm gp}, (M \boxtimes M)^{\rm gp}) \ar{d} \\ B^\boxtimes(P, P, M) \ar{r} & B^\boxtimes(P^{\rm gp}, P^{\rm gp}, M^{\rm gp})\end{tikzcd}\] is homotopy cartesian. Using the homotopy cartesian squares for the repletions $(P \boxtimes P)^{\rm rep}$ and $(M \boxtimes M)^{\rm rep}$, we see that we can argue with the Bousfield--Friedlander theorem as in the proof of Proposition \ref{prop:matherscube}. 
\end{proof}

\section{Graded ${\Bbb E}_{\infty}$-spaces}\label{sec:gradedeinfinity} Having established the necessary results on commutative ${\cal J}$-space monoids that involve the Bousfield--Friedlander theorem, we now pass to its underlying $\infty$-category. 

\subsection{The homotopy theory of graded ${\Bbb E}_{\infty}$-spaces}\label{subsec:htpythygr} Let ${\cal C}{\cal S}^{\cal J}_{\infty}$ denote the $\infty$-category underlying the positive ${\cal J}$-model structure on commutative ${\cal J}$-space monoids. By \cite[Example 4.1.7.6]{Lur17}, the fact that $({\cal C}{\cal S}^{\cal J}, \boxtimes, U^{\cal J})$ is a symmetric monoidal model category implies that ${\cal C}{\cal S}^{\cal J}_{\infty}$ is a symmetric monoidal $\infty$-category. We shall refer to its objects as ($QS^0$-)graded ${\Bbb E}_{\infty}$-spaces. Let us record some properties of the category of graded ${\Bbb E}_{\infty}$-spaces here: 

\begin{enumerate}
\item The Quillen adjunction \eqref{quillenadjunction} induces an adjunction \[{\Bbb S}^{\cal J}[-] \colon {\cal C}{\cal S}^{\cal J}_{\infty} \rightleftarrows{} {\cal C}{\rm Sp}^\Sigma_{\infty} \simeq {\rm CAlg}({\rm Sp}) \colon \Omega^{\cal J}(-)\] relating the categories of graded ${\Bbb E}_{\infty}$-spaces and ${\Bbb E}_{\infty}$-ring spectra. 
\item Similarly, the Quillen adjunction \[{\rm colim}_{\cal J} \colon {\cal S}^{\cal J} \rightleftarrows {\cal S} \colon {\rm const}_{\cal J}\] of \cite[Proposition 6.23]{SS12} gives rise to an adjunction of underlying $\infty$-categories. We shall denote the resulting left adjoint by $(-)_{h{\cal J}}$. 
\item As the group completion model structure ${\cal C}{\cal S}^{\cal J}_{\rm gp}$ is a left Bousfield localization of the positive model structure on ${\cal C}{\cal S}^{\cal J}$, the resulting category of grouplike graded ${\Bbb E}_{\infty}$-spaces ${\cal C}{\cal S}^{\cal J}_{\infty, {\rm gp}}$ is a localization of ${\cal C}{\cal S}^{\cal J}_{\infty}$. We model group completions by the corresponding localization functor $(-)^{\rm gp} \colon {\cal C}{\cal S}^{\cal J}_{\infty} \to {\cal C}{\cal S}^{\cal J}_{\infty, {\rm gp}}$. 
\item Associating to a positive fibrant commutative ${\cal J}$-space monoid $M$ its units $M^\times$ is a right adjoint to the inclusion of grouplike objects. In fact, the \emph{units model structure} ${\cal C}{\cal S}^{\cal J}_{\rm un}$ in commutative ${\cal J}$-space monoids of \cite[Theorem 5.12]{Sag16} is a right Bousfield localization, and is Quillen equivalent to the group completion model structure \cite[Corollary 5.13]{Sag16}. This induces a right adjoint $(-)^\times \colon {\cal C}{\cal S}^{\cal J}_{\infty} \to {\cal C}{\cal S}^{\cal J}_{\infty, {\rm un}} \simeq {\cal C}{\cal S}^{\cal J}_{\infty, {\rm gp}}$ to the inclusion ${\cal C}{\cal S}^{\cal J}_{\infty, {\rm gp}} \to {\cal C}{\cal S}^{\cal J}_{\infty}$.  

\end{enumerate}

All definitions stated and results proved in Section \ref{sec:commjspace} are homotopically meaningful and thus have natural analogs in this setting that we will use throughout.  

\subsection{Pulling back units}  Let $R$ be an ${\Bbb E}_{\infty}$-ring and let $J$ be an $R$-module. Recall that a \emph{derivation} of $R$ with values in $J$ is an augmented morphism of ${\Bbb E}_{\infty}$-rings $d \colon R \to R \oplus J$ (with ring structure on the target as in \cite[Remark 7.3.4.15]{Lur17}), so that the space of derivations is ${\rm Map}_{{\rm CAlg}_{/R}}(R, R \oplus J)$. We say that $\widetilde{R} \to R$ is a \emph{square-zero extension} by the $R$-module $J$ if there is a cartesian diagram of ${\Bbb E}_{\infty}$-rings \[\begin{tikzcd}[row sep = small]\widetilde{R} \ar{r} \ar{d} & R \ar{d}{d} \\ R \ar{r}{d_0} & R \oplus J[1],\end{tikzcd}\] where $d_0$ denotes the trivial derivation. 

\begin{lemma}\label{lem:unitspullback} Let $\widetilde{R} \to R$ be a square-zero extension of ${\Bbb E}_{\infty}$-rings. Then the square \[\begin{tikzcd}[row sep = small]{\rm GL}_1^{\cal J}(\widetilde{R}) \ar{r} \ar{d} & \ar{d} \Omega^{\cal J}(\widetilde{R}) \\  {\rm GL}_1^{\cal J}(R) \ar{r} & \Omega^{\cal J}(R)\end{tikzcd}\] is cartesian. 
\end{lemma}

\begin{proof} Consider the commutative cube 

\[\begin{tikzcd}[row sep = tiny]
&
{\rm GL}_1^{\cal J}(R)
\ar{rr}{}
\ar[]{dd}[near end]{}
& & \Omega^{\cal J}(R)
\ar{dd}{}
\\
{\rm GL}_1^{\cal J}(\widetilde{R})
\ar[crossing over]{rr}[near start]{}
\ar{dd}[swap]{}
\ar{ur}
& &\Omega^{\cal J}(\widetilde{R})
\ar{ur}
\\
&
{\rm GL}_1^{\cal J}(R \oplus J[1])
\ar[near start]{rr}{}
& & \Omega^{\cal J}(R \oplus J[1])
\\
{\rm GL}_1^{\cal J}(R)
\ar{ur}
\ar{rr}
& & \Omega^{\cal J}(R)
\ar[crossing over, leftarrow, near start]{uu}{}
\ar[swap]{ur}{}
\end{tikzcd}\] 
of graded ${\Bbb E}_{\infty}$-spaces. As the left- and right-hand faces are cartesian, it suffices to show that the back face is cartesian. By the pasting lemma, this follows from the fact that the square \[\begin{tikzcd}[row sep = small]{\rm GL}_1^{\cal J}(R \oplus J[1]) \ar{r} \ar{d} & \Omega^{\cal J}(R \oplus J[1]) \ar{d} \\ {\rm GL}_1^{\cal J}(R) \ar{r} & \Omega^{\cal J}(R)\end{tikzcd}\] is cartesian (cf.\ \cite[Proof of Lemma 11.27]{Rog09}). 
\end{proof}

\subsection{Replete morphisms and base-change} One convenient property of split replete morphisms is their behavior under base-change. The analogous statement in the context of integral monoids (but without the split condition) is \cite[Proposition I.4.2.1(6(b))]{Ogu18}. 

\begin{lemma}\label{lem:repletebasechangestable} Let $p_M \colon \widetilde{M} \to M$ be a replete morphism of commutative ${\cal J}$-space monoids which admits a section. Let $f \colon N \to M$ be a map. Then the base-change $p_N \colon \widetilde{N} \to N$ of $p_M$ along $f$ is also split replete. 
\end{lemma}

\begin{proof} By \cite[Lemma 2.12]{Lun21}, there is an equivalence $M \boxtimes W(\widetilde{M}) \xrightarrow{\simeq} \widetilde{M}$ over and under $M$, where $W(\widetilde{M})$ is the grouplike graded ${\Bbb E}_{\infty}$-space defined as the pullback of the diagram $U^{\cal J} \xrightarrow{} M^{\rm gp} \xleftarrow{} \widetilde{M}^{\rm gp}$. Pulling back along $f$, we obtain a cube \[\begin{tikzcd}[row sep = tiny]
&
\widetilde{N}
\ar{rr}{}
\ar[]{dd}[near end]{}
& & \widetilde{M}
\ar{dd}{}
\\
N \boxtimes W(\widetilde{M})
\ar[crossing over]{rr}[near start]{}
\ar{dd}[swap]{}
\ar{ur}{} 
& & M \boxtimes W(\widetilde{M})
\ar{ur}{\simeq}
\\
&
N
\ar[near start]{rr}{}
& & M
\\
N
\ar{ur}{=}
\ar{rr}
& & M
\ar[crossing over, leftarrow, near start]{uu}{}
\ar{ur}{=}
\end{tikzcd}\] of graded ${\Bbb E}_{\infty}$-spaces. The back and right-hand vertical faces are cartesian, and the front face is readily seen to be so by combining \cite[Corollary 11.4]{SS12} and \cite[Lemma 2.11]{Sag14}; we demonstrate the style of the argument below. From this, we see that $N \boxtimes W(\widetilde{M}) \xrightarrow{} \widetilde{N}$ is an equivalence over $N$, and so it remains to prove that $N \boxtimes W(\widetilde{M}) \to N$ is replete. By \cite[Lemma 2.9]{Lun21} and the fact that $W(\widetilde{M})$ is grouplike, we have equivalences \[N^{\rm gp} \boxtimes W(\widetilde{M}) \xrightarrow{\simeq} N^{\rm gp} \boxtimes W(\widetilde{M})^{\rm gp} \xrightarrow{\simeq} (N \boxtimes W(\widetilde{M}))^{\rm gp}.
\]Thus we have to show that \begin{equation}\label{splitreplete}\begin{tikzcd}[row sep = small]N \boxtimes W(\widetilde{M}) \ar{r} \ar{d} & N^{\rm gp} \boxtimes W(\widetilde{M}) \ar{d} \\ N \ar{r} & N^{\rm gp}\end{tikzcd}\end{equation} is cartesian. By \cite[Corollary 11.4]{SS12} this can be checked after applying $(-)_{h{\cal J}}$, and \cite[Lemma 2.11]{Sag14} implies that the resulting square is equivalent to one of the form \[\begin{tikzcd}[row sep = small]N_{h{\cal J}} \times W(\widetilde{M})_{h{\cal J}} \ar{r} \ar{d} & N^{\rm gp}_{h{\cal J}} \times W(\widetilde{M})_{h{\cal J}} \ar{d} \\ N_{h{\cal J}} \ar{r} & N^{\rm gp}_{h{\cal J}},\end{tikzcd}\] which is cartesian. 
\end{proof}

We remark that the square \eqref{splitreplete} being cartesian only depended upon $W(\widetilde{M})$ being a grouplike graded ${\Bbb E}_{\infty}$-space augmented over the initial object $U^{\cal J}$. This combines with \cite[Lemma 2.12]{Lun21} to prove:

\begin{corollary}\label{cor:splitreplete} Let $M \to N \to M$ be an augmented graded ${\Bbb E}_{\infty}$-space. Then $N \to M$ is replete if and only if there is an equivalence $M \boxtimes G \xrightarrow{\simeq} N$, where $G$ is a grouplike commutative ${\cal J}$-space monoid $G$ augmented over $U^{\cal J}$. \qed
\end{corollary}

\section{Logarithmic ring spectra}\label{sec:logringspectra} We now review some preliminary material on log ring spectra, following \cites{Rog09, SS12, Sag14, RSS15, RSS18}. 

\begin{definition} A \emph{pre-log ring spectrum} $(A, M, \alpha)$ consists of an ${\Bbb E}_{\infty}$-ring $A$, a graded ${\Bbb E}_{\infty}$-space $M$, and a map $\alpha \colon M \to \Omega^{\cal J}(A)$ of graded ${\Bbb E}_\infty$-spaces. 
\end{definition}

By adjunction, the structure map $\alpha$ gives rise to a unique map $\overline{\alpha} \colon {\Bbb S}^{\cal J}[M] \to A$ of ${\Bbb E}_{\infty}$-rings. We denote by ${\rm PreLog}$ the resulting category of pre-log ring spectra. The idea to use the Grothendieck construction in the following remark is based on a suggestion of an anonymous referee of the paper \cite{BLPO23}.

\begin{remark}[The $\infty$-category ${\rm PreLog}$]\label{rem:grothendieck} Let us give a more formal construction of the $\infty$-category ${\rm PreLog}$, as well as an explanation that it is presentable. To do this, we shall relate the projective model structure on pre-log ring spectra (as used in \cite{RSS15}) to the Grothendieck construction of model categories of Harpaz--Prasma \cite{HP15}. This exhibits ${\rm PreLog}$ as an instance of the $\infty$-categorical Grothendieck construction \cite[Proposition 3.1.2]{HP15}. 

We recall that ${\cal C}{\rm Sp}^{\Sigma}$ denotes the category of commutative symmetric ring spectra with its positive model structure. By \cite[Remark 6.1.3]{HP15}, the model structure obtained on the Grothendieck construction of the pseudo-functor \[{\cal C}{\rm Sp}^{\Sigma} \to {\rm ModCat}, \quad A \mapsto {\cal C}{\rm Sp}^{\Sigma}_{A/}\] is the projective model structure on the arrow category of ${\cal C}{\rm Sp}^{\Sigma}$. We would like to say that the projective model structure on pre-log ring spectra is obtained as the Grothendieck construction of the pseudo-functor \begin{equation}\label{pseudofunctor}{\cal C}{\cal S}^{\cal J} \xrightarrow{{\Bbb S}^{\cal J}[-]} {\cal C}{\rm Sp}^{\Sigma} \to {\rm ModCat}, \quad M \mapsto {\cal C}{\rm Sp}^{\Sigma}_{{\Bbb S}^{\cal J}[M]/}.\end{equation} Indeed, inspecting \cite[Definition 3.0.4]{HP15}, we find that this gives exactly the weak equivalences, fibrations, and cofibrations of the projective model structure of pre-log ring spectra. However, since ${\Bbb S}^{\cal J}[-]$ may not preserve \emph{all} weak equivalences, the composite \eqref{pseudofunctor} fails to be \emph{relative} in the sense of \cite[Definition 3.0.6]{HP15}: Weak equivalences do not necessarily induce Quillen equivalences. 

This is not a problem for our purposes. Precomposition with a cofibrant replacement functor determines a functor \[({\cal C}{\cal S}^{\cal J})^{\rm cof} \xrightarrow{{\Bbb S}^{\cal J}[-]} ({\cal C}{\rm Sp}^\Sigma)^{\rm cof} \to {\rm RelCat}, \quad M \mapsto {\cal C}{\rm Sp}^\Sigma_{{\Bbb S}^{\cal J}[M]/}\] which sends weak equivalences to Dwyer--Kan equivalences. As explained in \cite[Section 3.1]{HP15}, this determines a functor ${\cal C}{\cal S}^{\cal J}_{\infty} \to {\rm Cat}_{\infty}$, and we define ${\rm PreLog}$ to be its Grothendieck construction/unstraightening. Presentability of ${\rm PreLog}$ is now a consequence of Gepner--Haugseng--Nikolaus \cite[Theorem 10.3]{GHN17}. 
\end{remark}

\noindent The following remark is an informal summary of \cite[Construction 4.2]{Sag14}. 

\begin{remark}\label{rem:strictness} The examples discussed in the introduction, such as $({\rm ku}_p, \langle u \rangle_*)$, $(\ell_p, \langle v_1 \rangle_*)$, and $({\rm ko}_p, \langle \beta \rangle_*)$, all arise from homotopy classes $x \in \pi_d(R)$ that are \emph{strict} in a sense that we now elaborate upon.

As discussed in Section \ref{subsec:htpythygr}, any graded ${\Bbb E}_{\infty}$-space $M$ has an underlying ${\Bbb E}_{\infty}$-space $M_{h{\cal J}}$ with an augmentation to $QS^0$. We may associate to $x \in \pi_d(R)$ a graded ${\Bbb E}_{\infty}$-space $C(x)$ together with a map ${\Bbb S}^{\cal J}[C(x)] \to R$, and it is indeed the case that $C(x)_{h{\cal J}}$ is the free ${\Bbb E}_{\infty}$-space $\sqcup_{k \ge 0} B\Sigma_k$ on a single generator. Inverting $x$, we obtain a commutative outer square and a lifting problem \[\begin{tikzcd}[row sep = tiny]{\Bbb S}^{\cal J}[C(x)] \ar{rr} \ar{dd} \ar{dr} & & R \ar{dd} \\ \vspace{10 mm} & {\Bbb S}^{\cal J}[\langle x \rangle_*] \ar[dashed]{ur} \ar{dl} \\ {\Bbb S}^{\cal J}[C(x)^{\rm gp}] \ar{rr} & & R[x^{-1}].\end{tikzcd}\] The Barratt--Priddy--Quillen theorem implies that $C(x)^{\rm gp}_{h{\cal J}} \simeq QS^0$. By design, one ensures that $\langle x \rangle_{*, h{\cal J}} \simeq Q_{\ge 0}S^0 := QS^0 \times_{{\Bbb Z}} {\Bbb N}$; the non-negative path components of $QS^0$. The point is that $\langle x \rangle_{*, h{\cal J}}$, despite not being grouplike, enjoys the property that all of its path components are equivalent. This allows for the construction of $\langle \sqrt[n]{x} \rangle_*$, as described in e.g.\ \cite[Proof of Proposition 4.15]{Sag14} (see Section \ref{subsec:adjroots} for further discussion). 
\end{remark}

\subsection{Mapping spaces of pre-log ring spectra} For pre-log ring spectra $(A, M)$ and $(B, N)$, the space of maps ${\rm Map}_{\rm PreLog}((A, M), (B, N))$ sits in a cartesian square \begin{equation}\label{mappingspace}\begin{tikzcd}[row sep = small]{\rm Map}_{{\rm PreLog}}((A, M), (B, N)) \ar{r} \ar{d} & {\rm Map}_{{\cal C}{\cal S}^{\cal J}_{\infty}}(M, N) \ar{d} \\ {\rm Map}_{{\rm CAlg}}(A, B) \ar{r} & {\rm Map}_{{\rm CAlg}}({\Bbb S}^{\cal J}[M], A).\end{tikzcd}\end{equation} We observe that this follows from the description of ${\rm PreLog}$ as a Grothendieck construction sketched from Remark \ref{rem:grothendieck} by \cite[Proposition 2.4.4.3]{HTT}.

\subsection{The logification construction} The log condition in this setting reads:

\begin{definition} A pre-log ring spectrum $(A, M, \alpha)$ is \emph{log} if the map $\widetilde{\alpha}$ in the cartesian diagram \[\begin{tikzcd}[row sep = small]\alpha^{-1}{\rm GL}_1^{\cal J}(R) \ar{r}{\widetilde{\alpha}} \ar{d} & {\rm GL}_1^{\cal J}(R) \ar{d} \\ M \ar{r}{\alpha} & \Omega^{\cal J}(A)\end{tikzcd}\] of graded ${\Bbb E}_{\infty}$-spaces is an equivalence. 
\end{definition}

We shall write ${\rm Log}$ for the resulting category of log ring spectra. We now describe a left adjoint to the forgetful functor ${\rm Log} \to {\rm PreLog}$:

\begin{definition} If $(A, M, \alpha)$ is a pre-log ring spectrum, its \emph{logification} is defined by the cocartesian diagram \[\begin{tikzcd}[row sep = small]\alpha^{-1}{\rm GL}_1^{\cal J}(A) \ar{r}{\widetilde{\alpha}} \ar{d} & {\rm GL}_1^{\cal J}(A) \ar{d} \\ M \ar{r} & M^a,\end{tikzcd}\] with structure map $\alpha^a$ determined by $\alpha$ and the inclusion of the graded units ${\rm GL}_1^{\cal J}(A) \to \Omega^{\cal J}(A)$. The resulting pre-log ring spectrum $(A, M^a, \alpha^a)$ is log by \cite[Lemma 3.12]{Sag14}, while the resulting functor \[(-)^a \colon {\rm PreLog} \to {\rm Log}\] is left adjoint to the forgetful functor by \cite[Lemma 6.4]{Sag14}. 
\end{definition}

\begin{remark} In \cite[Section 3.2]{SSV16}, a \emph{log} model structure on the category of simplicial pre-log rings is established. It arises as a left Bousfield localization of the injective model structure on simplicial pre-log rings, and its fibrant objects are precisely those simplicial pre-log rings that satisfy the log condition. The \emph{verbatim} translation to the context of pre-log ring spectra is spelled out in the author's thesis \cite[Section 4.5.1]{Lun22}. The category ${\rm Log}$ is thus a localization of ${\rm PreLog}$, and we may model $(-)^a$ as the corresponding localization functor. We observe that, if $(A, M, \alpha)$ is log, the resulting morphism $(A, \alpha^{-1}{\rm GL}_1^{\cal J}(A)) \to (A, M)$ is adjoint to a map $(A, {\rm GL}_1^{\cal J}(A)) \to (A, M)$. 
\end{remark}

\subsection{Strict morphisms of log ring spectra} In our setup, we shall work with the following notion of strict morphisms:

\begin{definition} A map $(f, f^\flat) \colon (R, P) \to (A, M)$ of log ring spectra is \emph{strict} if the canonical map $(A, P^a) \to (A, M)$ is an equivalence. 
\end{definition}

The following is an analog of \cite[Proposition III.1.2.5]{Ogu18}:

\begin{lemma}\label{lem:strictiff} A morphism $(f, f^\flat) \colon (R, P) \to (A, M)$ of log ring spectra is strict if and only if the square \begin{equation}\label{squareofthelemma}\begin{tikzcd}[row sep = small](R, {\rm GL}_1^{\cal J}(R)) \ar{r}{(f, {\rm GL}_1^{\cal J}(f))} \ar{d} & (A, {\rm GL}_1^{\cal J}(A)) \ar{d} \\ (R, P) \ar{r}{(f, f^\flat)} & (A, M)\end{tikzcd}\end{equation} is cocartesian in the category of log ring spectra.
\end{lemma}

\begin{proof} The square \[\begin{tikzcd}[row sep = small](R, {\rm GL}_1^{\cal J}(R)) \ar{r}{(f, {\rm id})} \ar{d} & (A, {\rm GL}_1^{\cal J}(R)) \ar{d} \\ (R, P) \ar{r}{(f, {\rm id})} & (A, P)\end{tikzcd}\] is cocartesian in the category of pre-log ring spectra. If $(f, f^\flat)$ is strict, the logification of this square is the cocartesian square of log ring spectra predicted by the lemma. Conversely, \eqref{squareofthelemma} being cocartesian implies that $(f, f^\flat)$ is the cobase-change of the strict morphism $(R, {\rm GL}_1^{\cal J}(R)) \to (A, {\rm GL}_1^{\cal J}(A))$, and hence it is itself strict  (by e.g.\ the argument of \cite[Lemma 5.5]{SSV16}). 
\end{proof}

\subsection{Repletion} As in the case of ordinary pre-log rings, the repletion construction extends from graded ${\Bbb E}_{\infty}$-spaces to pre-log ring spectra:

\begin{definition} Let $(B, N) \to (A, M)$ be a map of pre-log ring spectra with $N \to M$ virtually surjective. The \emph{repletion} $(B^{\rm rep}, N^{\rm rep}) \to (A, M)$ over $(A, M)$ is defined by setting $B^{\rm rep} := B \otimes_{{\Bbb S}^{\cal J}[N]} {\Bbb S}^{\cal J}[N^{\rm rep}]$.
\end{definition}

\begin{remark} In \cite[Proposition 9.1]{Lun21}, we explain how to adapt the localization that gives the group completion model structure to exhibit ${\rm PreLog}^{\rm rep}_{(A, M)//(A, M)}$ as a localization of ${\rm PreLog}_{(A, M)//(A, M)}$. 
\end{remark}

\section{The replete tangent bundle}\label{sec:repletetangent}

Following a quick recollection of the tangent bundle construction of \cite[Section 7.3]{Lur17}, we construct a presentable fibration over ${\rm Log}$ that we call the \emph{replete tangent bundle}. Using this, we prove Theorems \ref{thm:repletetangent} and \ref{thm:logcotangent}. In addition to the stable envelope construction, we shall also use \emph{pointed envelopes} \cite[Definition 1.1]{dagiv}. These only briefly appear (but are never defined) in \cite{Lur17}; for this reason, our exposition is at times closer to that of \cite{dagiv} than \cite{Lur17}. 

\subsection{Recollections on the tangent bundle} Recall from \cite[Definition 5.5.3.2]{HTT} that a map of simplicial sets is a \emph{presentable fibration} if it is a (co)Cartesian fibration and its fibers are presentable $\infty$-categories. Throughout, we shall be concerned with presentable fibrations ${\cal E} \to {\cal C}$ where ${\cal C}$ is a presentable $\infty$-category. 

We invite the reader to keep in mind the example of the ``evaluation at the codomain''-functor \begin{equation}\label{coevaluation}{\rm Fun}(\Delta^1, {\cal C}) \to {\rm Fun}(\{1\}, {\cal C}) \simeq {\cal C}, \quad (A \to B) \mapsto B,\end{equation} for a presentable $\infty$-category ${\cal C}$, in which case the fibers are equivalent to ${\cal C}_{/B}$. 

To any presentable $\infty$-category ${\cal C}$ one can associate a \emph{pointed envelope} and a \emph{stable envelope} \cite[Definition 1.1]{dagiv}. By \cite[Example 1.7]{dagiv}, the inclusion of the full subcategory of pointed objects ${\cal C}_* \to {\cal C}$ is a pointed envelope of ${\cal C}$, while by \cite[Example 1.4]{dagiv} the composite ${\rm Stab}({\cal C}) \xrightarrow{\Omega^\infty_{{\cal C}_*}} {\cal C}_* \xrightarrow{} {\cal C}$ is a stable envelope of ${\cal C}$, where ${\rm Stab}({\cal C})$ is the stabilization of ${\cal C}$. Pointed and stable envelopes are unique up to equivalence \cite[Remark 1.8]{dagiv}.  

\begin{definition}(\cite[Definition 1.1]{dagiv}) Let $p \colon {\cal E} \to {\cal C}$ be a presentable fibration. A \emph{pointed envelope} of $p$ is a categorical fibration $u \colon {\cal E}' \to {\cal E}$ such that 
\begin{enumerate}
\item The composite functor ${\cal E}' \xrightarrow{u} {\cal E} \xrightarrow{p} {\cal C}$ is a presentable fibration;
\item the functor $u$ carries $(p \circ u)$-Cartesian morphisms to $p$-Cartesian morphisms; and
\item for every object $C \in {\cal C}$, the map ${\cal E}'_C \to {\cal E}_C$ is a pointed envelope of ${\cal E}_C$.
\end{enumerate} 
The notion of a \emph{stable envelope} is obtained by exchanging ``pointed'' for ``stable'' in the last condition. 
\end{definition}

The tangent bundle of a presentable $\infty$-category ${\cal C}$ is an instance of the stable envelope construction:

\begin{definition}(\cite[Definition 7.3.1.9]{Lur17}) A \emph{tangent bundle} $T_{\cal C} \to {\rm Fun}(\Delta^1, {\cal C})$ of ${\cal C}$ is a stable envelope of \eqref{coevaluation}. 
\end{definition}

\subsection{Recollections on the ${\Bbb E}_{\infty}$-cotangent complex} We now specialize to the case where ${\cal C} = {\rm CAlg}$ is the category of ${\Bbb E}_{\infty}$-rings, and we model the stabilization of a presentable $\infty$-category by its category of spectrum objects. Upon taking fibers above an ${\Bbb E}_{\infty}$-ring $A$, we thus obtain the stable envelope ${\rm Sp}({\rm CAlg}_{/A})$ of ${\rm CAlg}_{/A}$; we informally depict this as
\[\begin{tikzcd}[row sep = small]{\rm Sp}({\rm CAlg}_{/A}) \ar{r} \ar{d} & {\rm CAlg}_{/A} \ar{d} \ar{r} & \{A\} \ar{d} \\ T_{{\rm CAlg}} \ar{r} & {\rm Fun}(\Delta^1, {\rm CAlg}) \ar{r} & {\rm CAlg}.\end{tikzcd}\] 
The analog of Theorem \ref{thm:quillen}(1) is \cite[Corollary 7.3.4.14]{Lur17}, which states that ${\rm Sp}({\rm CAlg}_{/A}) \simeq {\rm Mod}_A$. By \cite[Theorem 7.3.4.18]{Lur17}, we can think of an object of the tangent bundle $T_{\rm CAlg}$ as a pair $(A, J)$ for $J$ an $A$-module, and the lower horizontal composite above, then, informally reads $(A, J) \mapsto (A \oplus J \to A) \mapsto A$. 

\begin{definition}(\cite[Definitition 7.3.2.14]{Lur17}) The \emph{absolute cotangent complex} ${\Bbb L} \colon {\rm CAlg} \to T_{\rm CAlg}$ is the composite ${\rm CAlg} \xrightarrow{} {\rm Fun}(\Delta^1, {\rm CAlg}) \xrightarrow{} T_{{\rm CAlg}}$ of the diagonal embedding and a left adjoint to $T_{{\rm CAlg}} \to {\rm Fun}(\Delta^1, {\rm CAlg})$ relative to ${\rm CAlg}$ (in the sense of \cite[Definition 7.3.2.2]{Lur17}). 
\end{definition}

By \cite[Remark 7.3.2.17]{Lur17}, the absolute cotangent complex specializes to $\Sigma_A^{\infty}(A) := \Sigma_+^\infty(A) \in {\rm Sp}({\rm CAlg}_{/A})$ on each fiber. Under the equivalence ${\rm Sp}({\rm CAlg}_{/A}) \simeq {\rm Mod}_A$, this recovers the $A$-module computing topological Andr\'e--Quillen homology \cite[Remark 7.3.0.1]{Lur17}. See also Basterra--Mandell \cite{BM05}. 

\subsection{Towards the replete tangent bundle} We now aim to carry out these constructions in the context of log ring spectra. The category of pre-log ring spectra does not seem to admit a description as the category of algebras over an operad, and as such it does not immediately fit in the framework of Lurie's cotangent complex formalism. We will overcome this by following the same strategy as in Section \ref{sec:prelude} in this framework. 

We begin by establishing an analog of Corollary \ref{cor:replaugmented}. For this, we will need the following construction:

\begin{construction}\label{constr:splitlogstr} Let $(A, M)$ be a pre-log ring spectrum and consider an object $(B, N) \in {\rm PreLog}_{(A, M)//(A, M)}$. Applying graded units to the structure maps $A \to B \to A$, we obtain a diagram ${\rm GL}_1^{\cal J}(A) \to {\rm GL}_1^{\cal J}(B) \to {\rm GL}_1^{\cal J}(A)$ of grouplike commutative ${\cal J}$-space monoids, and so \cite[Lemma 2.12]{Lun21} applies to obtain an equivalence \begin{equation}\label{unitssplit}{\rm GL}_1^{\cal J}(A) \boxtimes ({\rm GL}_1^{\cal J}(B)/{\rm GL}_1^{\cal J}(A)) \xrightarrow{\simeq} {\rm GL}_1^{\cal J}(B),\end{equation} where ${\rm GL}_1^{\cal J}(B)/{\rm GL}_1^{\cal J}(A)$ is defined as the pullback of $U^{\cal J} \xrightarrow{} {\rm GL}_1^{\cal J}(A) \xleftarrow{} {\rm GL}_1^{\cal J}(B)$, where $U^{\cal J}$ is the initial object in ${\cal C}{\cal S}^{\cal J}_{\infty}$. The composites $M \xrightarrow{} N \xrightarrow{} \Omega^{\cal J}(B)$ and \[{\rm GL}_1^{\cal J}(B)/{\rm GL}_1^{\cal J}(A) \xrightarrow{} {\rm GL}_1^{\cal J}(A) \boxtimes ({\rm GL}_1^{\cal J}(B)/{\rm GL}_1^{\cal J}(A)) \xrightarrow{\simeq} {\rm GL}_1^{\cal J}(B) \xrightarrow{} \Omega^{\cal J}(B)\] induce a pre-log structure $M \boxtimes ({\rm GL}_1^{\cal J}(B)/{\rm GL}_1^{\cal J}(A)) \to \Omega^{\cal J}(B)$ on $B$. 
\end{construction}

The following is an analog of Lemma \ref{lem:replogiremains}.

\begin{lemma}\label{lem:replogiremains2} Let $(B, N) \in {\rm PreLog}^{\rm rep}_{(A, M)//(A, M)}$ be an augmented replete pre-log ring spectrum over a pre-log ring spectrum $(A, M, \alpha)$. Then $(B, N^a)$ is naturally equivalent to $(B, M^a \boxtimes ({\rm GL}_1(B)/{\rm GL}_1(A)))$. In particular, $(B, N^a)$ is an augmented replete $(A, M^a)$-algebra. 
\end{lemma}

\begin{proof} Our proof is an adaptation of that of Lemma \ref{lem:replogiremains} in this context. By the assumption that the structure map $N \to M$ is replete, \cite[Lemma 2.12]{Lun21} applies to obtain an equivalence $M \boxtimes W(N) \xrightarrow{\simeq} N$, where $W(N)$ is defined as the pullback of $U^{\cal J} \xrightarrow{} M^{\rm gp} \xleftarrow{} N^{\rm gp}$. The resulting (equivalent) structure map $M \boxtimes W(N) \xrightarrow{\simeq} N \xrightarrow{} \Omega^{\cal J}(B)$ factors as the lower horizontal composite in the diagram 
\[\begin{tikzcd}[row sep = small]\alpha^{-1}{\rm GL}_1^{\cal J}(A) \boxtimes W(N) \ar{d} \ar{r} & {\rm GL}_1^{\cal J}(A) \boxtimes {\rm GL}_1^{\cal J}(B) \ar{d} \ar{r} & {\rm GL}_1^{\cal J}(B) \ar{d} \\ M \boxtimes W(N) \ar{r} & \Omega^{\cal J}(A) \boxtimes {\rm GL}_1^{\cal J}(B) \ar{r} & \Omega^{\cal J}(B)\end{tikzcd}\] of commutative ${\cal J}$-space monoids. The left-hand square is clearly cartesian, while the right-hand square is cartesian by Lemma \ref{lem:unitsmultiplytomakecartesian} below. It follows that the logification $N^a$ is determined by the outer cocartesian rectangle  \[\begin{tikzcd}[row sep = small]\alpha^{-1}{\rm GL}_1^{\cal J}(A) \boxtimes W(N) \ar{d} \ar{r} & {\rm GL}_1^{\cal J}(A) \boxtimes {\rm GL}_1^{\cal J}(B) \ar{d} \ar{r} & {\rm GL}_1^{\cal J}(B) \ar{d} \\ M \boxtimes W(N) \ar{r} & M^a \boxtimes {\rm GL}_1^{\cal J}(B) \ar{r} & N^a.\end{tikzcd}\] The left-hand square is cocartesian by definition, and so the right-hand square is cocartesian. Since both squares in the rectangle  \[\begin{tikzcd}[row sep = small] {\rm GL}_1^{\cal J}(A) \ar{d} \ar{r} & {\rm GL}_1^{\cal J}(A) \boxtimes {\rm GL}_1^{\cal J}(B) \ar{d} \ar{r} & {\rm GL}_1^{\cal J}(B) \ar{d} \\ M^a \ar{r} & M^a \boxtimes {\rm GL}_1^{\cal J}(B) \ar{r} & N^a\end{tikzcd}\] are cocartesian, the result follows from the splitting \eqref{unitssplit}.  
\end{proof}

In the proof of Lemma \ref{lem:replogiremains2}, we used:

\begin{lemma}\label{lem:unitsmultiplytomakecartesian} Let $A$ be an ${\Bbb E}_{\infty}$-ring and let $B \in {\rm CAlg}_{A//A}$ be an augmented $A$-algebra. The square \[\begin{tikzcd}[row sep = small]{\rm GL}_1^{\cal J}(A) \boxtimes {\rm GL}_1^{\cal J}(B) \ar{d} \ar{r} & {\rm GL}_1^{\cal J}(B) \ar{d} \\ \Omega^{\cal J}(A) \boxtimes {\rm GL}_1^{\cal J}(B) \ar{r} & \Omega^{\cal J}(B)\end{tikzcd}\] is cartesian. 
\end{lemma} 

\begin{proof} This is very similar to the argument of Lemmas 
\ref{lem:unitspullback} and \ref{lem:repletebasechangestable}. Combining \cite[Corollary 11.4]{SS12} and \cite[Lemma 2.11]{Sag14} as in the proof of Lemma \ref{lem:repletebasechangestable}, we find that it suffices to prove that the front face of the commutative cube \[\begin{tikzpicture}[baseline= (a).base]
\node[scale=.85] (a) at (0,0){\begin{tikzcd}[column sep = tiny, row sep = tiny]
&
\pi_0({\rm GL}_1^{\cal J}(A)_{h{\cal J}}) \times \pi_0({\rm GL}_1^{\cal J}(B)_{h{\cal J}})
\ar{rr}{}
\ar[]{dd}[near end]{}
& & \pi_0({\rm GL}_1^{\cal J}(B)_{h{\cal J}})
\ar{dd}{}
\\
{\rm GL}_1^{\cal J}(A)_{h{\cal J}} \times {\rm GL}_1^{\cal J}(B)_{h{\cal J}}
\ar[crossing over]{rr}[near start]{}
\ar{dd}[swap]{}
\ar{ur}
& & {\rm GL}_1^{\cal J}(B)_{h{\cal J}}
\ar{ur}
\\
&
\pi_0(\Omega^{\cal J}(A)_{h{\cal J}}) \times \pi_0({\rm GL}_1^{\cal J}(B)_{h{\cal J}})
\ar[near start]{rr}{}
& & \pi_0(\Omega^{\cal J}(B)_{h{\cal J}})
\\
\Omega^{\cal J}(A)_{h{\cal J}} \times {\rm GL}_1^{\cal J}(B)_{h{\cal J}} 
\ar{ur}
\ar{rr}
& & \Omega^{\cal J}(B)_{h{\cal J}}
\ar[crossing over, leftarrow, near start]{uu}{}
\ar[swap]{ur}{}
\end{tikzcd}};\end{tikzpicture}\] is cartesian. The right-hand face is cartesian, as for any commutative ${\cal J}$-space monoid $M$, an element of the graded signed monoid $\pi_{0, *}(M)$ is a unit if and only if it represents one in $\pi_0(M_{h{\cal J}})$ (cf.\ \cite[Corollary 4.16]{SS12} and the surrounding discussion). For the same reason, the left-hand face is cartesian. The back face is isomorphic to \[\begin{tikzcd}[row sep = small]{\rm GL}_1(\pi_*(A))/\{\pm 1\} \times {\rm GL}_1(\pi_*(B))/\{\pm 1\} \ar{r} \ar{d} & {\rm GL}_1(\pi_*(B))/\{\pm 1\} \ar{d} \\ \pi_*(A)/\{\pm 1\} \times {\rm GL}_1(\pi_*(B))/\{\pm 1\} \ar{r} & \pi_*(B)/\{\pm 1\}.\end{tikzcd}\] Since $A \to B$ admits a retraction, it is precisely the units of $\pi_*(A)$ that map to units in $\pi_*(B)$, so that this square is cartesian. Hence the front face of the cube is cartesian, as desired.  
\end{proof}

\begin{proposition}\label{prop:augmentedalgrepl} Let $(A, M)$ be a log ring spectrum. The forgetful functor ${\rm Log}^{\rm rep}_{(A, M)//(A, M)} \to {\rm CAlg}_{A//A}$ is an equivalence.  
\end{proposition}

\begin{proof} Throughout this proof, we shall use the shorthand $G_B$ for ${\rm GL}_1^{\cal J}(B)/{\rm GL}_1^{\cal J}(A)$. The functor is essentially surjective, as for any $B \in {\rm CAlg}_{A//A}$, it lifts to the replete augmented $(A, M)$-algebra $(B, M \boxtimes G_B)$ by Lemma \ref{lem:replogiremains2}. It thus remains to prove that it is fully faithful. For this, it suffices to prove that, given augmented $A$-algebras $B$ and $C$, the left-hand vertical map \[\begin{tikzpicture}[baseline= (a).base]
\node[scale=.9] (a) at (0,0){\begin{tikzcd}[row sep = small]{\rm Map}_{{\rm Log}_{(A, M)//(A, M)}}((B, M \boxtimes G_B), (C, M \boxtimes G_C)) \ar{r} \ar{d} & {\rm Map}_{({\cal C}{\cal S}^{\cal J}_{\infty})_{M // M}}(M \boxtimes G_B, M \boxtimes G_C) \ar{d} \\ {\rm Map}_{{\rm CAlg}_{A//A}}(B, C) \ar{r} & {\rm Map}_{{\rm CAlg}_{{\Bbb S}^{\cal J}[M] // A}}({\Bbb S}^{\cal J}[M \boxtimes G_B], C)\end{tikzcd}};\end{tikzpicture}\] in the defining cartesian square \eqref{mappingspace} for mapping spaces in ${\rm Log}$ is an equivalence. It thus suffices to prove that the right-hand vertical map is an equivalence. We consider the commutative diagram \[\begin{tikzpicture}[baseline= (a).base]
\node[scale=.98] (a) at (0,0){\begin{tikzcd}[row sep = small]{\rm Map}_{({\cal C}{\cal S}^{\cal J}_{\infty})_{M // M}}(M \boxtimes G_B, M \boxtimes G_C) \ar{r} \ar{dr} \ar{dd}{\simeq} \ar[swap]{dd}{\text{Cobase-change } U^{\cal J} \to M} & {\rm Map}_{{\rm CAlg}_{{\Bbb S}^{\cal J}[M] // A}}({\Bbb S}^{\cal J}[M \boxtimes G_B], C) \\ \vspace{10 mm} & {\rm Map}_{({\cal C}{\cal S}^{\cal J}_{\infty})_{M // \Omega^{\cal J}(A)}}(M \boxtimes G_B, \Omega^{\cal J}(C)) \ar{u}{\simeq} \ar[swap]{u}{({\Bbb S}^{\cal J}, \Omega^{\cal J})\text{-adjunction}} \ar[swap]{d}{\simeq} \ar{d}{\text{Cobase-change } U^{\cal J} \to M} \\ {\rm Map}_{({\cal C}{\cal S}^{\cal J}_{\infty})_{/M}}(G_B, M \boxtimes G_C) \ar{r} & {\rm Map}_{({\cal C}{\cal S}^{\cal J}_{\infty})_{/\Omega^{\cal J}(A)}}(G_B, \Omega^{\cal J}(C)) \\ {\rm Map}_{({\cal C}{\cal S}^{\cal J}_{\infty})_{/U^{\cal J}}}(G_B, G_C) \ar[swap]{u} {\simeq} \ar{u}{\text{Base-change } U^{\cal J} \to M} \ar{r}{\simeq} \ar[swap]{r}{\substack{\text{Base-change } \\ U^{\cal J} \to {\rm GL}_1^{\cal J}(A)}} & {\rm Map}_{({\cal C}{\cal S}^{\cal J}_{\infty})_{/{\rm GL}_1^{\cal J}(A)}}(G_B, {\rm GL}_1^{\cal J}(C)) \ar{u}{\simeq} \ar[swap]{u}{{\rm GL}_1^{\cal J}(-) \text{ right adjoint}}\end{tikzcd}};\end{tikzpicture}\] of mapping spaces, where the arrows decorated $\simeq$ are equivalences for the indicated reason, and we have implicitly used that the category of grouplike commutative ${\cal J}$-space monoids is a full subcategory of all commutative ${\cal J}$-space monoids when utilizing that ${\rm GL}_1^{\cal J}(-)$ is a right adjoint. This concludes the proof. 
\end{proof}

\begin{example}\label{ex:loghh} Employing the homotopy invariant notions of logification and repletion of \cite{SSV16}, the equivalence of Corollary \ref{cor:replaugmented} extends to a Quillen equivalence of simplicial objects, as one can for instance see by imitating the proof of Proposition \ref{prop:augmentedalgrepl}. For a fixed pre-log ring $(A, M)$, this means that e.g.\ the composite functor \[{\rm sPreLog}_{(A, M) // (A, M)} \xrightarrow{(-, -)^{\rm rep}} {\rm sPreLog}_{(A, M^a) // (A, M^a)}^{\rm rep} \xrightarrow{(-, -)^a} {\rm sLog}^{\rm rep}_{(A, M^a) // (A, M^a)}\] from simplicial augmented $(A, M)$-algebras naturally takes values in simplicial augmented $A$-algebras (via the equivalence of Corollary \ref{cor:replaugmented}. The simplicial tensor $S^1 \otimes (A, M)$ is an augmented simplicial pre-log $(A, M)$-algebra, with augmentation induced by the collapse map of the circle. Under the displayed composite functor, this uniquely determines an augmented commutative $A$-algebra. By construction, this coincides with Rognes' \emph{log Hochschild homology} ${\rm HH}(A, M)$ of the pre-log ring $(A, M)$ \cite[Definition 3.23]{Rog09}. The argument of \cite[Theorem 4.24]{RSS15} (or its conjunction with \cite[Corollary 3.4]{BLPO23Prism}) shows that ${\rm HH}(-, -)$ is invariant under the logification construction. By Proposition \ref{prop:augmentedalgrepl}, the analogous remark also applies to log topological Hochschild homology ${\rm THH}(A, M)$. 
\end{example}

\begin{corollary}\label{cor:augmentedalgrepl} There is a canonical equivalence ${\rm Sp}({\rm Log}^{\rm rep}_{(A, M)//(A, M)}) \simeq {\rm Mod}_A$.
\end{corollary}

\begin{proof} This follows from Proposition \ref{prop:augmentedalgrepl} and \cite[Corollary 7.3.4.14]{Lur17}.
\end{proof}

The following is the spectral analog of Example \ref{ex:splitsquarezero2}. 

\begin{remark}\label{rem:splitsquarezero} Following \cite[Remark 7.3.4.16]{Lur17}, let ${\cal C}$ be a presentably symmetric monoidal, stable $\infty$-category, let $A \in {\rm CAlg}({\cal C})$ be a commutative algebra object, and let $J \in {\rm Mod}_A({\cal C})$ be an $A$-module. One defines the \emph{split square-zero extension} $A \oplus J$ as the image of $J$ under the composite \[{\rm Mod}_A({\cal C}) \simeq {\rm Sp}({\rm CAlg}({\cal C})_{/A}) \xrightarrow{\Omega^\infty} {\rm CAlg}({\cal C})_{/A},\] where the equivalence is \cite[Theorem 7.3.4.13]{Lur17}. While the category ${\rm Log}$ does not seem to fit in this framework, Proposition \ref{prop:augmentedalgrepl} and Corollary \ref{cor:augmentedalgrepl} suggest that the ``infinitesimal theory'' of log ring spectra can largely be ported from that of ordinary ${\Bbb E}_{\infty}$-rings. More explicitly, let us consider the composite \[{\rm Mod}_A \simeq {\rm Sp}({\rm Log}^{\rm rep}_{(A, M)//(A, M)}) \xrightarrow{\Omega^\infty} {\rm Log}^{\rm rep}_{(A, M)//(A, M)}.\] The image of an $A$-module $J$ identifies with $(A \oplus J, M \boxtimes {\rm GL}_1^{\cal J}(A \oplus J)/{\rm GL}_1^{\cal J}(A))$. By definition, this recovers the split square-zero extensions of log ring spectra used previously in the literature (cf.\ \cite[Definition 11.6]{Rog09} and \cite[Construction 5.6]{Sag14}), that we shall denote by $(A, M) \oplus J := (A \oplus J, M \oplus J)$. 
\end{remark}

\subsection{The replete pointed envelope} Let ${\cal C}$ be a presentable $\infty$-category. As explained in \cite[Notation 1.57]{dagiv}, one can explicitly model the pointed envelope $P_{\cal C} \to {\rm Fun}(\Delta^1, {\cal C})$ of the presentable fibration \eqref{coevaluation} as the full subcategory of ${\rm Fun}(\Delta^2, {\cal C})$ consisting of those triangles that compose to an equivalence; that is, those commutative diagrams \[\begin{tikzcd}[row sep = small]X \ar{dr} \ar{rr} & & Z \\ & Y \ar{ur} \end{tikzcd}\] for which the horizontal map $X \to Z$ an equivalence. The functor $P_{\cal C} \to {\rm Fun}(\Delta^1, {\cal C})$ informally sends such a diagram to the morphism $Y \to Z$, that is, it is induced by the evaluation ${\rm Fun}(\Delta^2, {\cal C}) \to {\rm Fun}(\Delta^{\{1, 2\}}, {\cal C}) \simeq {\rm Fun}(\Delta^1, {\cal C})$. 

\begin{definition} Let $P_{\rm Log}$ be the pointed envelope of ${\rm Fun}(\Delta^1, {\rm Log}) \to {\rm Log}$. The \emph{replete pointed envelope} $P_{{\rm Log}}^{\rm rep}$ is the full subcategory of $P_{\rm Log}$ spanned by those triangles \begin{equation}\label{repletetriangles}\begin{tikzcd}[row sep = small](A, M) \ar{dr} \ar{rr} & & (C, K) \\ & (B, N) \ar{ur} \end{tikzcd}\end{equation} for which $(B, N) \to (C, K)$ is replete.  The functor $P^{\rm rep}_{{\rm Log}} \to {\rm Fun}(\Delta^1, {\rm Log})$ is obtained by restriction of the functor $P_{\rm Log} \to {\rm Fun}(\Delta^1, {\rm Log})$.  
\end{definition}

We remark that, since $(A, M) \to (C, K)$ is an equivalence in the above triangle, the $M \to K$ is necessarily virtually surjective.

\begin{lemma}\label{lem:pointedfiber} The functor $P^{\rm rep}_{\rm Log} \to {\rm Log}$ is a presentable fibration with fibers $(P^{\rm rep}_{\rm Log})_{(A, M)}$ canonically equivalent to ${\rm CAlg}_{A // A}$. 
\end{lemma}

\begin{proof} It is an inner fibration, being the composite of the inclusion $P_{\rm Log}^{\rm rep} \to P_{\rm Log}$ and the presentable fibration $P_{\rm Log} \to {\rm Log}$. The condition of \cite[Definition 2.4.2.1(ii)]{HTT} is satisfied by Lemma \ref{lem:repletebasechangestable}. Hence $P_{\rm Log}^{\rm rep} \to {\rm Log}$ is a Cartesian fibration. By construction, the fibers over a log ring $(A, M)$ are equivalent to ${\rm Log}^{\rm rep}_{(A, M)//(A, M)}$ (cf.\ \cite[Remark 1.3, Example 1.7]{dagiv}), which is equivalent to ${\rm CAlg}_{A//A}$ by Proposition \ref{prop:augmentedalgrepl}. This concludes the proof. 
\end{proof}

By construction and Lemma \ref{lem:pointedfiber}, we have a commutative diagram \[\begin{tikzcd}[row sep = small]P_{\rm Log} \ar{dr} & & P_{\rm Log}^{\rm rep} \ar{ll} \ar{dl}\\ \vspace{10 mm} & {\rm Log}\end{tikzcd}\] of Cartesian fibrations over ${\rm Log}$. The theory of \emph{relative adjunctions} of \cite[Section 7.3.2]{Lur17} provides us with the following ``globalization'' of the repletion functor:

\begin{lemma}\label{lem:repglobal} The functor $P^{\rm rep}_{\rm Log} \to P_{\rm Log}$ admits a left adjoint relative to ${\rm Log}$, which on each fiber realizes the repletion functor ${\rm Log}_{(A, M)//(A, M)} \to {\rm Log}^{\rm rep}_{(A, M)//(A, M)}$. 
\end{lemma} 

\begin{proof} We wish to apply \cite[Proposition 7.3.2.6]{Lur17}. As the inclusion functor ${\rm Log}_{(A, M)//(A, M)}^{\rm rep} \to {\rm Log}_{(A, M)//(A, M)}$ admits a left adjoint (the repletion functor) for any log ring spectrum $(A, M)$, it suffices to check that $P_{\rm Log}^{\rm rep} \to P_{\rm Log}$ carries $(P_{\rm Log}^{\rm rep} \to {\rm Log})$-Cartesian morphisms to $(P_{\rm Log} \to {\rm Log})$-Cartesian morphisms. But this is clear using the characterization of \cite[Proposition 2.4.4.3]{HTT}, since the functor $P^{\rm rep}_{\rm Log} \to P_{\rm Log}$ is the inclusion of a full subcategory. 
\end{proof}

Objects \eqref{repletetriangles} of $P^{\rm rep}_{\rm Log}$ have replete structure maps $(B, N) \to (C, K)$ that admit a section up to homotopy, and so they are strict by Corollary \ref{cor:splitreplete} and \cite[Lemma 3.4.1.5]{Lun22}. This means that the only additional data necessary to determine the diagram \eqref{repletetriangles} is (1) the underlying diagram of ${\Bbb E}_{\infty}$-rings and (2) the log structure $K$ on the codomain $C$. This is made formal below, where we denote by $P_{{\rm CAlg}} \to {\rm CAlg}$ the pointed envelope of the functor ${\rm Fun}(\Delta^1, {\rm CAlg}) \to {\rm Fun}(\{1\}, {\rm CAlg}) \simeq {\rm CAlg}$. 

\begin{lemma}\label{lem:pointedeqlog} There is an equivalence \[P^{\rm rep}_{{\rm Log}} \xrightarrow{\simeq} {\rm Log} \times_{{\rm CAlg}} P_{\rm CAlg}\] of Cartesian fibrations over ${\rm Log}$. 
\end{lemma}

\begin{proof} The pullback ${\rm Log} \times_{{\rm CAlg}} P_{\rm CAlg} \to {\rm Log}$ is a Cartesian fibration since the pointed envelope $P_{\rm CAlg} \to {\rm CAlg}$ is a Cartesian fibration. The resulting functor $P_{\rm Log}^{\rm rep} \to {\rm Log} \times_{\rm CAlg} P_{\rm CAlg}$ of Cartesian fibrations over ${\rm Log}$ sends Cartesian morphisms over ${\rm Log}$ to Cartesian morphisms over ${\rm Log}$, since the forgetful functor ${\rm Log} \to {\rm CAlg}$ commutes with limits. The functor is an equivalence on each fiber by Proposition \ref{prop:augmentedalgrepl}, which concludes the proof by \cite[Corollary 2.4.4.4]{HTT}. 
\end{proof}

\begin{remark} In light of Lemma \ref{lem:pointedeqlog}, the reader may wonder why we simply did not \emph{define} the replete pointed envelope as the pullback ${\rm Log} \times_{{\rm CAlg}} P_{\rm CAlg} \to {\rm Log}$. This is related to the discussion of Remark \ref{rem:rognescomp}: We find this line of exposition to more clearly highlight the repletion functor as a means of ``cashing out'' the additional data provided by the log structure, while we would find it less natural to define an object which is \emph{a priori} independent of the log structure. In particular, we find statements like Lemma \ref{lem:repglobal} to be more transparent from this perspective, while the effect of the repletion procedure would be hidden in Proposition \ref{prop:augmentedalgrepl} if we took the pullback ${\rm Log} \times_{{\rm CAlg}} P_{\rm CAlg}$ as the definition of the replete pointed envelope. 
\end{remark}

\subsection{The replete tangent bundle} We are now ready to give the definition of the replete tangent bundle. By Lemma \ref{lem:pointedfiber}, the functor $P^{\rm rep}_{\rm Log} \to {\rm Log}$ is a presentable fibration, and so we may form its stable envelope:

\begin{definition} The \emph{replete tangent bundle} $T^{\rm rep}_{\rm Log} \to {\rm Log}$ is the stable envelope of the presentable fibration $P_{\rm Log}^{\rm rep} \to {\rm Log}$. \end{definition} 

The definition comes with some immediate pleasant consequences: 

\begin{corollary}\label{cor:repletefiber} For a log ring spectrum $(A, M)$, the fiber $(T^{\rm rep}_{\rm Log})_{(A, M)}$ of the replete tangent bundle is canonically equivalent to the category ${\rm Mod}_A$ of modules over the underlying ring spectrum $A$. 
\end{corollary}

\begin{proof} By Lemma \ref{lem:pointedfiber}, the fiber is equivalent to the stabilization ${\rm Sp}({\rm CAlg}_{A // A})$, and so the result follows from \cite[Corollary 7.3.4.14]{Lur17}.
\end{proof}

\begin{corollary}\label{cor:stableeqlog} There is an equivalence \[T^{\rm rep}_{\rm Log} \xrightarrow{\simeq} {\rm Log} \times_{{\rm CAlg}} T_{\rm CAlg}\] of Cartesian fibrations over ${\rm Log}$. 
\end{corollary}

\begin{proof} This follows from Lemma \ref{lem:pointedeqlog} and \cite[Remark 7.3.1.3]{Lur17}. 
\end{proof}

\subsection{Construction of the log cotangent complex} The following cube summarizes the construction of the replete tangent bundle $T^{\rm rep}_{\rm Log}$ in terms of the ordinary tangent bundle $T_{\rm PreLog}$ of the presentable $\infty$-category ${\rm PreLog}$: \[
\begin{tikzcd}[row sep = tiny, column sep = small]
&
{\rm Sp}({\rm Log}^{\rm rep}_{(A, M^a) // (A, M^a)})
\ar{rr}{}
\ar[]{dd}[near end]{}
& & T^{\rm rep}_{\rm Log}
\ar{dd}
\\
{\rm Sp}({\rm PreLog}_{/(A, M)})
\ar[crossing over]{rr}[near start]{}
\ar{dd}[swap]{}
\ar{ur}
& & T_{\rm PreLog}
\ar{ur}{}
\\
&
{\rm Log}^{\rm rep}_{(A, M^a) // (A, M^a)}
\ar[near start]{rr}{} \ar[near end]{dd}{}
& & P^{\rm rep}_{\rm Log} \ar{dd}
\\
{\rm PreLog}_{/(A, M)} \ar{dd}
\ar{ur}{}
\ar[crossing over]{rr}{}
& & {\rm Fun}(\Delta^1, {\rm PreLog})  
\ar[crossing over, leftarrow, near start,swap]{uu}{}
\ar[swap]{ur}{}
\\
& \{(A, M^a)\} \ar[near start]{rr}
& & 
{\rm Log}
\\
\{(A, M)\} \ar{ur} \ar{rr} & & {\rm PreLog} \ar{ur} \ar[crossing over, leftarrow, near start,swap]{uu}{}
\end{tikzcd}\]
Here the bottom square is defined by logification. The middle square is defined by logification, cobase-change, and repletion. More formally, we have defined the functor ${\rm Fun}(\Delta^1, {\rm PreLog}) \to P^{\rm rep}_{\rm Log}$ as the composite \begin{equation}\label{logcotangentcomp}{\rm Fun}(\Delta^1, {\rm PreLog}) \to {\rm Fun}(\Delta^1, {\rm Log}) \to P_{\rm Log} \to P^{\rm rep}_{\rm Log}\end{equation} given by logification, the left adjoint provided by \cite[Proof of Lemma 7.3.3.21]{Lur17}, and the left adjoint provided by Lemma \ref{lem:repglobal}, respectively. The top square is defined to be the stable envelope of the bottom cube.

\begin{remark}\label{rem:lifeiscomplicated} Due to our passing to pointed envelopes in \emph{log} ring spectra, the fiberwise description of \eqref{logcotangentcomp} is more complicated than in the case of Construction \ref{constr:replabelianization2}. On each fiber, we have the description \[((B, N) \to (A, M)) \mapsto ((A, M^a) \to ((A \otimes B)^{\rm rep}, (M \boxtimes N)^{a, {\rm rep}}) \to (A, M^a)).\] Here the repletion is taken with respect to the map $(M^a \boxtimes N^a)^a \to M^a$ induced by the identity on $M^a$ and $N^a \to M^a$. Observe that we do not mimic Construction \ref{constr:replabelianization2} directly, as we do not know how to construct the relevant (variants of) relative adjunctions when the base category varies. For this reason, we pass to log ring spectra immediately.
\end{remark}

By Corollary \ref{cor:repletefiber} (or Corollary \ref{cor:stableeqlog} combined with \cite[Theorem 7.3.4.18]{Lur17}), we can think of the objects of the replete tangent bundle $T^{\rm rep}_{\rm Log}$ as triples $((A, M), J)$ with $(A, M)$ a log ring spectrum and $J$ an $A$-module. From this point of view, we may describe the composite $T_{\rm Log}^{\rm rep} \to P^{\rm rep}_{\rm Log} \to {\rm Log}$ by the formula \[((A, M), J) \mapsto ((A, M) \to (A, M) \oplus J \to (A, M)) \mapsto (A, M),\] cf.\ Remark \ref{rem:splitsquarezero}. 

\begin{lemma}\label{lem:cottxleft} The functor $T^{\rm rep}_{\rm Log} \to P^{\rm rep}_{\rm Log}$ admits a left adjoint relative to ${\rm Log}$.\end{lemma} 

\begin{proof} On each fiber, the functor under consideration is equivalent to \[{\rm Sp}({\rm Log}^{\rm rep}_{(A, M) // (A, M)}) \xrightarrow{\Omega^\infty} {\rm Log}^{\rm rep}_{(A, M)//(A, M)},\] which admits the left adjoint $\Sigma^\infty$. We now aim to apply \cite[Proposition 7.3.2.6]{Lur17} to the diagram \[\begin{tikzcd}[row sep = small]P^{\rm rep}_{\rm Log} \ar{dr} & & T^{\rm rep}_{\rm Log} \ar{ll} \ar{dl} \\ \vspace{10 mm} & {\rm Log}.\end{tikzcd}\] For this, we have to show that morphisms in $T^{\rm rep}_{\rm Log}$ that are cartesian over ${\rm Log}$ are sent to morphisms in $P^{\rm rep}_{\rm Log}$ that are cartesian over ${\rm Log}$. But in light of Lemma \ref{lem:pointedeqlog} and Corollary \ref{cor:stableeqlog}, this is true for the same reason that it is true for the functor $T_{\rm CAlg} \to P_{\rm CAlg}$ over ${\rm CAlg}$ (cf.\ \cite[Definition 7.3.2.14]{Lur17}): Given a cartesian morphism $((A, M), J_1) \to ((B, N), J_2)$ over ${\rm Log}$ in $T^{\rm rep}_{\rm Log}$, there is an equivalence of $A$-modules $J_1 \xrightarrow{\simeq} J_2$. Hence the square \[\begin{tikzcd}[row sep = small](A, M) \oplus J_1\ar{r} \ar{d} & (A, M) \ar{d} \\ (B, N) \oplus J_2 \ar{r} & (B, N)\end{tikzcd}\] is cartesian, which concludes the proof. 
\end{proof}

\noindent Let us pause to record that our results assemble to a full proof of Theorem \ref{thm:repletetangent}. 

\begin{proof}[Proof of Theorem \ref{thm:repletetangent}] Part (1) is Corollary \ref{cor:repletefiber} and part (2) is Lemma \ref{lem:cottxleft}. 
\end{proof}

\begin{definition} The \emph{log cotangent complex} is the composite functor \[{\Bbb L}^{\rm rep} \colon {\rm PreLog} \xrightarrow{(-)^a} {\rm Log} \xrightarrow{} {\rm Fun}(\Delta^1, {\rm Log}) \xrightarrow{} P_{\rm Log} \xrightarrow{} P^{\rm rep}_{\rm Log} \xrightarrow{} T^{\rm rep}_{\rm Log}\] given by logification, the diagonal embedding, the left adjoint provided by \cite[Proof of Lemma 7.3.3.21]{Lur17}, the left adjoint provided by Lemma \ref{lem:repglobal}, and the left adjoint provided by Lemma \ref{lem:cottxleft}, respectively. 
\end{definition}

\subsection{Identifying the log cotangent complex} By Corollary \ref{cor:repletefiber}, for each pre-log ring spectrum $(A, M)$ we obtain an $A$-module that we denote by ${\Bbb L}_{(A, M)}^{\rm rep}$ - the \emph{log cotangent complex} of $(A, M)$. A variant ${\rm TAQ}(A, M)$ of the log cotangent complex was introduced and studied by Rognes \cite[Sections 11 and 13]{Rog09} and Sagave \cite{Sag14} - therein referred to as \emph{log topological Andr\'e--Quillen homology}. We now aim to prove Theorem \ref{thm:logcotangent}, which states that these $A$-modules are canonically equivalent. We first simplify the rather complicated description of Remark \ref{rem:lifeiscomplicated}. Let $(B, N) \to (A, M)$ be a map of log ring spectra. There is a map \begin{equation}\label{lifeismanageablemap}(A \otimes B) \otimes_{{\Bbb S}^{\cal J}[M \boxtimes N]} {\Bbb S}^{\cal J}[(M \boxtimes N)^{\rm rep}] \to (A \otimes B) \otimes_{{\Bbb S}^{\cal J}[(M \boxtimes_{} N)^a]} {\Bbb S}^{\cal J}[(M \boxtimes_{} N)^{a, {\rm rep}}].\end{equation} The codomain is as in Remark \ref{rem:lifeiscomplicated}. In the domain, the repletion is taken with respect to the map induced by the identity on $M$ and the map $N \to M$, and the map is induced by logification.

\begin{proposition}\label{prop:lifeismanageable} The map \eqref{lifeismanageablemap} is an equivalence after one suspension in the category of augmented $A$-algebras.  
\end{proposition}

\begin{proof} We very closely follow the proof strategy of \cite[Theorem 4.24]{RSS15}, of which the argument here should be considered a special case. We spell out the argument for the convenience of the reader. Let us write $(C, K, \gamma)$ for the pre-log ring spectrum $(A \otimes B, M \boxtimes_{} N)$. Let $\gamma^{-1}{\rm GL}_1^{\cal J}(C)^{\rm rep}$ denote the repletion of the morphism $\gamma^{-1}{\rm GL}_1^{\cal J}(C) \to \alpha^{-1}{\rm GL}_1^{\cal J}(A)$ (which is virtually surjective, as $(C, K) \to (A, M)$ admits a section). For a cocomplete category ${\cal C}$ pointed at $c$, we shall write $S^1 \odot_c -$ for the pointed tensor with $S^1$. Consider first the pushout diagram \begin{equation}\label{gammapushout}\begin{tikzpicture}[baseline= (a).base]
\node[scale=.84] (a) at (0,0){\begin{tikzcd}[row sep = small, column sep = tiny]S^1 \odot_{\alpha^{-1}{\rm GL}_1^{\cal J}(A)} \gamma^{-1}{\rm GL}_1^{\cal J}(C) \ar{d} \ar{r} & S^1 \odot_{\alpha^{-1}{\rm GL}_1^{\cal J}(A)^{\rm gp}} \gamma^{-1}{\rm GL}_1^{\cal J}(C)^{\rm gp} \ar{d} \\ S^1 \odot_{\alpha^{-1}{\rm GL}_1^{\cal J}(A)} \gamma^{-1}{\rm GL}_1^{\cal J}(C)^{\rm rep} \ar{r} & S^1 \odot_{\alpha^{-1}{\rm GL}_1^{\cal J}(A)^{\rm gp}} (\gamma^{-1}{\rm GL}_1^{\cal J}(C)^{\rm rep} \boxtimes_{\gamma^{-1}{\rm GL}_1^{\cal J}(C)} \gamma^{-1}{\rm GL}_1^{\cal J}(C)^{\rm gp}).\end{tikzcd}};\end{tikzpicture}\end{equation} of graded ${\Bbb E}_{\infty}$-spaces. The codomain of the right-hand vertical map is grouplike, as its domain is grouplike, and the left-hand vertical map is  surjective on $\pi_0((-)_{h{\cal J}})$ by Lemma \ref{lem:isoonpizerohj} below. Combining this with the facts that group completions commute with homotopy pushouts \cite[Lemma 2.9]{Lun21} and that the left-hand vertical map is an equivalence after group completion  implies that the right-hand vertical map was an equivalence to begin with. Observe that the map $\gamma^{-1}{\rm GL}_1^{\cal J}(C) \to \Omega^{\cal J}(C)$ factors over ${\rm GL}_1^{\cal J}(C)$, and hence over $\gamma^{-1}{\rm GL}_1^{\cal J}(C)^{\rm gp}$. From this, we obtain the diagram \begin{equation}\label{gammaeq}\begin{tikzcd}[row sep = small, column sep = tiny] \vspace{10 mm}  & C \xrightarrow{\cong} C \otimes_{{\Bbb S}^{\cal J}[\gamma^{-1}{\rm GL}_1^{\cal J}(C)^{\rm gp}]} {\Bbb S}^{\cal J}[\gamma^{-1}{\rm GL}_1^{\cal J}(C)^{\rm gp}] \ar{d}{} \\ \vspace{10 mm} & C \otimes_{{\Bbb S}^{\cal J}[\gamma^{-1}{\rm GL}_1^{\cal J}(C)^{\rm gp}]} ({\Bbb S}^{\cal J}[\gamma^{-1}{\rm GL}_1^{\cal J}(C)^{\rm gp}] \otimes_{{\Bbb S}^{\cal J}[\gamma^{-1}{\rm GL}_1^{\cal J}(C)]} {\Bbb S}^{\cal J}[\gamma^{-1}{\rm GL}_1^{\cal J}(C)^{\rm rep}]) \ar{d}{\cong} \\ \vspace{10 mm} & C \otimes_{{\Bbb S}^{\cal J}[\gamma^{-1}{\rm GL}_1^{\cal J}(C)]} {\Bbb S}^{\cal J}[\gamma^{-1}{\rm GL}_1^{\cal J}(C)^{\rm rep}]. \end{tikzcd}\end{equation} Since the right-hand vertical map in \eqref{gammapushout} is an equivalence, this composite becomes an equivalence after one suspension $S^1 \odot_A -$ in augmented commutative $A$-algebras.  

Consider now the commutative cube \[\begin{tikzpicture}[baseline= (a).base]
\node[scale=.73] (a) at (0,0){\begin{tikzcd}[row sep = tiny, column sep = tiny]
&
{\Bbb S}^{\cal J}[S^1 \odot_{\alpha^{-1}{\rm GL}_1^{\cal J}(A)} \gamma^{-1}{\rm GL}_1^{\cal J}(C)^{\rm rep}] 
\ar{rr}{}
\ar[]{dd}[near end]{}
& & {\Bbb S}^{\cal J}[S^1 \odot_M K^{\rm rep}] 
\ar{dd}{}
\\
{\Bbb S}^{\cal J}[S^1 \odot_{\alpha^{-1}{\rm GL}_1^{\cal J}(A)} \gamma^{-1}{\rm GL}_1^{\cal J}(C)] 
\ar[crossing over]{rr}[near start]{}
\ar{dd}[swap]{}
\ar{ur}{} 
& & {\Bbb S}^{\cal J}[S^1 \odot_M K]
\ar{ur}{}
\\
&
{\Bbb S}^{\cal J}[S^1 \odot_{{\rm GL}_1^{\cal J}(A)} {\rm GL}_1^{\cal J}(C)^{\rm rep}]
\ar[near start]{rr}{}
& & {\Bbb S}^{\cal J}[S^1 \odot_{M^a} K^{a, {\rm rep}}]
\\
{\Bbb S}^{\cal J}[S^1 \odot_{{\rm GL}_1^{\cal J}(A)} {\rm GL}_1^{\cal J}(C)]
\ar{ur}{\simeq}
\ar{rr}
& & {\Bbb S}^{\cal J}[S^1 \odot_{M^a} K^a]
\ar[crossing over, leftarrow, near start]{uu}{}
\ar{ur}{}
\end{tikzcd}};\end{tikzpicture}\] of ${\Bbb E}_{\infty}$-rings. The front vertical face is cocartesian by the definition of the logification $K^a$, while the back face is cocartesian by \cite[Lemma 4.26]{RSS15} (the cocartesian squares needed for its statement are the defining cocartesian squares of $M^a \simeq M$ and $K^a$), and the composite \eqref{gammaeq} being an equivalence after one suspension implies that the left-hand face is cocartesian after base-change along the induced morphism ${\Bbb S}^{\cal J}[S^1 \odot_{{\rm GL}_1^{\cal J}(A)} {\rm GL}_1^{\cal J}(C)] \to S^1 \odot_A C$. This implies that the square \[\begin{tikzpicture}[baseline= (a).base]
\node[scale=.91] (a) at (0,0){\begin{tikzcd}[row sep = small, column sep = tiny]S^1 \odot_A (C \otimes_{{\Bbb S}^{\cal J}[\gamma^{-1}{\rm GL}_1^{\cal J}(C)]} {\Bbb S}^{\cal J}[\gamma^{-1}{\rm GL}_1^{\cal J}(C)^{\rm rep}]) \ar{r}{\simeq} \ar{d} & S^1 \odot_A (C \otimes_{{\Bbb S}^{\cal J}[{\rm GL}_1^{\cal J}(C)]} {\Bbb S}^{\cal J}[{\rm GL}_1^{\cal J}(C)^{\rm rep}]) \ar{d} \\ S^1 \odot_A (C \otimes_{{\Bbb S}^{\cal J}[K]} {\Bbb S}^{\cal J}[K^{\rm rep}]) \ar{r} & S^1 \odot_A (C \otimes_{{\Bbb S}^{\cal J}[K^a]} {\Bbb S}^{\cal J}[K^{a, {\rm rep}}])\end{tikzcd}};\end{tikzpicture}\] is cocartesian, with the top horizontal map an equivalence. Hence the lower horizontal map is an equivalence, as desired. 
\end{proof}

Let $S^1 \odot_M -$ denote the pointed tensor with $S^1$ in the category $({\cal C}{\cal S}^{\cal J}_{\infty})_{M // M}$ of graded ${\Bbb E}_{\infty}$-spaces pointed at $M$. In the proof of Proposition \ref{prop:lifeismanageable}, we used the following:

\begin{lemma}\label{lem:isoonpizerohj} If $N \in ({\cal C}{\cal S}^{\cal J}_{\infty})_{M // M}$ is a graded ${\Bbb E}_{\infty}$-space pointed at $M$, then the repletion map $S^1 \odot_M N \to S^1 \odot_M N^{\rm rep}$ is a surjection on $\pi_0(-_{h{\cal J}})$. 
\end{lemma}

\begin{proof} By \cite[Lemma 2.12]{Lun21}, there is an equivalence $M \boxtimes W(N) \xrightarrow{\simeq} N^{\rm rep}$, where $W(N)$ is the pullback of $U^{\cal J} \xrightarrow{} M^{\rm gp} \xleftarrow{} N^{\rm gp}$.  Imitating \cite[Proposition 3.1]{RSS18} gives a natural equivalence $M \boxtimes W(N) \simeq M \times (N^{\rm gp}_{h{\cal J}}/M^{\rm gp}_{h{\cal J}})$ over and under $M$. Hence there are equivalences \[S^1 \odot_M N^{\rm rep} \xleftarrow{\simeq} S^1 \odot_M (M \boxtimes W(N)) \simeq S^1 \odot_M (M \times (N^{\rm gp}_{h{\cal J}}/M^{\rm gp}_{h{\cal J}})).\] Observe now that $S^1 \odot_M (M \times (N^{\rm gp}_{h{\cal J}}/M^{\rm gp}_{h{\cal J}})) \simeq M \boxtimes B(N^{\rm gp}_{h{\cal J}}/M^{\rm gp}_{h{\cal J}})$, since $B(-)$ is the suspension functor on grouplike ${\Bbb E}_{\infty}$-spaces. This concludes the proof, as the repletion map is one over $M$, and $B(N^{\rm gp}_{h{\cal J}}/M^{\rm gp}_{h{\cal J}})$ is path-connected. 
\end{proof}

\begin{proof}[Proof of Theorem \ref{thm:logcotangent}] By construction (see Remark \ref{rem:lifeiscomplicated}), the log cotangent complex ${\Bbb L}_{(A, M)}^{\rm rep}$ corresponds to the suspension spectrum \[\Sigma^{\infty} ((A \otimes A) \otimes_{{\Bbb S}^{\cal J}[(M \boxtimes_{} M)^a]} {\Bbb S}^{\cal J}((M \boxtimes_{} M)^{a, {\rm rep}}), ((M \boxtimes_{} M)^{a, {\rm rep}}) \] in ${\rm Sp}({\rm Log}^{\rm rep}_{(A, M^a) // (A, M^a)}) \simeq {\rm Mod}_A$. Under the equivalence of categories of Proposition \ref{prop:augmentedalgrepl} and the equivalences of Proposition \ref{prop:lifeismanageable}, this corresponds to the suspension spectrum \[\Sigma^{\infty}((A \otimes A) \otimes_{{\Bbb S}^{\cal J}[M \boxtimes M]} {\Bbb S}^{\cal J}[(M \boxtimes M)^{\rm rep}])\] in ${\rm Sp}({\rm CAlg}_{A//A})$. By \cite[Section 9]{Lun21}, this $A$-module is equivalent to ${\rm TAQ}(A, M)$.
\end{proof}

For a map $(R, P) \to (A, M)$ of pre-log ring spectra, we define the \emph{relative log cotangent complex} ${\Bbb L}_{(A, M) / (R, P)}^{\rm rep}$ as the cofiber of $A \otimes_R {\Bbb L}_{(R, P)}^{\rm rep} \to {\Bbb L}_{(A, M)}^{\rm rep}$. By Theorem \ref{thm:logcotangent}, the following definition coincides with the one considered in \cite{Sag14}:

\begin{definition}[{\cite[Definition 5.22]{Sag14}}] A map $(R, P) \to (A, M)$ of pre-log ring spectra is \emph{formally log \'etale} if the $A$-module ${\Bbb L}_{(A, M)/ (R, P)}^{\rm rep}$ vanishes. 
\end{definition}

\section{The log Postnikov tower}\label{sec:logpostnikov} Given a log ring spectrum $(R, P)$ with $R$ connective, we now aim to construct a tower \[\cdots \to (\tau_{\le 2}(R), P_{\le 2}) \to (\tau_{\le 1}(R), P_{\le 1}) \to (\pi_0(R), P_{\le 0})\] of log square-zero extensions under $(R, P)$. We will first recall the notion of log square-zero extensions from \cite{Lun22}, and explain how this definition is effectively forced upon us by the analysis in Section \ref{sec:repletetangent}. We then proceed to prove Theorem \ref{thm:squarezeroiff}, reconciling our log square-zero extensions with the strict square-zero extensions used in log geometry. 

\subsection{Log square-zero extensions} Recall from the discussion of Remark \ref{rem:splitsquarezero} that, for a presentably symmetric monoidal, stable $\infty$-category ${\cal C}$ with a commutative algebra object $A$ and an $A$-module $J$, we define the square-zero extension $A \oplus J$ to be the image of $J$ under \[{\rm Mod}_A({\cal C}) \simeq {\rm Sp}({\rm CAlg}({\cal C})_{/A}) \xrightarrow{\Omega^\infty} {\rm CAlg}({\cal C})_{/A}.\] The heuristics of Section \ref{sec:repletetangent} suggest that the correct analogous object $(A, M) \oplus J$ in the context of log ring spectra is the image of $J$ under \begin{equation}\label{compositetologrep}{\rm Mod}_A \simeq {\rm Sp}({\rm Log}^{\rm rep}_{(A, M)//(A, M)}) \xrightarrow{\Omega^{\infty}} {\rm Log}^{\rm rep}_{(A, M)//(A, M)},\end{equation} where the equivalence is Corollary \ref{cor:augmentedalgrepl}. In Remark \ref{rem:splitsquarezero}, we identified this with the square-zero extension $(A \oplus J, M \oplus J)$ considered in \cite[Construction 5.6]{Sag14}. 

\begin{definition} Let $(A, M)$ be a log ring spectrum and let $J$ be an $A$-module. 
\begin{enumerate}
\item The \emph{split log square-zero extension} $(A, M) \oplus J$ is the image of $J$ under \eqref{compositetologrep}. 
\item A \emph{log derivation} of $(A, M)$ with values in $J$ is an augmented $(A, M)$-algebra map $(d, d^\flat) \colon (A, M) \to (A, M) \oplus J$; that is, the space of log derivations is the mapping space ${\rm Der}((A, M), J) := {\rm Map}_{{\rm Log}_{/(A, M)}}((A, M), (A, M) \oplus J)$.
\item A map $(\widetilde{A}, \widetilde{M}) \to (A, M)$ of log ring spectra is \emph{log square-zero} is there exists a log derivation $(d, d^\flat) \colon (A, M) \to (A, M) \oplus J[1]$ which sits in a cartesian diagram of log ring spectra \begin{equation}\label{logsquarezero}\begin{tikzcd}[row sep = small](\widetilde{A}, \widetilde{M}) \ar{r} \ar{d} & (A, M) \ar{d}{(d, d^\flat)} \\ (A, M) \ar{r}{(d_0, d_0^\flat)} & (A, M) \oplus J[1] \end{tikzcd}\end{equation} with $(d_0, d_0^\flat)$ the trivial log derivation. In this situation, we say that $(\widetilde{A}, \widetilde{M})$ is a \emph{log square-zero extension of $(A, M)$ by $J$}, cf.\ \cite[Definition 7.4.1.6]{Lur17}.
\end{enumerate}
\end{definition}

By Remark \ref{rem:splitsquarezero}, the definitions above recover those considered in \cite[Section 11 and 13]{Rog09}, \cite{Sag14}, and \cite[Chapter 4]{Lun22}. The following is a consequence of our construction (or Theorem \ref{thm:logcotangent}):

\begin{lemma}\label{lem:logcotangentrep} The log cotangent complex corepresents log derivations; that is, there is a canonical equivalence \[{\rm Map}_{{\rm Mod}_A}({\Bbb L}_{(A, M)}^{\rm rep}, J) \simeq {\rm Der}((A, M), J)\] for log ring spectra $(A, M)$ and $A$-modules $J$. \qed
\end{lemma}

\subsection{Starting from a strict square-zero extension} We now aim to prove Theorem \ref{thm:squarezeroiff}. With notation as in its statement, we shall first show that, if $(p, p^\flat)$ is strict, then $(p, p^\flat)$ is a log square-zero extension by $J$. Since $p$ is a square-zero extension by $J$, we obtain a cartesian diagram 

\begin{equation}\label{tobelogified}\begin{tikzcd}[row sep = small](\widetilde{R}, \widetilde{P}) \ar{r} \ar{d}{} & (R, \widetilde{P}) \ar{d}{(d, {\rm id})} \\ (R, \widetilde{P}) \ar{r}{(d_0, {\rm id})} & (R \oplus J[1], \widetilde{P})\end{tikzcd}\end{equation} of pre-log ring spectra. Observe that it is essential that we use the inverse image pre-log structure $\widetilde{P} \to P \to \Omega^{\cal J}(R)$, so that the pre-log structure on the square-zero extension $R \oplus J[1]$ is unambiguously defined. The following should be considered a higher variant of \cite[Theorem 4.17]{SSV16}, as we explain in Remark \ref{rem:ssvcomparison}. 

\begin{proposition}\label{prop:strictimpliessquarezero} Let $(p, p^\flat)$ be strict. The logification of the cartesian diagram \eqref{tobelogified} is equivalent to a cartesian diagram of the form \[\begin{tikzcd}[row sep = small](\widetilde{R}, \widetilde{P}) \ar{r} \ar{d} & (R, P) \ar{d}{(d, d^\flat)} \\ (R, P) \ar{r}{(d_0, d_0^\flat)} & (R \oplus J[1], P \oplus J[1])\end{tikzcd}\] of log ring spectra. In particular, $(p, p^\flat)$ is a log square-zero extension.
\end{proposition}

The remainder of this subsection is dedicated to the proof of Proposition \ref{prop:strictimpliessquarezero}. We first give a convenient description of the log structure on $R$ in the presence of a strict map $(\widetilde{R}, \widetilde{P}) \to (R, P)$ with $\widetilde{R} \to R$ square-zero.

\begin{lemma}\label{lem:inverseimagelogi} The diagram \[\begin{tikzcd}[row sep = small]{\rm GL}_1^{\cal J}(\widetilde{R}) \ar{r}{{\rm GL}_1^{\cal J}(p)} \ar{d} & {\rm GL}_1^{\cal J}(R) \ar{d} \\ \widetilde{P} \ar{r} & P\end{tikzcd}\] of commutative ${\cal J}$-space monoids is cocartesian. 
\end{lemma}

\begin{proof} Consider the cartesian squares \[\begin{tikzcd}[row sep = small](p \circ \widetilde{\beta})^{-1}{\rm GL}_1^{\cal J}(R) \ar{d} \ar{r} & p^{-1}{\rm GL}_1^{\cal J}(R) \ar{r} \ar{d} & {\rm GL}_1^{\cal J}(R) \ar{d} \\ \widetilde{P} \ar{r}{\widetilde{\beta}} & \Omega^{\cal J}(\widetilde{R}) \ar{r}{\Omega^{\cal J}(p)} & \Omega^{\cal J}(R) \end{tikzcd}\] of graded ${\Bbb E}_{\infty}$-spaces. We have that $p^{-1}{\rm GL}_1^{\cal J}(R) \simeq {\rm GL}_1^{\cal J}(\widetilde{R})$ by Lemma \ref{lem:unitspullback}, so that $(p \circ \widetilde{\beta})^{-1}{\rm GL}_1^{\cal J}(R) \simeq {\rm GL}_1^{\cal J}(\widetilde{R})$ since $(\widetilde{R}, \widetilde{P})$ is log. Hence the pushout of the diagram \[\widetilde{P} \xleftarrow{} {\rm GL}_1^{\cal J}(\widetilde{R}) \xrightarrow{} {\rm GL}_1^{\cal J}(R)\] of graded ${\Bbb E}_{\infty}$-spaces models the underlying graded ${\Bbb E}_{\infty}$-space of the logification of $(R, \widetilde{P})$, which is equivalent to $P$ by the assumption that $(\widetilde{R}, \widetilde{P}) \to (R, P)$ is strict. This concludes the proof. 
\end{proof}

Let us now identify the logification of $(R \oplus J[1], \widetilde{P})$.

\begin{lemma}\label{lem:logisquarezero} The logification of the pre-log structure \begin{equation}\label{longprelog}\widetilde{P} \xrightarrow{\widetilde{\beta}} \Omega^{\cal J}(\widetilde{R}) \xrightarrow{\Omega^{\cal J}(p)} \Omega^{\cal J}(R) \xrightarrow{d} \Omega^{\cal J}(R \oplus J[1])\end{equation} is equivalent to the split square-zero extension $(R, P) \oplus J[1]$ of Remark \ref{rem:splitsquarezero}.
\end{lemma}

\begin{proof} Pulling back the pre-log structure \eqref{longprelog} along the canonical inclusion morphism ${\rm GL}_1^{\cal J}(R \oplus J[1]) \to \Omega^{\cal J}(R \oplus J[1])$, we obtain a commutative diagram \[\begin{tikzcd}[row sep = small]\widetilde{\beta}^{-1}{\rm GL}_1^{\cal J}(\widetilde{R}) \ar{r}{\simeq} \ar{d} & {\rm GL}_1^{\cal J}(\widetilde{R}) \ar{r}{{\rm GL}_1^{\cal J}(p)} \ar{d} & {\rm GL}_1^{\cal J}(R) \ar{d} \ar{r}{{\rm GL}_1^{\cal J}(d)} & {\rm GL}_1^{\cal J}(R \oplus J[1]) \ar{d} \\ \widetilde{P} \ar{r} & \Omega^{\cal J}(\widetilde{R})\ar{r} & \Omega^{\cal J}(R) \ar{r}{\Omega^{\cal J}(d)} & \Omega^{\cal J}(R \oplus J[1]) \end{tikzcd}\] in which every square is a pullback: The left-hand square is so by definition, while the middle and right-hand squares are pullbacks by Lemma \ref{lem:unitspullback}. Hence the logification of the pre-log structure \eqref{longprelog} is obtained by cobase-change of the map ${\rm GL}_1^{\cal J}(\widetilde{R}) \to \widetilde{P}$ along the composite \[{\rm GL}_1^{\cal J}(\widetilde{R}) \xrightarrow{{\rm GL}_1^{\cal J}(p)} {\rm GL}_1^{\cal J}(R) \xrightarrow{{\rm GL}_1^{\cal J}(d)} {\rm GL}_1^{\cal J}(R \oplus J[1]),\] which is equivalent to the composite \[{\rm GL}_1^{\cal J}(\widetilde{R}) \xrightarrow{{\rm GL}_1^{\cal J}(p)} {\rm GL}_1^{\cal J}(R) \xrightarrow{{\rm GL}_1^{\cal J}(d_0)} {\rm GL}_1^{\cal J}(R \oplus J[1]).\] By Lemma \ref{lem:inverseimagelogi}, we hence obtain that the logification of the pre-log structure \eqref{longprelog} can be described as the pushout of the diagram \[P \xleftarrow{} {\rm GL}_1^{\cal J}(R) \xrightarrow{{\rm GL}_1^{\cal J}(d_0)} {\rm GL}_1^{\cal J}(R \oplus J[1])\] of commutative ${\cal J}$-space monoids. The result follows from the equivalence \begin{equation}\label{sageq}{\rm GL}_1^{\cal J}(R) \boxtimes (1 + J[1]) \xrightarrow{\simeq} {\rm GL}_1^{\cal J}(R \oplus J[1])\end{equation} of \cite[Lemma 5.4]{Sag14}. 
\end{proof}

\begin{remark}\label{rem:ssvcomparison} An analog of Proposition \ref{prop:strictimpliessquarezero} appears in \cite[Theorem 4.17]{SSV16} for integral log rings $(R, P)$. In its proof, the authors consider the projection \[\delta \colon R \xrightarrow{d} R \oplus J[1] \xrightarrow{} J[1]\]  for a derivation $d$, and they define a monoid map \[{\rm GL}_1(R) \to P \oplus J[1], \quad x \mapsto (x, x^{-1}\delta(x)),\] where we have used the log condition on $(R, P)$ to identify $x \in {\rm GL}_1(R)$ with its image in $P$. In the presence of a strict square-zero extension $(\widetilde{R}, \widetilde{P}) \to (R, P)$, this defines a log derivation $(d, d^\flat) \colon (R, P) \to (R, P) \oplus J[1]$ in a natural way, as the proof of \emph{loc.\ cit.}\ shows. The analog of this appears at the very end of the proof of Lemma \ref{lem:logisquarezero}, where we make implicit reference to a homotopy inverse of the equivalence \eqref{sageq}. Indeed, in the context of ordinary rings, this is given by \[{\rm GL}_1(R \oplus J[1]) \xrightarrow{\cong} {\rm GL}_1(R) \oplus (1 + J[1]), \quad (x, j) \mapsto (x, (1, x^{-1}j)).\] 
\end{remark}

We now return to working towards a proof of Proposition \ref{prop:strictimpliessquarezero}. Lemma \ref{lem:logisquarezero} implies that the logification of the square \eqref{tobelogified} has an underlying square of commutative ${\cal J}$-space monoids of the form  \begin{equation}\label{amazingthatthisiscartesian}\begin{tikzcd}[row sep = small]\widetilde{P} \ar{r} \ar{d} & P \ar{d}{d^\flat} \\ P \ar{r}{d_0^\flat} & P \oplus J[1].\end{tikzcd}\end{equation}

\begin{lemma}\label{lem:logicartesian} The square \eqref{amazingthatthisiscartesian} is cartesian. 
\end{lemma}

\begin{proof}  Observe that the projection $P \oplus J[1]  \to P$ is replete by Corollary \ref{cor:splitreplete}. Hence $d^\flat$ is replete by pasting and the assumption that $J$ is connective. It thus suffices to prove that the square is cartesian after replacing $d^\flat$ with its group completion $d^{\flat, {\rm gp}}$. Consider the square \[\begin{tikzcd}[row sep = small]\widetilde{P} \boxtimes_{{\rm GL}_1^{\cal J}(\widetilde{R})} {\rm GL}_1^{\cal J}(\widetilde{R}) \ar{r} \ar{d} & \widetilde{P}^{\rm gp} \boxtimes_{{\rm GL}_1^{\cal J}(\widetilde{R})} {\rm GL}_1^{\cal J}(R) \ar{d} \\ \widetilde{P} \boxtimes_{{\rm GL}_1^{\cal J}(\widetilde{R})} {\rm GL}_1^{\cal J}(R)\ar{r} & \widetilde{P}^{\rm gp} \boxtimes_{{\rm GL}_1^{\cal J}(\widetilde{R})} {\rm GL}_1^{\cal J}(R \oplus J[1])\end{tikzcd}\] of graded ${\Bbb E}_{\infty}$-spaces.  Lemma \ref{lem:inverseimagelogi} shows that the upper right-hand and lower left-hand corners model $P$, while (the proof of) Lemma \ref{lem:logisquarezero} shows that the lower right-hand corner models $P \oplus J[1]$. Hence this square models \eqref{amazingthatthisiscartesian} with $d^\flat$ replaced with $d^{\flat, {\rm gp}}$, and since $J$ is connective, Lemma \ref{lem:bousfieldfriedlander} applies to conclude.
\end{proof}

\begin{proof}[Proof of Proposition \ref{prop:strictimpliessquarezero}] By Lemma \ref{lem:logisquarezero}, the square of logifications is of the desired form. By Lemma \ref{lem:logicartesian}, it is cartesian. 
\end{proof}

\subsection{Starting with a log square-zero extension} We now aim to prove the other direction of Theorem \ref{thm:squarezeroiff}.

\begin{proposition}\label{prop:squarezerostrict} Let $(p, p^\flat) \colon (\widetilde{R}, \widetilde{P}, \widetilde{\alpha}) \to (R, P, \alpha)$ be a square-zero extension of pre-log ring spectra by a connective $R$-module $J$. Then $(p, p^\flat)$ is strict.
\end{proposition}

\begin{proof} Consider the commutative diagram \begin{equation}\label{twopullbacks}\begin{tikzcd}[row sep = small](\alpha \circ p^\flat)^{-1}{\rm GL}_1^{\cal J}(R) \ar{r} \ar{d} & \alpha^{-1}{\rm GL}_1^{\cal J}(R) \ar{r}{\overline{\alpha}} \ar{d}{\overline{i}} & {\rm GL}_1^{\cal J}(R) \ar{d}{i} \\ \widetilde{P} \ar{r}{p^\flat} & P \ar{r}{\alpha} & \Omega^{\cal J}(R)\end{tikzcd}\end{equation} of commutative ${\cal J}$-space monoids, in which both squares are homotopy cartesian. We will show that the left-hand square is homotopy cocartesian, which we claim suffices to conclude. Indeed, if the left-hand square in \eqref{twopullbacks} is cocartesian, then both squares in the diagram \[\begin{tikzcd}[row sep = small](\alpha \circ p^\flat)^{-1}{\rm GL}_1^{\cal J}(R) \ar{r} \ar{d} & \alpha^{-1}{\rm GL}_1^{\cal J}(R) \ar{r}{\overline{\alpha}} \ar{d}{\overline{i}} & {\rm GL}_1^{\cal J}(R) \ar{d} \\ \widetilde{P} \ar{r}{p^\flat} & P \ar{r} & P^a\end{tikzcd}\] are cocartesian, where the right-hand square is the defining cocartesian square for the logification of $(R, P)$. Since the left-hand square is cocartesian, so is the outer square. But the pushout of $\widetilde{P} \xleftarrow{} {(\alpha \circ p^\flat)^{-1}}{\rm GL}_1^{\cal J}(R) \xrightarrow{} {\rm GL}_1^{\cal J}(R)$ is the logification of $\widetilde{P} \xrightarrow{p^\flat} P \xrightarrow{\alpha} \Omega^{\cal J}(R)$, and so $(p, p^\flat)$ is strict.

We now proceed to prove that the left-hand square in \eqref{twopullbacks} is cocartesian. Notice that the map $\overline{i}$ gives rise to a map $\overline{i} \oplus J[1] \colon \alpha^{-1}{\rm GL}_1^{\cal J}(R) \oplus J[1] \to P \oplus J[1]$. Consider the commutative cube \[\begin{tikzcd}[row sep = tiny]
&
{\widetilde P}
\ar{rr}{}
\ar[]{dd}[near end]{}
& & P
\ar{dd}{d^\flat}
\\
(\alpha \circ p^\flat)^{-1}{\rm GL}_1^{\cal J}(R)
\ar[crossing over]{rr}[near start]{}
\ar{dd}[swap]{}
\ar{ur}
& & \alpha^{-1}{\rm GL}_1^{\cal J}(R)
\ar{ur}{\overline{i}}
\\
&
P
\ar[near start]{rr}{d_0^\flat}
& & P \oplus J[1]
\\
\alpha^{-1}{\rm GL}_1^{\cal J}(R)
\ar{ur}{\overline{i}}
\ar{rr}{\overline{i}^*d_0^\flat}
& & \alpha^{-1}{\rm GL}_1^{\cal J}(R) \oplus J[1]
\ar[crossing over, leftarrow, near start]{uu}{\overline{i}^*d^\flat}
\ar[swap]{ur}{\overline{i} \oplus J[1]}
\end{tikzcd}\] of commutative ${\cal J}$-space monoids. Here the back face is the homotopy cartesian square exhibiting $\widetilde{P} \to P$ as part of a log square-zero extension. The right-hand face arises from pulling back the monoid derivation $d^\flat$ along $\overline{\alpha}$ and is as such homotopy cartesian. The left-hand face is the left-hand homotopy cartesian square in the diagram (\ref{twopullbacks}), and this implies that the front face is homotopy cartesian. Moreover, the bottom face is homotopy cocartesian. Since $d^\flat$ is replete (cf.\ proof of Lemma \ref{lem:logicartesian}), Proposition \ref{prop:matherscube} applies to conclude the proof. 
\end{proof}

\begin{proof}[Proof of Theorem \ref{thm:squarezeroiff}] This is Proposition \ref{prop:strictimpliessquarezero} and Proposition \ref{prop:squarezerostrict}. 
\end{proof}

\subsection{The log Postnikov tower} Let $(R, P)$ be a log ring spectrum with $R$ connective. We wish to define a tower of log square-zero extensions under $(R, P)$ that is compatible with the Postnikov tower of $R$. As a consequence of Theorem \ref{thm:squarezeroiff}, there is essentially no choice:

\begin{definition}\label{def:logpostnikov} There $(R, P)$ be a log ring spectrum with $R$ connective. We define the \emph{log Postnikov tower} of $(R, P)$ to be the tower \[\cdots \to (\tau_{\le 2}(R), P_{\le 2}) \to (\tau_{\le 1}(R), P_{\le 1}) \to (\pi_0(R), P_{\le 0}),\] where each $(\tau_{\le n}(R), P_{\le n}) := (\tau_{\le n}(R), P^a)$ is defined as the logification of the inverse image pre-log structure $P \to \Omega^{\cal J}(R) \to \Omega^{\cal J}(\tau_{\le n}(R))$. 
\end{definition}

\begin{corollary} Each map $(\tau_{\le n}(R), P_{\le n}) \to (\tau_{\le (n -1)}(R), P_{\le (n - 1)})$ is a log square-zero extension by $\pi_n(R)[n + 1]$. 
\end{corollary}

\begin{proof} This is true by \cite[Proposition 7.1.3.9]{Lur17} at the level of ${\Bbb E}_{\infty}$-rings. By Theorem \ref{thm:squarezeroiff}, it suffices to prove that the map is strict. Each truncation morphism $(R, P) \to (\tau_{\le n}(R), P_{\le n})$ is strict by definition. Lemma \ref{lem:strictiff} thus implies that \[(\tau_{\le (n - 1)}(R), P_{\le (n - 1)}) \simeq (\tau_{\le (n - 1)}(R), {\rm GL}_1^{\cal J}(\tau_{\le (n - 1)}(R))) \otimes_{(R, {\rm GL}_1^{\cal J}(R))} (R, P).\] Using that the map $(R, {\rm GL}_1^{\cal J}(R)) \to (\tau_{\le (n - 1)}(R), {\rm GL}_1^{\cal J}(\tau_{\le (n - 1)}(R)))$ factors over $(\tau_{\le n}(R), P_{\le n})$ and applying Lemma \ref{lem:strictiff} again to the map $(R, P) \to (\tau_{\le n}(R), P_{\le n})$, we identify the above relative coproduct with \[(\tau_{\le (n - 1)}(R), {\rm GL}_1^{\cal J}(\tau_{\le (n - 1)}(R))) \otimes_{(\tau_{\le n}(R), {\rm GL}_1^{\cal J}(\tau_{\le n}(R)))} (\tau_{\le n}(R), P_{\le n}),\] which concludes the proof by a third application to Lemma \ref{lem:strictiff}.   
\end{proof}

\begin{remark} By Lemma \ref{lem:inverseimagelogi}, each of the graded ${\Bbb E}_{\infty}$-spaces $P_{\le n}$ admit the explicit description $P \boxtimes_{{\rm GL}_1^{\cal J}(R)} {\rm GL}_1^{\cal J}(\tau_{\le n}(R))$. In particular, the underlying tower of graded ${\Bbb E}_{\infty}$-spaces is cobase-changed from the Postnikov tower of $R$ in the sense that it is of the form \[\cdots  \! \to \!P \boxtimes_{{\rm GL}_1^{\cal J}\!\!(R)} {\rm GL}_1^{\cal J}(\tau_{\le 2}(R)) \to P \boxtimes_{{\rm GL}_1^{\cal J}\!\!(R)} {\rm GL}_1^{\cal J}(\tau_{\le 1}(R)) \to P \boxtimes_{{\rm GL}_1^{\cal J}\!\!(R)} {\rm GL}_1^{\cal J}(\pi_0(R)).\] We thus see that the bottom-most stage of the tower is not necessarily discrete, in contrast to the Postnikov tower of a connective ${\Bbb E}_{\infty}$-ring. 
\end{remark}

\begin{proposition}\label{prop:logpostnikovconverge} Let $(R, P)$ be a log ring spectrum with $R$ connective. The limit of the log Postnikov tower is equivalent to $(R, P)$.
\end{proposition}

\begin{proof}  As Postnikov towers of ring spectra converge, it suffices to show that the limit of the diagram of underlying commutative ${\cal J}$-space monoids is $P$. Lemma \ref{lem:bousfieldfriedlander} shows that both squares in the diagram \[\begin{tikzcd}[row sep = small]P \boxtimes_{{\rm GL}_1^{\cal J}(R)} {\rm GL}_1^{\cal J}(R) \ar{d} \ar{r} & P^{\rm gp} \boxtimes_{{\rm GL}_1^{\cal J}(R)} {\rm GL}_1^{\cal J}(R) \ar{d} \\ P \boxtimes_{{\rm GL}_1^{\cal J}(R)} {\rm GL}_1^{\cal J}(\tau_{\le n}(R)) \ar{r} \ar{d} & P^{\rm gp} \boxtimes_{{\rm GL}_1^{\cal J}(R)} {\rm GL}_1^{\cal J}(\tau_{\le n}(R)) \ar{d} \\ P \boxtimes_{{\rm GL}_1^{\cal J}(R)} {\rm GL}_1^{\cal J}(\pi_0(R)) \ar{r} & P^{{\rm gp}} \boxtimes_{{\rm GL}_1^{\cal J}(R)} {\rm GL}_1^{\cal J}(\pi_0(R))\end{tikzcd}\] are cartesian. As limits commute with limits, the limit in question is equivalent to \[(P \boxtimes_{{\rm GL}_1^{\cal J}(R)} {\rm GL}_1^{\cal J}(\pi_0(R))) \times_{(P^{\rm gp} \boxtimes_{{\rm GL}_1^{\cal J}(R)} {\rm GL}_1^{\cal J}(\pi_0(R)))} {\rm lim}(P^{\rm gp} \boxtimes_{{\rm GL}_1^{\cal J}(R)} {\rm GL}_1^{\cal J}(\tau_{\le n}(R))).\] It thus suffices to prove that ${\rm lim}(P^{\rm gp} \boxtimes_{{\rm GL}_1^{\cal J}(R)} {\rm GL}_1^{\cal J}(\tau_{\le n}(R))) \simeq P^{\rm gp}$, as the outer rectangle above is cartesian. The result now follows from Corollary \ref{cor:finally} below. \end{proof}

In the part of the proof of Proposition \ref{prop:logpostnikovconverge} provided above, we reduced to checking that the tower \[\cdots  \to P^{\rm gp} \boxtimes_{{\rm GL}_1^{\cal J}\!\!(R)} {\rm GL}_1^{\cal J}(\tau_{\le 1}(R)) \to P^{\rm gp} \boxtimes_{{\rm GL}_1^{\cal J}\!\!(R)} {\rm GL}_1^{\cal J}(\pi_0(R))\] converges to $P^{\rm gp}$. For this, we shall exploit the relationship between grouplike commutative ${\cal J}$-space monoids and the category of connective spectra over the sphere of \cite[Theorems 1.6, 5.10]{Sag16}, therein modelled by an appropriate slice of the stable model structure on Segal $\Gamma$-spaces. The slice is over a certain Segal $\Gamma$-space $b{\cal J}$, which models the sphere spectrum. These are connected by a chain of Quillen equivalences and the resulting derived functor is denoted $\gamma$, which gives an equivalence relating their underlying $\infty$-categories. As a first step, we need:

\begin{lemma}\label{lem:truncationnconn} The map $\gamma({\rm GL}_1^{\cal J}(R)) \to \gamma({\rm GL}_1^{\cal J}(\tau_{\le n}(R)))$ is $n$-connected. 
\end{lemma}

We warn the reader that this is not as obvious as it might seem at first sight. While the analogous statement is clear for the spectrum of units ${\rm gl}_1(R)$, the spectrum of \emph{graded units} $\gamma({\rm GL}_1^{\cal J}(R))$ exhibits different behaviour, even when the ring spectrum $R$ is connective. 

\begin{proof}[Proof of Lemma \ref{lem:truncationnconn}] Consider first the case where $R = \tau_{\le (n + 1)}(R)$, so that the truncation map $\tau_{\le (n + 1)}(R) \to \tau_{\le n}(R)$ is a square-zero extension. The cartesian square exhibiting this map as such gives a cartesian square of commutative ${\cal J}$-space monoids upon applying ${\rm GL}_1^{\cal J}(-)$. Hence there is a cartesian square \[\begin{tikzcd}[row sep = small]\gamma({\rm GL}_1^{\cal J}(\tau_{\le (n + 1)}(R))) \ar{r} \ar{d} & \gamma({\rm GL}_1^{\cal J}(\tau_{\le n}(R))) \ar{d} \\ \gamma({\rm GL}_1^{\cal J}(\tau_{\le n}(R))) \ar{r} & \gamma({\rm GL}_1^{\cal J}(\tau_{\le n}(R) \oplus \pi_{n + 1}(R)[n + 2])) \end{tikzcd}\] of Segal $\Gamma$-spaces over $b{\cal J}$. By \cite[Lemma 5.4, Proposition 5.5, Lemma 7.2]{Sag14}, the lower right-hand corner splits as $\gamma({\rm GL}_1^{\cal J}(\tau_{\le n}(R))) \times \gamma((1 \oplus \pi_{n + 1}(R)[n + 2]))$, and the term $\gamma((1 \oplus \pi_{n + 1}(R)[n + 2])$ models the connective spectrum $\pi_{n + 1}(R)[n + 2]$. Hence the fiber of the map of spectra $\gamma({\rm GL}_1^{\cal J}(\tau_{\le (n + 1)}(R))) \to \gamma({\rm GL}_1^{\cal J}(\tau_{\le n}(R)))$ is $n$-connected, as desired.

 The result follows as $\gamma({\rm GL}_1^{\cal J}(R))$ is the limit of the tower \[\cdots \to \gamma({\rm GL}_1^{\cal J}(\tau_{\le (n + 1)}(R))) \to \gamma({\rm GL}_1^{\cal J}(\tau_{\le n}(R)))\] where each map from $\gamma({\rm GL}_1^{\cal J}(\tau_{\le (n + k)}(R)))$ is $n$-connected, cf.\ \cite[Proof of Lemma 3.3]{BMS19}. 
\end{proof}

\begin{corollary}\label{cor:finally} Each map $\gamma(P^{\rm gp}) \to \gamma(P^{\rm gp} \boxtimes_{{\rm GL}_1^{\cal J}(R)} {\rm GL}_1^{\cal J}(\tau_{\le n}(R)))$ is $n$-connected. In particular, the canonical map \[\gamma(P^{\rm gp}) \to {\rm lim}(\gamma(P^{\rm gp} \boxtimes_{{\rm GL}_1^{\cal J}(R)} {\rm GL}_1^{\cal J}(\tau_{\le n}(R))))\] is an equivalence. Since the functor $\gamma$ is an equivalence, the canonical map \[P^{\rm gp} \to {\rm lim}(P^{\rm gp} \boxtimes_{{\rm GL}_1^{\cal J}(R)} {\rm GL}_1^{\cal J}(\tau_{\le n}(R)))\] is also an equivalence.  
\end{corollary}

\begin{proof} This follows from Lemma \ref{lem:truncationnconn} and the cocartesian squares \[\begin{tikzcd}[row sep = small]\gamma({\rm GL}_1^{\cal J}(R)) \ar{r} \ar{d} & \gamma({\rm GL}_1^{\cal J}(\tau_{\le n}(R))) \ar{d} \\ \gamma(P^{\rm gp}) \ar{r} & \gamma(P^{\rm gp} \boxtimes_{{\rm GL}_1^{\cal J}(R)} {\rm GL}_1^{\cal J}(\tau_{\le n}(R)))\end{tikzcd}\] of connective spectra over the sphere. It is, in particular, a cocartesian square of spectra, so the lower map is $n$-connected because the upper one is. 
\end{proof}

\section{Log \'etale rigidity}\label{sec:logetalerigidity} 

All ring spectra in this section are assumed to be connective. This in particular means that all pre-log ring spectra $(R, P)$ are \emph{connective} in the sense of \cite[Definition 7.7]{Lun21}: Both $R$ and ${\Bbb S}^{\cal J}[P]$ are assumed to be connective.

 Recall that, for a log ring spectrum $(R, P)$, the log Postnikov tower of Definition \ref{def:logpostnikov} gives a tower \[(R, P) \to \cdots \to (\tau_{\le 2}(R), P_{\le 2}) \to (\tau_{\le 1}(R), P_{\le 1}) \to (\pi_0(R), P_{\le 0})\] of log square-zero extensions of log ring spectra. 

\begin{lemma}\label{lem:mappingspacesq0} Let $(\widetilde{R}, \widetilde{P}) \to (\widetilde{A}, \widetilde{M})$ be a formally log \'etale morphism of log ring spectra, and assume that $(\widetilde{B}, \widetilde{N}) \to (B, N)$ is a log square-zero extension of $(\widetilde{R}, \widetilde{P})$-algebras. Then the map \[{\rm Map}_{{\rm Log}_{(\widetilde{R}, \widetilde{P})/}}((\widetilde{A}, \widetilde{M}), (\widetilde{B}, \widetilde{N})) \xrightarrow{} {\rm Map}_{{\rm Log}_{(\widetilde{R}, \widetilde{P})/}}((\widetilde{A}, \widetilde{M}), (B, N))\] induced by $(\widetilde{B}, \widetilde{N}) \to (B, N)$ is an equivalence.
\end{lemma}

\begin{proof} Using the pullback square \eqref{logsquarezero} for $(\widetilde{B}, \widetilde{N}) \to (B, N)$, we can check the statement with $(\widetilde{B}, \widetilde{N}) \to (B, N)$ replaced by $(B, N) \to (B \oplus J[1], N \oplus J[1])$, and thus with $(\widetilde{B}, \widetilde{N}) \to (B, N)$ replaced with $(B \oplus J[1], N \oplus J[1]) \to (B, N)$. We have reduced to checking that \[{\rm Map}_{{\rm Log}_{(\widetilde{R}, \widetilde{P})/}}((\widetilde{A}, \widetilde{M}), (B \oplus J[1], N \oplus J[1])) \xrightarrow{} {\rm Map}_{{\rm Log}_{(\widetilde{R}, \widetilde{P})/}}((\widetilde{A}, \widetilde{M}), (B, N))\] is an equivalence. The fiber over any $(\widetilde{R}, \widetilde{P})$-algebra map $(\widetilde{A}, \widetilde{M}) \to (B, N)$ is the space of augmented $(\widetilde{R}, \widetilde{P})$-algebra maps \[{\rm Map}_{{\rm Log}_{(\widetilde{R}, \widetilde{P})//(B, N)}}((\widetilde{A}, \widetilde{M}), (B \oplus J[1], N \oplus J[1])).\] By Lemma \ref{lem:logcotangentrep}, this is equivalent to ${\rm Map}_{{\rm Mod}_B}(B \otimes_{\widetilde{A}} {\Bbb L}_{(\widetilde{A}, \widetilde{M})/(\widetilde{R}, \widetilde{P})}^{\rm rep}, J[1])$, which is contractible by assumption.
\end{proof}

\begin{corollary}\label{cor:truncationalonglogpostnikov} Let $(R, P) \to (A, M)$ be a formally log \'etale morphism and let $(B, N)$ be an $(R, P)$-algebra. Then map \[{\rm Map}_{{\rm Log}_{(R, P)/}}((A, M), (B, N)) \to {\rm Map}_{{\rm Log}_{(R, P)/}}((A, M), (\pi_0(B), N_{\le 0}))\] induced by the log Postnikov truncation $(B, N) \to (\pi_0(B), N_{\le 0})$ is an equivalence. 
\end{corollary}

\begin{proof} By Proposition \ref{prop:logpostnikovconverge}, ${\rm Map}_{{\rm Log}_{(R, P)/}}((A, M), (B, N))$ is the limit of the mapping spaces  ${\rm Map}_{{\rm Log}_{(R, P)/}}((B, N), (\tau_{\le n}(B), N_{\le n}))$ along the maps in the log Postnikov tower, so that Lemma \ref{lem:mappingspacesq0} applies to conclude. 
\end{proof}

\begin{remark} Before stating the next result, we remind the reader that Theorem \ref{thm:chartedlogetale} is \emph{not} concerned with the functor \begin{equation}\label{naivetruncation}{\rm Log}_{(R, P)/} \to {\rm Log}_{(\pi_0(R), P_{\le 0})/}, \quad (A, M) \mapsto (\pi_0(A), M_{\le 0}),\end{equation} but rather the base-change functor along $(R, P) \to (\pi_0(R), P_{\le 0})$. These perspectives coincide with the notion of \'etaleness considered in \cite[Section 7.5]{Lur17}, as \'etale morphisms $R \to A$ therein are required to be \emph{flat}; in particular, there is an isomorphism $\pi_0(A) \otimes_{\pi_0(R)} \pi_*(R) \xrightarrow{\cong} \pi_*(A)$.  This condition is not relevant for us, and the functor \eqref{naivetruncation} does not restrict to one of formally log \'etale objects. Nonetheless, the following truncation property for formally log \'etale morphisms will prove convenient for us:
\end{remark} 

\begin{corollary}\label{cor:truncationalltheway} Let $(R, P) \to (A, M)$ be a formally log \'etale morphism. The map \[{\rm Map}_{{\rm Log}_{(R, P)/}}((A, M), (B, N)) \xrightarrow{} {\rm Map}_{{\rm Log}_{(\pi_0(R), P_{\le 0})/}}((\pi_0(A), M_{\le 0}), (\pi_0(B), N_{\le 0}))\] induced by the log Postnikov truncations is an equivalence. 
\end{corollary}

\begin{proof} By Corollary \ref{cor:truncationalonglogpostnikov}, we can replace ${\rm Map}_{{\rm Log}_{(R, P)/}}((A, M), (B, N))$ with the mapping space ${\rm Map}_{{\rm Log}_{(R, P)/}}((A, M), (\pi_0(B), N_{\le 0}))$. Since ${\rm PreLog}$ is a full subcategory of ${\rm Log}$, we may replace ${\rm Log}$ by ${\rm PreLog}$. Consider the commutative cube \[\begin{tikzpicture}[baseline= (a).base]
\node[scale=.59] (a) at (0,0){\begin{tikzcd}[column sep = tiny, row sep = tiny]
&
{\rm Map}_{(R, P)/}((A, M), (\pi_0(B), N_{\le 0}))
\ar{rr}{}
\ar[]{dd}[near end]{}
& & {\rm Map}_{{P/}}(M, N_{\le 0})
\ar{dd}{}
\\
{\rm Map}_{{(\pi_0(R), P)/}}((\pi_0(A), M), (\pi_0(B), N_{\le 0}))
\ar[crossing over]{rr}[near start]{}
\ar{dd}[swap]{}
\ar{ur}
& & {\rm Map}_{{P/}}(M, N_{\le 0})
\ar{ur}{=}
\\
&
{\rm Map}_{{R/}}(A, \pi_0(B))
\ar[near start]{rr}{}
& & {\rm Map}_{{{\Bbb S}^{\cal J}[P]/}}({\Bbb S}^{\cal J}[M], \pi_0(B))
\\
{\rm Map}_{{\pi_0(R)/}}(\pi_0(A), \pi_0(B))
\ar{ur}{\simeq}
\ar{rr}
& & {\rm Map}_{{{\Bbb S}^{\cal J}[P]/}}({\Bbb S}^{\cal J}[M], \pi_0(B))
\ar[crossing over, leftarrow, near start]{uu}{}
\ar[]{ur}{=}
\end{tikzcd}};\end{tikzpicture}\] induced by the truncation maps $(-, -) \to (\pi_0(-), -)$ of \emph{pre-}log ring spectra (so that e.g.\ $\pi_0(A)$ carries the inverse image pre-log structure $M \to \Omega^{\cal J}(A) \to \Omega^{\cal J}(\pi_0(A))$), where we have used short-hands of the form ${\rm Map}_{{\cal C}_{x/}} := {\rm Map}_{x/}$. Here the front and back vertical faces are the defining cartesian squares for mapping spaces in ${\rm PreLog}$, and it is clear that the right-hand face is cartesian. Hence the left-hand face is cartesian, and we conclude that both maps \[\begin{tikzcd}[row sep = small]{\rm Map}_{{\rm PreLog}_{(\pi_0(R), P_{\le 0})/}}((\pi_0(A), M_{\le 0}), (\pi_0(B), N_{\le 0})) \ar{d}{\simeq} \\ {\rm Map}_{{\rm PreLog}_{(\pi_0(R), P)/}}((\pi_0(A), M), (\pi_0(B), N_{\le 0})) \ar{d}{\simeq} \\ {\rm Map}_{{\rm PreLog}_{(R, P)/}}((A, M), (\pi_0(B), N_{\le 0}))\end{tikzcd}\] are equivalences, where in the first equivalence we have used that logification is left adjoint to the forgetful functor (recall that $M_{\le 0}$ is the logification of $M \to \Omega^{\cal J}(A) \to \Omega^{\cal J}(\pi_0(A))$).
\end{proof}

\begin{lemma}\label{lem:fullyfaithful} Let $(\widetilde{R}, \widetilde{P}) \to (R, P)$ be a strict morphism of log ring spectra whose underlying map of ${\Bbb E}_{\infty}$-rings is $0$-connected. The base-change functor \[(-, -) \otimes_{(\widetilde{R}, \widetilde{P})} (R, P) \colon {\rm Log}_{(\widetilde{R}, \widetilde{P})/} \xrightarrow{} {\rm Log}_{(R, P)/}\] restricts to a fully faithful functor on formally log \'etale objects. 
\end{lemma}

\begin{proof} Let $(\widetilde{R}, \widetilde{P}) \to (\widetilde{A}, \widetilde{M})$ be a formally log \'etale morphism of log ring spectra, let $(\widetilde{B}, \widetilde{N})$ be an $(\widetilde{R}, \widetilde{P})$-algebra, and let $(A, M)$ and $(B, N)$ denote the respective cobase-changes along $(\widetilde{R}, \widetilde{P}) \to (R, P)$. We need to show that the top horizontal arrow in the diagram \[\begin{tikzpicture}[baseline= (a).base]
\node[scale=.83] (a) at (0,0){\begin{tikzcd}[column sep = tiny, row sep = small]{\rm Map}_{{\rm Log}_{(\widetilde{R}, \widetilde{P})/}}((\widetilde{A}, \widetilde{M}), (\widetilde{B}, \widetilde{N})) \ar{r} \ar{d}{\simeq} & {\rm Map}_{{\rm Log}_{(R, P)/}}((A, M), (B, N)) \ar{d}{\simeq} \\ {\rm Map}_{{\rm Log}_{(\pi_0(\widetilde{R}), \widetilde{P}_{\le 0})/}}((\pi_0(\widetilde{A}), \widetilde{M}_{\le 0}), (\pi_0(\widetilde{B}), \widetilde{N}_{\le 0})) \ar{r} & {\rm Map}_{{\rm Log}_{(\pi_0(R), P_{\le 0})/}}((\pi_0(A), M_{\le 0}), (\pi_0(B), N_{\le 0}))\end{tikzcd}};\end{tikzpicture}\] is an equivalence. The vertical arrows are equivalences by Corollary \ref{cor:truncationalltheway}. The $0$-connectedness hypothesis implies that $\pi_0(\widetilde{R}) \cong \pi_0(R)$, and consequently we obtain $\pi_0(\widetilde{A}) \cong \pi_0(A)$ and $\pi_0(\widetilde{B}) \cong \pi_0(B)$ by our standing connectivity hypotheses. We shall now argue that there are equivalences \begin{equation}\label{bottomlogpostnikov}\widetilde{P}_{\le 0} \xrightarrow{\simeq} P_{\le 0}, \quad \widetilde{M}_{\le 0} \xrightarrow{\simeq} M_{\le 0}, \quad \text{and} \quad \widetilde{N}_{\le 0} \xrightarrow{\simeq} N_{\le 0},\end{equation} which will conclude the proof.  For the first of these, we observe that there is an equivalence \[(\pi_0(\widetilde{R}), \widetilde{P}_{\le 0}) \simeq (\widetilde{R}, \widetilde{P}) \otimes_{(\widetilde{R}, {\rm GL}_1^{\cal J}(\widetilde{R}))} (\pi_0(\widetilde{R}), {\rm GL}_1^{\cal J}(\pi_0(\widetilde{R})))\] of log ring spectra by Lemma \ref{lem:strictiff}, the right-hand side of which is isomorphic to  the coproduct $(\widetilde{R}, \widetilde{P}) \otimes_{(\widetilde{R}, {\rm GL}_1^{\cal J}(\widetilde{R}))} (\pi_0(R), {\rm GL}_1^{\cal J}(\pi_0(R)))$. This, in turn, is isomorphic to \[((\widetilde{R}, \widetilde{P}) \otimes_{(\widetilde{R}, {\rm GL}_1^{\cal J}(\widetilde{R}))} (R, {\rm GL}_1^{\cal J}(R))) \otimes_{(R, {\rm GL}_1^{\cal J}(R))} (\pi_0(R), {\rm GL}_1^{\cal J}(\pi_0(R))).\] By the assumption that $(\widetilde{R}, \widetilde{P}) \to (R, P)$ is strict and Lemma \ref{lem:strictiff}, we obtain that this is equivalent to  \[(R, P) \otimes_{(R, {\rm GL}_1^{\cal J}(R))} (\pi_0(R), {\rm GL}_1^{\cal J}(\pi_0(R))) \simeq (\pi_0(R), P_{\le 0}),\] where we have once again used Lemma \ref{lem:strictiff} for the last equivalence. As strict morphisms are stable under cobase-change by \cite[Lemma 5.5]{SSV16}, the maps $(\widetilde{A}, \widetilde{M}) \to (A, M)$ and $(\widetilde{B}, \widetilde{N}) \to (B, N)$ are also strict, so that the two remaining equivalences of \eqref{bottomlogpostnikov} are obtained by the same argument.
\end{proof}

\begin{proposition}\label{prop:esssurjective} Let $(R, P) \to (A, M)$ be a formally log \'etale map of log ring spectra, and let $(\widetilde{R}, \widetilde{P}) \to (R, P)$ be a square-zero extension by a $0$-connected $R$-module $J$. Then there exists an essentially unique log square-zero extension $(\widetilde{A}, \widetilde{M}) \to (A, M)$ by $A \otimes_R J$ together with a formally log \'etale morphism $(\widetilde{R}, \widetilde{P}) \to (\widetilde{A}, \widetilde{M})$ which exhibits $(\widetilde{A}, \widetilde{M})$ as a deformation of $(A, M)$ to $(\widetilde{R}, \widetilde{P})$; that is, the square \begin{equation}\label{thisiscocartesian}\begin{tikzcd}[row sep = small](\widetilde{R}, \widetilde{P}) \ar{r} \ar{d} & (\widetilde{A}, \widetilde{M}) \ar{d} \\ (R, P) \ar{r} & (A, M)\end{tikzcd}\end{equation} of log ring spectra is cocartesian. 
\end{proposition}

\begin{proof} We shall construct a commutative cube \begin{equation}\label{inductioncube}\begin{tikzpicture}[baseline= (a).base]
\node[scale=.85] (a) at (0,0){\begin{tikzcd}[column sep = small, row sep = tiny]
&
(R, P)
\ar{rr}{}
\ar[]{dd}[near end]{}
& & (A, M)
\ar{dd}{}
\\
(\widetilde{R}, \widetilde{P})
\ar[crossing over]{rr}[near start]{}
\ar{dd}[swap]{(p, p^\flat)}
\ar{ur}{}
& & (\widetilde{A}, \widetilde{M})
\ar{ur}{}
\\
&
(R \oplus J[1], (P + J[1])^{\cal J})
\ar[near start]{rr}{}
& & (A \oplus (A \otimes_R J)[1], M \oplus (A \otimes_R J)[1])
\\
(R, P)
\ar[]{ur}{(d, d^\flat)}
\ar[swap]{rr}{(f, f^\flat)}
& & (A, M)
\ar[crossing over, leftarrow, near start]{uu}{(q, q^\flat)}
\ar[dashed]{ur}
\end{tikzcd}};\end{tikzpicture}\end{equation} of log ring spectra. Here $(d, d^\flat)$ is a log derivation which exhibits $(p, p^\flat)$ as a square-zero extension. The vertical maps of the back face are the trivial derivations.

The space of dashed arrows making the bottom face of the cube (\ref{inductioncube}) commute is an element of \[{\rm Map}_{(R, P)//(A, M)}((A, M), (A \oplus (A \otimes_R J)[1], M \oplus (A \otimes_R J)[1]).\] By Lemma \ref{lem:logcotangentrep}, this is equivalent to \[{\rm Map}_{{\rm Mod}_A}({\Bbb L}_{(A, M) / (R, P)}^{\rm rep}, A \otimes_R J[1]),\] which is contractible by assumption. This gives an essentially unique dashed arrow making the bottom face commute.
We define $(\widetilde{A}, \widetilde{M})$ to be the associated infinitesimal extension, and we claim that it has the properties predicted by the proposition.  The square \eqref{thisiscocartesian} is cocartesian at the level of ${\Bbb E}_{\infty}$-rings by \cite[Theorem 3.25]{dagiv}. Since the vertical morphisms are strict, it follows that \eqref{thisiscocartesian} is also cocartesian at the level of commutative ${\cal J}$-space monoids. 

It only remains to show that the resulting map $(\widetilde{f}, \widetilde{f}^\flat) \colon (\widetilde{R}, \widetilde{P}) \to (\widetilde{A}, \widetilde{M})$ is formally log \'etale. We claim that there are equivalences \[A \otimes_{{\widetilde{A}}} {\Bbb L}_{(\widetilde{A}, \widetilde{M}) / (\widetilde{R}, \widetilde{P})}^{\rm rep} \xrightarrow{\simeq} {\Bbb L}_{(A, \widetilde{M}) / (R, \widetilde{P})}^{\rm rep} \xrightarrow{\simeq} {\Bbb L}_{(A, M) / (R, P)}^{\rm rep} \simeq 0\] of $A$-modules. For this, we use Theorem \ref{thm:logcotangent} to identify the cotangent complexes with the relevant log ${\rm TAQ}$-terms. The first equivalence follows from the homotopy cocartesian square \eqref{thisiscocartesian} and flat base change for log ${\rm TAQ}$ \cite[Proposition 11.29]{Rog09}. The second follows from logification invariance of log {\rm TAQ} \cite[Corollary 11.23]{Rog09} and strictness of $(q, q^\flat)$. Since $\widetilde{A} \to A$ is a square-zero extension by a $0$-connected module, the extension of scalars functor $A {\otimes_{\widetilde{A}}} -$ is conservative; see \cite[Lemma 3.3]{PVK22}. This concludes the proof.
\end{proof}

\begin{proof}[Proof of Theorem \ref{thm:logetalebase}] The base-change functor is essentially surjective by Proposition \ref{prop:esssurjective}, while it is fully faithful by Lemma \ref{lem:fullyfaithful}. 
\end{proof}

\subsection{Charted log \'etale morphisms} In the following definition, we once again make reference to the functor $\gamma$ of \cite[Section 3]{Sag16}, which associates to a commutative ${\cal J}$-space monoid $M$ a certain augmented Segal $\Gamma$-space $\gamma(M)$. Its underlying $\Gamma$-space is very special, and hence uniquely determines a connective spectrum. We invite the reader to think of $\gamma(M)$ as the connective spectrum associated to the ``underlying grouplike ${\Bbb E}_{\infty}$-space'' of the group completion $M^{\rm gp}$. 

\begin{definition}\label{def:chartedlogetale} A morphism $(R, P^a) \to (A, M^a)$ of log ring spectra is \emph{charted log \'etale} if it arises as the logification of a morphism $(R, P) \to (A, M)$ of pre-log ring spectra such that 
\begin{enumerate}
\item the induced morphism $R \otimes_{{\Bbb S}^{\cal J}[P]} {\Bbb S}^{\cal J}[M] \to A$ is \'etale; and
\item the $A$-module $A \otimes (\gamma(M)/\gamma(P))$ vanishes. 
\end{enumerate}
We shall refer to such a map $(R, P) \to (A, M)$ of pre-log ring spectra as a \emph{log \'etale chart} for the charted log \'etale morphism $(R, P^a) \to (A, M^a)$.
\end{definition} 

The notion of charted log \'etale morphisms appeared previously in \cite[Definition 8.4]{Lun21} under the terminology ``log \'etale''. 

\begin{lemma}\label{lem:logetalechartformallylogetale} Charted log \'etale morphisms are formally log \'etale.
\end{lemma}

\begin{proof} By \cite[Lemma 8.7]{Lun21}, the $A$-module ${\rm TAQ}((A, M) / (R, P))$ vanishes, so that Theorem \ref{thm:logcotangent} applies to conclude.  
\end{proof}

\begin{remark} The precise relationship between the notions ``charted log \'etale'' in the sense of Definition \ref{def:chartedlogetale} and formally log \'etale in the sense of the vanishing of the log cotangent complex is closely related to the discussion of \cite[Remark 11.26]{Rog09}. By \cite[Lemma 11.25]{Rog09} and \cite[Lemma 6.2]{Sag14}, there is a cofiber sequence of $A$-modules \[A \otimes (\gamma(M)/\gamma(P)) \to {\Bbb L}^{\rm rep}_{(A, M) / (R, P)} \to {\Bbb L}_{A / R \otimes_{{\Bbb S}^{\cal J}[P]} {\Bbb S}^{\cal J}[M]},\] where we have used Theorem \ref{thm:logcotangent} to identify the middle term. Definition \ref{def:chartedlogetale} asks that the two outer terms vanish, together with a finiteness hypothesis on $R \otimes_{{\Bbb S}^{\cal J}[P]} {\Bbb S}^{\cal J}[M] \to A$. By definition, formal log \'etaleness only asks for the middle term to vanish. Under further finiteness hypotheses (notice that we have not imposed any on $P \to M$ itself), the analogous notions coincide in classical log geometry \cite[(3.5)]{Kat89}. 
\end{remark}

Let $(R, P^a) \to (A, M^a)$ be a charted log \'etale morphism. Despite our very suggestive notation, there may be several distinct pre-log ring spectra $(R, P)$ that participate in log \'etale charts $(R, P) \to (A, M)$ for $(R, P^a) \to (A, M^a)$. The following definition gets rid of this ambiguity:

\begin{definition}\label{def:chartedlogetaleatp} Let $(R, P)$ be a fixed pre-log ring spectrum with logification $(R, P^a)$. We define the category ${\rm Log}_{(R, P^a)/}^{{\rm chl\acute{e}t}, P}$ of \emph{charted log \'etale $(R, P^a)$-algebras at $P$} to be the category spanned by those charted log \'etale maps $(R, P^a) \to (A, M^a)$ that arise as the logification of log \'etale charts of the form $(R, P) \to (A, M)$. 
\end{definition}

In the following remark, we discuss some technical aspects of the above definition.

\begin{remark}\label{rem:chartedlogetale} It may seem that we have added an unreasonable amount of restrictions on the category of ``log \'etale'' morphisms that we consider. We will now argue that this is a rather natural choice from the perspective of classical log geometry. 

Log rings are not to log schemes what rings are to schemes; instead, log rings serve as \emph{charts} of log structures. For this reason, we should think of log ring spectra as charts for (the for now hypothetical notion of) \emph{spectral log schemes}. This is the basis for the intuition that we now present. 

Definition \ref{def:chartedlogetale} is an immediate adaptation of \cite[Theorem 3.5]{Kat89}, which states that a map of sufficiently nice log schemes is log \'etale precisely when it \'etale locally admits charts satisfying natural analogs of Definition \ref{def:chartedlogetale}. Our choice to further fix $P$ is justified by a lemma of Nizio{\l} \cite[Lemma 2.8]{Niz08}, which is a generalization of \cite[Lemma 3.1.6]{Kat}. It states that, for a log \'etale morphism of log schemes with a fixed chart on the target, one can always extend this to a chart for the log \'etale morphism that satisfies the obvious analog of Definition \ref{def:chartedlogetale}. 

Let us stress that we have \emph{not} proved an analog of this ``chart extension'' lemma in our context, but we merely use it as justification for our choice of subcategory of ``log \'etale objects.'' In particular, we are currently unable to prove that the notion of charted log \'etale morphisms of Definition \ref{def:chartedlogetale} is closed under composition. Nonetheless, we may of course compose morphisms $(A, M^a) \to (B, N^a)$ in ${\rm Log}_{(R, P^a)/}^{{\rm chl\acute{e}t}, P}$. While the transitivity sequence \[B \otimes_{A} {\Bbb L}_{(A, M^a)/(R, P^a)}^{\rm rep} \to {\Bbb L}_{(B, N^a)/(R, P^a)}^{\rm rep} \to {\Bbb L}_{(B, N^a)/(A, M^a)}^{\rm rep}\] reveals that $(A, M^a) \to (B, N^a)$ is formally log \'etale, this does \emph{not} mean that it is charted log \'etale. 
\end{remark}

\begin{lemma}\label{lem:basechangestrict} Let $(\widetilde{R}, \widetilde{P})$ be a pre-log ring spectrum with logification $(\widetilde{R}, \widetilde{P}^a)$. Base-change along a strict morphism $(\widetilde{R}, \widetilde{P}^a) \to (R, P^a)$ induces a functor \[{\rm Log}_{(\widetilde{R}, \widetilde{P}^a)/}^{{\rm chl\acute{e}t}, \widetilde{P}} \to {\rm Log}_{({R}, {P}^a)/}^{{\rm chl\acute{e}t}, \widetilde{P}}.\]
\end{lemma}

\begin{proof} Consider a charted log \'etale morphism $(\widetilde{R}, \widetilde{P}^a) \to (\widetilde{A}, \widetilde{M}^a)$ with log \'etale chart $(\widetilde{R}, \widetilde{P}) \to (\widetilde{A}, \widetilde{M})$. We claim that the base-change $(R, P^a) \to (A, M^a)$ of  $(\widetilde{R}, \widetilde{P}^a) \to (\widetilde{A}, \widetilde{M}^a)$ along $(\widetilde{R}, \widetilde{P}^a) \to (R, P^a)$ is charted log \'etale with log \'etale chart $(R, \widetilde{P}) \to (A, \widetilde{M})$. The morphism $R \otimes_{{\Bbb S}^{\cal J}[\widetilde{P}]} {\Bbb S}^{\cal J}[\widetilde{M}] \to A$ is \'etale, as it is the base-change of the \'etale morphism $\widetilde{R} \otimes_{{\Bbb S}^{\cal J}[\widetilde{P}]} {\Bbb S}^{\cal J}[\widetilde{M}] \to \widetilde{A}$ along $\widetilde{R} \otimes_{{\Bbb S}^{\cal J}[\widetilde{P}]} {\Bbb S}^{\cal J}[\widetilde{M}] \to R \otimes_{{\Bbb S}^{\cal J}[\widetilde{P}]} {\Bbb S}^{\cal J}[\widetilde{M}]$. The morphism $A \otimes \gamma(\widetilde{P}) \to A \otimes \gamma(\widetilde{M})$ is an equivalence, being induced up from the equivalence $\widetilde{A} \otimes \gamma(\widetilde{P}) \to \widetilde{A} \otimes \gamma(\widetilde{M})$ along $\widetilde{A} \to A$. The strictness hypothesis ensures that the map $(R, \widetilde{P}) \to (A, \widetilde{M})$ indeed logifies to $(R, P^a) \to (A, M^a)$. 
\end{proof}

\subsection{Log \'etale rigidity} The precise formulation of Theorem \ref{thm:chartedlogetale} is:

\begin{theorem}\label{thm:precisechartedlogetale} Let $(R, P)$ be a pre-log ring spectrum. Base-change along the log Postnikov truncation $(R, P^a) \to (\pi_0(R), P_{\le 0}^a)$ induces an equivalence of categories \[{\rm Log}_{(R, P^a)/}^{{\rm chl\acute{e}t}, P} \xrightarrow{\simeq} {\rm Log}_{({\pi_0(R)}, {P_{\le 0}^a})/}^{{\rm chl\acute{e}t}, P}.\]
\end{theorem}

\begin{proof} Lemma \ref{lem:logetalechartformallylogetale} and Lemma \ref{lem:basechangestrict} imply that we may still apply Lemma \ref{lem:fullyfaithful} to infer that the functor is fully faithful, as the below argument shows. We shall now argue that it is essentially surjective. Consider a charted log \'etale morphism $(\pi_0(R), P_{\le 0}^a) \to (A_0, M_{0})$ with a given log \'etale chart $(\pi_0(R), P) \to (A_0, M)$. There is a commutative diagram \[\begin{tikzcd}[row sep = small]R \otimes_{{\Bbb S}^{\cal J}[P]} {\Bbb S}^{\cal J}[M] \ar{r}  \ar[dashed]{d} \ar[bend left = 5 mm]{rr} & \pi_0(R) \otimes_{{\Bbb S}^{\cal J}[P]} {\Bbb S}^{\cal J}[M] \ar{r} \ar{d} & \pi_0(R \otimes_{{\Bbb S}^{\cal J}[P]} {\Bbb S}^{\cal J}[M]) \ar{d} \\ A \ar[bend right = 5 mm, dashed]{rr}& A_0 \ar{r} & \pi_0(A_0)\end{tikzcd}\] of ${\Bbb E}_{\infty}$-rings. Observe that $\pi_0(\pi_0(R) \otimes_{{\Bbb S}^{\cal J}[P]} {\Bbb S}^{\cal J}[M]) \cong \pi_0(R \otimes_{{\Bbb S}^{\cal J}[P]} {\Bbb S}^{\cal J}[M])$. The ${\Bbb E}_{\infty}$-ring $A$ and the dashed morphisms are obtained from \'etale rigidity equivalence ${\rm CAlg}^{\rm \acute{e}t}_{R \otimes_{{\Bbb S}^{\cal J}[P]} {\Bbb S}^{\cal J}[M]/} \xrightarrow{\simeq} {\rm CAlg}^{\rm \acute{e}t}_{\pi_0(R \otimes_{{\Bbb S}^{\cal J}[P]} {\Bbb S}^{\cal J}[M])/}$ \cite[Theorem 7.5.0.6]{Lur17}; in particular, the outer rectangle is cocartesian. Appealing instead to the equivalence ${\rm CAlg}^{\rm \acute{e}t}_{\pi_0(R) \otimes_{{\Bbb S}^{\cal J}[P]} {\Bbb S}^{\cal J}[M]/} \xrightarrow{\simeq} {\rm CAlg}^{\rm \acute{e}t}_{\pi_0(R \otimes_{{\Bbb S}^{\cal J}[P]} {\Bbb S}^{\cal J}[M])/}$, we find that $A_0$ is essentially unique among \'etale $(\pi_0(R) \otimes_{{\Bbb S}^{\cal J}[P]} {\Bbb S}^{\cal J}[M])$-algebras for which the right-hand square is cocartesian. Since $R \otimes_{{\Bbb S}^{\cal J}[P]} {\Bbb S}^{\cal J}[M] \to A$ is \'etale, the pushout of the diagram $A \xleftarrow{} R \otimes_{{\Bbb S}^{\cal J}[P]} {\Bbb S}^{\cal J}[M] \xrightarrow{} \pi_0(R) \otimes_{{\Bbb S}^{\cal J}[P]} {\Bbb S}^{\cal J}[M]$ is another \'etale $(\pi_0(R) \otimes_{{\Bbb S}^{\cal J}[P]} {\Bbb S}^{\cal J}[M])$-algebra with this property, hence equivalent to $A_0$. In particular, there is a compatible morphism $A \to A_0$ making the resulting left-hand square cocartesian. 

Let us endow $A$ with the inverse-image pre-log structure $(A, M)$ along the composite $M \to \Omega^{\cal J}(R \otimes_{{\Bbb S}^{\cal J}[P]} {\Bbb S}^{\cal J}[M]) \to \Omega^{\cal J}(A)$. We claim that the induced map of logifications $(R, P^a) \to (A, M^a)$ is the formally log \'etale chart which we seek. To see this, consider the diagram \[\begin{tikzcd}[row sep = small](R, P) \ar{r} \ar{d} & (R \otimes_{{\Bbb S}^{\cal J}[P]} {\Bbb S}^{\cal J}[M], M) \ar{r} \ar{d} & (A, M) \ar{d} \\ (\pi_0(R), P) \ar{r} & (\pi_0(R) \otimes_{{\Bbb S}^{\cal J}[P]} {\Bbb S}^{\cal J}[M], M) \ar{r} & (A_0, M)\end{tikzcd}\] of pre-log ring spectra. The upper horizontal composite exhibits $(R, P^a) \to (A, M^a)$ as a charted log \'etale: We already know that the morphism $R \otimes_{{\Bbb S}^{\cal J}[P]} {\Bbb S}^{\cal J}[M] \to A$ is \'etale, and the $A$-module $A \otimes (M^{\rm gp}/P^{\rm gp})$ vanishes since $A_0 \otimes (M^{\rm gp}/P^{\rm gp})$ does and $A \to A_0$ is $0$-connected \cite[Lemma 3.3]{PVK22}. As both squares are cocartesian, so is the outer rectangle, and hence we obtain a cocartesian square of log ring spectra after logification. This concludes the proof. 
\end{proof}


\begin{bibdiv}
\begin{biblist}

\bib{Ang08}{article}{
   author={Angeltveit, Vigleik},
   title={Topological Hochschild homology and cohomology of $A_\infty$ ring
   spectra},
   journal={Geom. Topol.},
   volume={12},
   date={2008},
   number={2},
   pages={987--1032},
   issn={1465-3060},
   review={\MR{2403804}},
   doi={10.2140/gt.2008.12.987},
}


\bib{ABM23}{misc}{
      author={Ausoni, C.},
	author={Bayındır, H.}
	author={Moulinos, T.}
       title={Adjunction of roots, algebraic K-theory and chromatic redshift},
        date={2023},
        note={\arxivlink{2211.16929}},
}


\bib{Bas99}{article}{
   author={Basterra, M.},
   title={Andr\'{e}-Quillen cohomology of commutative $S$-algebras},
   journal={J. Pure Appl. Algebra},
   volume={144},
   date={1999},
   number={2},
   pages={111--143},
   issn={0022-4049},
   review={\MR{1732625}},
   doi={10.1016/S0022-4049(98)00051-6},
}

\bib{BF78}{article}{
   author={Bousfield, A. K.},
   author={Friedlander, E. M.},
   title={Homotopy theory of $\Gamma $-spaces, spectra, and bisimplicial
   sets},
   conference={
      title={Geometric applications of homotopy theory},
      address={Proc. Conf., Evanston, Ill.},
      date={1977},
   },
   book={
      series={Lecture Notes in Math.},
      volume={658},
      publisher={Springer, Berlin-New York},
   },
   isbn={3-540-08859-8},
   date={1978},
   pages={80--130},
   review={\MR{0513569}},
}


\bib{BLPO23Prism}{article}{
   author={Binda, F.},
   author={Lundemo, T.},
   author={Park, D.},
	author = {{\O}stv{\ae}r, P. A.}
   title={Logarithmic prismatic cohomology via logarithmic THH},
   journal={IMRN.},
   date={2023},
   doi={10.1093/imrn/rnad224},
}

\bib{BLPO23}{article}{
      author={Binda, F.},
      author={Lundemo, T.},
      author={Park, D.},
	author = {{\O}stv{\ae}r, P. A.}
       title={A Logarithmic Hochschild-Kostant-Rosenberg Theorem and Generalized Residue Sequences in Logarithmic Hochschild Homology},
	journal = {Adv. Math.}
	issn = {0001-8708}
        date={2023},
        doi = {10.1016/j.aim.2023.109354}
}

\bib{BM05}{article}{
   author={Basterra, Maria},
   author={Mandell, Michael A.},
   title={Homology and cohomology of $E_\infty$ ring spectra},
   journal={Math. Z.},
   volume={249},
   date={2005},
   number={4},
   pages={903--944},
   issn={0025-5874},
   review={\MR{2126222}},
   doi={10.1007/s00209-004-0744-y},
}

\bib{BMS19}{article}{
   author={Bhatt, Bhargav},
   author={Morrow, Matthew},
   author={Scholze, Peter},
   title={Topological Hochschild homology and integral $p$-adic Hodge
   theory},
   journal={Publ. Math. Inst. Hautes \'{E}tudes Sci.},
   volume={129},
   date={2019},
   pages={199--310},
   issn={0073-8301},
   review={\MR{3949030}},
   doi={10.1007/s10240-019-00106-9},
}

\bib{Dev20}{article}{
   author={Devalapurkar, Sanath},
   title={Roots of unity in $K(n)$-local rings},
   journal={Proc. Amer. Math. Soc.},
   volume={148},
   date={2020},
   number={7},
   pages={3187--3194},
   issn={0002-9939},
   review={\MR{4099803}},
   doi={10.1090/proc/14960},
}

\bib{egaiv}{article}{
   author={Grothendieck, A.},
   title={\'{E}l\'{e}ments de g\'{e}om\'{e}trie alg\'{e}brique. IV.
   \'{E}tude locale des sch\'{e}mas et des morphismes de sch\'{e}mas IV},
   language={French},
   journal={Inst. Hautes \'{E}tudes Sci. Publ. Math.},
   number={32},
   date={1967},
   pages={361},
   issn={0073-8301},
   review={\MR{0238860}},
}


\bib{GHN17}{article}{
   author={Gepner, David},
   author={Haugseng, Rune},
   author={Nikolaus, Thomas},
   title={Lax colimits and free fibrations in $\infty$-categories},
   journal={Doc. Math.},
   volume={22},
   date={2017},
   pages={1225--1266},
   issn={1431-0635},
   review={\MR{3690268}},
}

\bib{Hau22}{article}{
   author={Haugseng, Rune},
   title={On (co)ends in $\infty$-categories},
   journal={J. Pure Appl. Algebra},
   volume={226},
   date={2022},
   number={2},
   pages={Paper No. 106819, 16},
   issn={0022-4049},
   review={\MR{4279286}},
   doi={10.1016/j.jpaa.2021.106819},
}

\bib{HP15}{article}{
   author={Harpaz, Yonatan},
   author={Prasma, Matan},
   title={The Grothendieck construction for model categories},
   journal={Adv. Math.},
   volume={281},
   date={2015},
   pages={1306--1363},
   issn={0001-8708},
   review={\MR{3366868}},
   doi={10.1016/j.aim.2015.03.031},
}

\bib{HP23}{article}{
   author={Hesselholt, Lars},
   author={Pstr\k{a}gowski, P},
   title={Dirac Geometry I: Commutative Algebra},
   journal={Peking Math J.},
   date={2023},
   doi={10.1007/s42543-023-00072-6},
}

\bib{HW22}{article}{
   author={Hahn, Jeremy},
   author={Wilson, Dylan},
   title={Redshift and multiplication for truncated Brown-Peterson spectra},
   journal={Ann. of Math. (2)},
   volume={196},
   date={2022},
   number={3},
   pages={1277--1351},
   issn={0003-486X},
   review={\MR{4503327}},
   doi={10.4007/annals.2022.196.3.6},
}

\bib{Kat89}{article}{
   author={Kato, Kazuya},
   title={Logarithmic structures of Fontaine-Illusie},
   conference={
      title={Algebraic analysis, geometry, and number theory},
      address={Baltimore, MD},
      date={1988},
   },
   book={
      publisher={Johns Hopkins Univ. Press, Baltimore, MD},
   },
   isbn={0-8018-3841-X},
   date={1989},
   pages={191--224},
   review={\MR{1463703}},
}


\bib{Kat}{misc}{
      author={Kato, K.},
       title={Logarithmic degeneration and Dieudonn\'e theory},
        note={Preprint.},
}

\bib{KS04}{article}{
   author={Kato, Kazuya},
   author={Saito, Takeshi},
   title={On the conductor formula of Bloch},
   journal={Publ. Math. Inst. Hautes \'{E}tudes Sci.},
   number={100},
   date={2004},
   pages={5--151},
   issn={0073-8301},
   review={\MR{2102698}},
   doi={10.1007/s10240-004-0026-6},
}

\bib{KS58}{article}{
   author={Kodaira, K.},
   author={Spencer, D. C.},
   title={On deformations of complex analytic structures. I, II},
   journal={Ann. of Math. (2)},
   volume={67},
   date={1958},
   pages={328--466},
   issn={0003-486X},
   review={\MR{0112154}},
   doi={10.2307/1970009},
}



\bib{Law20}{misc}{
      author={Lawson, T.},
       title={Adjoining roots in homotopy theory},
        date={2020},
        note={\arxivlink{2002.01997}},
}

\bib{Lun21}{article}{
   author={Lundemo, Tommy},
   title={On the relationship between logarithmic TAQ and logarithmic THH},
   journal={Doc. Math.},
   volume={26},
   date={2021},
   pages={1187--1236},
   issn={1431-0635},
   review={\MR{4324464}},
}

\bib{Lun22}{misc}{
      author={Lundemo, T.},
       title={On Formally \'Etale Morphisms in Derived and Higher Logarithmic Geometry},
        date={2022},
        note={PhD Thesis (Radboud University Nijmegen). Available at \url{https://hdl.handle.net/2066/252233}},
}


\bib{dagiv}{misc}{
      author={Lurie, J.},
       title={Derived Algebraic Geometry IV: Deformation Theory},
        date={2009},
        note={Preprint, available at the author's home page},
}

\bib{HTT}{book}{
   author={Lurie, Jacob},
   title={Higher topos theory},
   series={Annals of Mathematics Studies},
   volume={170},
   publisher={Princeton University Press, Princeton, NJ},
   date={2009},
   pages={xviii+925},
   isbn={978-0-691-14049-0},
   isbn={0-691-14049-9},
   review={\MR{2522659}},
   doi={10.1515/9781400830558},
}

\bib{Lur17}{misc}{
      author={Lurie, J.},
       title={Higher Algebra},
        date={2017},
        note={Preprint, available at the author's home page},
}
	
\bib{SAG}{misc}{
      author={Lurie, J.},
       title={Spectral Algebraic Geometry},
        date={2018},
        note={Preprint, available at the author's home page},
}

\bib{Mat}{article}{
   author={Mather, Michael},
   title={Pull-backs in homotopy theory},
   journal={Canadian J. Math.},
   volume={28},
   date={1976},
   number={2},
   pages={225--263},
   issn={0008-414X},
   review={\MR{0402694}},
   doi={10.4153/CJM-1976-029-0},
}

\bib{MMSS01}{article}{
   author={Mandell, M. A.},
   author={May, J. P.},
   author={Schwede, S.},
   author={Shipley, B.},
   title={Model categories of diagram spectra},
   journal={Proc. London Math. Soc. (3)},
   volume={82},
   date={2001},
   number={2},
   pages={441--512},
   issn={0024-6115},
   review={\MR{1806878}},
   doi={10.1112/S0024611501012692},
}

\bib{Niz08}{article}{
   author={Nizio\l , Wies\l awa},
   title={$K$-theory of log-schemes. I},
   journal={Doc. Math.},
   volume={13},
   date={2008},
   pages={505--551},
   issn={1431-0635},
   review={\MR{2452875}},
}

\bib{Ogu18}{book}{
   author={Ogus, Arthur},
   title={Lectures on logarithmic algebraic geometry},
   series={Cambridge Studies in Advanced Mathematics},
   volume={178},
   publisher={Cambridge University Press, Cambridge},
   date={2018},
   pages={xviii+539},
   isbn={978-1-107-18773-3},
   review={\MR{3838359}},
   doi={10.1017/9781316941614},
}

\bib{Ols05}{article}{
   author={Olsson, Martin C.},
   title={The logarithmic cotangent complex},
   journal={Math. Ann.},
   volume={333},
   date={2005},
   number={4},
   pages={859--931},
   issn={0025-5831},
   review={\MR{2195148}},
   doi={10.1007/s00208-005-0707-6},
}


\bib{PVK22}{article}{
   author={Pstr\polhk{a}gowski, Piotr},
   author={VanKoughnett, Paul},
   title={Abstract Goerss-Hopkins theory},
   journal={Adv. Math.},
   volume={395},
   date={2022},
   pages={Paper No. 108098, 51},
   issn={0001-8708},
   review={\MR{4363589}},
   doi={10.1016/j.aim.2021.108098},
}

\bib{Qui}{article}{
   author={Quillen, Daniel},
   title={On the (co-) homology of commutative rings},
   conference={
      title={Applications of Categorical Algebra},
      address={Proc. Sympos. Pure Math., Vol. XVII, New York},
      date={1968},
   },
   book={
      series={Proc. Sympos. Pure Math.},
      volume={XVII},
      publisher={Amer. Math. Soc., Providence, RI},
   },
   date={1970},
   pages={65--87},
   review={\MR{0257068}},
}

\bib{Rog09}{article}{
   author={Rognes, John},
   title={Topological logarithmic structures},
   conference={
      title={New topological contexts for Galois theory and algebraic
      geometry (BIRS 2008)},
   },
   book={
      series={Geom. Topol. Monogr.},
      volume={16},
      publisher={Geom. Topol. Publ., Coventry},
   },
   date={2009},
   pages={401--544},
   review={\MR{2544395}},
   doi={10.2140/gtm.2009.16.401},
}

\bib{RogLoen}{misc}{
      author={Rognes, J.},
       title={Log THH and TC},
        date={2009},
        note={Talk notes, available at the author's home page},
}
	

\bib{RSS15}{article}{
   author={Rognes, John},
   author={Sagave, Steffen},
   author={Schlichtkrull, Christian},
   title={Localization sequences for logarithmic topological Hochschild
   homology},
   journal={Math. Ann.},
   volume={363},
   date={2015},
   number={3-4},
   pages={1349--1398},
   issn={0025-5831},
   doi={10.1007/s00208-015-1202-3},
}



\bib{RSS18}{article}{
   author={Rognes, John},
   author={Sagave, Steffen},
   author={Schlichtkrull, Christian},
   title={Logarithmic topological Hochschild homology of topological
   $K$-theory spectra},
   journal={J. Eur. Math. Soc. (JEMS)},
   volume={20},
   date={2018},
   number={2},
   pages={489--527},
   issn={1435-9855},
   review={\MR{3760301}},
   doi={10.4171/JEMS/772},
}

\bib{Sag14}{article}{
   author={Sagave, Steffen},
   title={Logarithmic structures on topological $K$-theory spectra},
   journal={Geom. Topol.},
   volume={18},
   date={2014},
   number={1},
   pages={447--490},
   issn={1465-3060},
   review={\MR{3159166}},
   doi={10.2140/gt.2014.18.447},
}

\bib{Sag16}{article}{
   author={Sagave, Steffen},
   title={Spectra of units for periodic ring spectra and group completion of
   graded $E_\infty$ spaces},
   journal={Algebr. Geom. Topol.},
   volume={16},
   date={2016},
   number={2},
   pages={1203--1251},
   issn={1472-2747},
   review={\MR{3493419}},
   doi={10.2140/agt.2016.16.1203},
}

\bib{SS12}{article}{
   author={Sagave, Steffen},
   author={Schlichtkrull, Christian},
   title={Diagram spaces and symmetric spectra},
   journal={Adv. Math.},
   volume={231},
   date={2012},
   number={3-4},
   pages={2116--2193},
   issn={0001-8708},
   review={\MR{2964635}},
   doi={10.1016/j.aim.2012.07.013},
}

\bib{SS19}{article}{
   author={Sagave, Steffen},
   author={Schlichtkrull, Christian},
   title={Virtual vector bundles and graded Thom spectra},
   journal={Math. Z.},
   volume={292},
   date={2019},
   number={3-4},
   pages={975--1016},
   issn={0025-5874},
   review={\MR{3980280}},
   doi={10.1007/s00209-018-2131-0},
}

\bib{SSV16}{article}{
   author={Sagave, Steffen},
   author={Sch\"{u}rg, Timo},
   author={Vezzosi, Gabriele},
   title={Derived logarithmic geometry I},
   journal={J. Inst. Math. Jussieu},
   volume={15},
   date={2016},
   number={2},
   pages={367--405},
   issn={1474-7480},
   review={\MR{3480969}},
   doi={10.1017/S1474748014000322},
}

\bib{SVW99}{article}{
   author={Schw\"{a}nzl, R.},
   author={Vogt, R. M.},
   author={Waldhausen, F.},
   title={Adjoining roots of unity to $E_\infty$ ring spectra in good
   cases---a remark},
   conference={
      title={Homotopy invariant algebraic structures},
      address={Baltimore, MD},
      date={1998},
   },
   book={
      series={Contemp. Math.},
      volume={239},
      publisher={Amer. Math. Soc., Providence, RI},
   },
   isbn={0-8218-1057-X},
   date={1999},
   pages={245--249},
   review={\MR{1718085}},
   doi={10.1090/conm/239/03606},
}

\bib{TV08}{article}{
   author={To\"{e}n, Bertrand},
   author={Vezzosi, Gabriele},
   title={Homotopical algebraic geometry. II. Geometric stacks and
   applications},
   journal={Mem. Amer. Math. Soc.},
   volume={193},
   date={2008},
   number={902},
   pages={x+224},
   issn={0065-9266},
   review={\MR{2394633}},
   doi={10.1090/memo/0902},
}






	
\end{biblist}
\end{bibdiv}
\end{document}